\documentclass[10pt]{article}
\usepackage{graphicx}
\usepackage{subcaption}
\usepackage{amsmath}
\usepackage{amsthm}
\usepackage{amssymb}
\usepackage{amsfonts}
\usepackage{tikz}

\newtheorem{theorem}{Theorem}[section]

\newtheorem{lemma}[theorem]{Lemma}

\newtheorem{remark}{Remark}

\begin{document}
\title{Hyperpriors for Mat\'ern fields  with applications in Bayesian inversion}
%\author{Lassi Roininen, Mark Girolami, Sari Lasanen and Markku Markkanen}
\maketitle

% Enter the first author's name and address:
\centerline{\scshape Lassi Roininen}
\medskip
{\footnotesize
% please put the address of the first author
 \centerline{Department of Mathematics, Imperial College London}
%   \centerline{Other lines}
   \centerline{South Kensington Campus, London SW7 2AZ, United Kingdom}
} % Do not forget to end the {\footnotesize by the sign }

\medskip

\centerline{\scshape Mark Girolami}
\medskip
{\footnotesize
% please put the address of the first author
 \centerline{Department of Mathematics, Imperial College London, and,}
   \centerline{Alan Turing Institute, London, United Kingdom}
%   \centerline{South Kensington Campus, London SW7 2AZ, United Kingdom}
} % Do not forget to end the {\footnotesize by the sign }

\medskip

%\centerline{\scshape Andrew Stuart}
%\medskip
%{\footnotesize
% % please put the address of the second  and third author
% \centerline{ California Institute of Technology}
%%   \centerline{Other lines}
%   \centerline{Pasadena, USA}
%}

\medskip

\centerline{\scshape Sari Lasanen}
\medskip
{\footnotesize
 % please put the address of the second  and third author
 \centerline{ University of Oulu}
%   \centerline{Other lines}
   \centerline{Oulu, Finland}
}

\medskip

\centerline{\scshape Markku Markkanen}
\medskip
{\footnotesize
 % please put the address of the second  and third author
 \centerline{ Eigenor Corporation}
%   \centerline{Other lines}
   \centerline{Lompolontie 1, FI-99600 Sodankyl\"a, Finland}
}

\bigskip
%
%% The name of the associate editor will be entered by an editorial staff
%% "Communicated by the associate editor name" is not needed for special issue.
% \centerline{(Communicated by the associate editor name)}

\sloppy

%The abstract of your paper
\begin{abstract}
We introduce non-stationary Mat\'ern field priors with stochastic partial differential equations, and construct correlation  length-scaling  with hyperpriors.
We model both the hyperprior and the Mat\'ern prior as continuous-parameter random fields.
As hypermodels, we use Cauchy and Gaussian random fields, which we map suitably to a desired correlation length-scaling range. 
For computations, we discretise the models with finite difference methods.
We consider the convergence of the discretised prior and posterior to the discretisation limit.
We apply the developed methodology to certain interpolation and numerical differentiation problems, and show numerically that we can make Bayesian inversion which promotes competing constraints of smoothness and edge-preservation.
For computing the conditional mean estimator of the posterior distribution, we use a combination of Gibbs and Metropolis-within-Gibbs sampling algorithms.
\end{abstract}

\section{Introduction}

In many Bayesian statistical estimation algorithms, the objective is to explore the posterior distribution of  a continuous-parameter unknown $v(x)$, $x\in\mathbb{R}^d$, $d=1,2,\dots$  given a direct or indirect noisy realisation $y\in\mathbb{R}^M$ of the unknown  $v$ at fixed locations.
Bayesian estimation algorithms are  widely used for example  in spatial statistics, Machine Learning and Bayesian statistical inverse problems and specific  applications include e.g.\ remote sensing, medical imaging and ground prospecting \cite{Bardsley2013,Calvetti2007,Kaipio2005,Lindgren2011,Roininen2014,Stuart2010}.

A priori information is a key factor in any Bayesian statistical estimation algorithm, as it is used to stabilise the posterior distribution.
Gaussian processes and fields are common choices as priors, because of their analytical and computational properties, and because of their easy construction through the mean and covariance functions.

Gaussian priors are known to promote smooth estimates \cite{Kaipio2005}. However, the smoothness of the unknown is inherently problem-specific. For example in many atmospheric remote sensing algorithms,  the unknown is often assumed to be continuous \cite{Norberg2015}. 
On the other hand, in many subsurface imaging or in medical tomography applications, the unknown may have anisotropies, inhomogeneities, and even discontinuities \cite{Dunlop2016}.

A common method for  expanding  prior models outside the scope of Gaussian distributions is to apply hyperparametric models \cite{Calvetti2007b,Calvetti2008}.
In this paper, we study  hypermodels for inhomogeneous Mat\'ern fields, and
our specific objective is to demonstrate that the proposed priors are flexible enough to promote both smooth and edge-preserving estimates.

Mat\'ern fields, a class of Gaussian Markov random fields \cite{Lindgren2011,Roininen2014,Rue2005}, are often defined as stationary Gaussian random fields, i.e.\  with them we can model isotropic unknown. 
By a simple change of variables, these fields can model also anisotropic features.
Via modelling the Mat\'ern fields with stochastic partial differential equations and by locally defining the correlation parameters, we can even construct inhomogeneous random fields \cite{Lindgren2011,Roininen2013a}.
In addition, we can make them  computationally very efficient, as  the finite-dimensional approximation of the inverse covariance matrix, the precision matrix, is  sparse by construction.

Let us denote by $v^N$  the finite-dimensional approximation of the  continuous-parameter random field  $v$.
The solution of a Bayesian estimation problem is a so-called posterior distribution.
We give it as an unnormalised probability density
\begin{equation} \label{eqn:posterior}
	D\left(v^N|y\right) = \frac{D\left(v^N\right)D\left(y|v^N\right)}{D(y)} \propto D\left(v^N\right)D\left(y|v^N\right),
\end{equation}
where the likelihood density $D\left(y|v^N\right)$ is obtained e.g.\ through some physical observation system and the a priori density $D\left(v^N\right)$ reflects our information of the unknown object before any actual measurement is done. We take $v^N$ to be an approximation of a Mat\'ern field.
$D(y)$ is a normalisation constant, which we often can omit from the analysis.

We model the Mat\'ern field length-scaling  as a continuous-parameter random field $\ell(x)$, i.e.\ 
 we construct a prior for certain parameters of the  prior $D(v^N)$.  %
By denoting the discrete approximation of the continuous-parameter field $\ell $ by $\ell^N$, we may include the hyperprior into the posterior distribution \eqref{eqn:posterior}, and hence write  the posterior probability density as 
\begin{equation*}
D\left(v^N,\ell^N \vert y\right) \propto D\left(v^N,\ell^N\right)D\left(y\vert v^N\right) = D\left(\ell^N\right)D\left(v^N\vert \ell^N\right) D\left(y\vert v^N\right),
\end{equation*}
where $D\left(v^N,\ell^N\right)$ is the prior constructed from the hyperprior $D\left(\ell^N\right)$  and the prior itself is $D\left(v^N\vert \ell^N\right)$. 
The idea is similar to the length-scale modelling in \cite{Paciorek2003,Paciorek2006}, but  we will use using stochastic partial differential equations and  sparse matrices,  hence gaining  computational advantages.

To simplify the sampling of  $\ell$, we introduce an auxiliary random field $u$, and  apply a further parametrisation  $\ell=\ell(x;  u)$.
We choose to model $u$ as a continuous-parameter Cauchy or Gaussian random field, but it could be also something else, for example an $\alpha$-stable random field \cite{Markkanen2016,Sullivan2016}. 
In this paper, we will  put  an emphasis on the  discretisation of these hypermodels and the convergence of the discrete models to continuous models.
From a computational point of view, we need to discuss  Markov chain Monte Carlo (MCMC) methods, and in particular we will consider Gibbs  and Metropolis-within-Gibbs sampling.

Modelling non-stationary Gaussian random fields and using them in e.g.\ interpolation is not a new idea, see for example Fuglstad  et  al.\ 2015 \cite{Fuglstad2015} for a comprehensive reference list.
Other constructions of non-stationary Gaussian processes, given lately notable attention, include e.g.\ so-called deep learning algorithms, where an $n$-layer Gaussian process is formed in such a way that the deepest level is a stationary Gaussian process.
These methods produce interesting priors from an engineering perspective. These methods, however,  are typically lacking rigorous mathematical analysis  \cite{Paciorek2003,Paciorek2006}.
Also, many deep learning methods typically rely on full covariance matrices, hence computational burden might become a bottleneck, especially in high-dimensional problems.

We note that discontinuities are best  modelled with  specific non-Gaussian prior constructions. 
A common choice is a  total variation prior, which promotes edge-preserving estimates. 
However, total variation priors are well-known to behave as Gaussian smoothness priors when the discretisation is made denser and denser (Lassas and Siltanen 2004)  \cite{Lassas2004}.
%,
Constructions of proposed non-Gaussian priors, which do not converge to Gaussian fields in the discretisation limit, include the Besov space priors (Lassas et al.\ 2009) \cite{Lassas2009}, which are constructed on a wavelet basis, hierarchical Mumford-Shah priors  (Helin and Lassas, 2011) \cite {helin}, and recently applied Cauchy priors (Markkanen et al. 2016) \cite{Markkanen2016}.
Construction of a Mat\'ern style random field with non-Gaussian noise has been studied by Bolin 2014 \cite{Bolin2014}.
These algorithms may work suitably well for example for edge-preserving Bayesian inversion, but they may lack the smoothness or limiting properties, which we aim to deploy in the proposed hypermodel construction.

The rest of this paper is organised as follows: 
In Section \ref{sec:LinearBayes}, we review the basics of linear Bayesian statistical estimation algorithms.
In Section \ref{sec:Priors}, we discuss continuous Mat\'ern fields and  construct certain hypermodels.
In Section \ref{sec:discretisation}, we consider discretisation of the hypermodels, and, convergence of the discretised models to the continuous models.
We consider rough hypermodels in Section \ref{section:roughhypermodels}, and discuss Gibbs and Metropolis-within-Gibbs algorithms in Section \ref{section:MCMC}.
In Section \ref{sec:numerics}, we show by a number of numerical examples how to use the constructed model in interpolation algorithms.

\section{Linear Bayesian estimation problems}
\label{sec:LinearBayes}

Let us consider a continuous-parameter linear statistical estimation problem 
\begin{equation}\label{eqn:continuousobservation}
y = \mathcal{A}v+e.
\end{equation}
We assume that we know exactly one realisation of $y$ and the linear mapping $\mathcal A$ from some function space (e.g.\ separable Banach space) to a finite-dimensional space $\mathbb{R}^M$.
We further assume that we know the statistical properties of the noise $e$.
From now on, we assume that $e$ is zero-mean Gaussian white noise statistically independent of $v$.
We emphasise that we do not know the realisation of the noise $e$.
Given these assumptions, our objective is to estimate the posterior distribution of $v$.

For computational Bayesian statistical estimation problems, we discretise Equation \eqref{eqn:continuousobservation}, and write it as a matrix equation
\begin{equation} \label{eqn:discreteobservation}
y = Av^N+e.
\end{equation}
Likelihood probability, a factor of the posterior distribution (see Equation \eqref{eqn:posterior}), can then be given as
\begin{equation*}
D(y\vert v^N) \propto \exp\left(-\frac{1}{2}(y-Av^N)^T\Sigma^{-1}(y-Av^N)\right) ,
\end{equation*}
where the Gaussian measurement noise $e\sim \mathcal{N}(0,\Sigma)$, and $\Sigma$ is noise covariance matrix.

We start the preparations for the numerical sampling of the posterior for the hyperprior.
First, we will consider the case when  the unknown $v^N$,  conditioned with the hyperparameter
$\ell^N$, has Gaussian distribution 
 $\mathcal{N}(0,C ) $.

We can write the prior as a normalised probability density
\begin{equation*}
D\left(v^N\vert \ell_N\right) = \frac{1}{\sqrt{(2\pi)^N\vert C \vert}} \exp\left(-\frac{1}{2}(v^N)^TC^{-1} v^N\right).
\end{equation*}
Here we have included the normalisation constant, as we aim to include the hyperparameters of $v^N$  into the  matrix $C$, i.e.\  $ |C|=|C(\ell^N)|  $ is not a constant, for the different values of the hyperparameter, unlike the normalisation constant $\vert \Sigma\vert$ in the likelihood density.
Our aim is to decompose the prior inverse covariance, i.e.\ precision matrix, as $C(\ell^N)^{-1}=(L(\ell^N))^T L(\ell^N)$.
This means that, similarly to Equation \eqref{eqn:continuousobservation}, we can present the prior also as an observation equation $L(\ell^N)v^N=-w^N$, where $w^N\sim \mathcal{N}(0,I)$.
Hence, we can form a stacked matrix equation
\begin{equation*}
\begin{pmatrix} A \\ L(\ell^N)  \end{pmatrix} v^N + \begin{pmatrix} e \\ w^N \end{pmatrix} = \begin{pmatrix} y \\ 0 \end{pmatrix}.
\end{equation*}
The benefit of this formulation is that we can easily make Gibbs sampling of the posterior distribution with this formulation, and 
% i.e.\ we compute conditional mean (CM) estimator
%
hence  to compute the posterior mean
\begin{equation*}
v^N_{\mathrm{CM}} = \frac{\int v^ND(v^N |\ell^N ) D(\ell^N)D(y|v^N)dv^N d\ell^N }{\int D(v^N|\ell^N ) D(\ell^N)D(y|v^N)dv^N d\ell^N}.
\end{equation*}

   We can write similarly the estimate for the variable length-scale  $\ell$
\begin{equation*}
\begin{split}
\ell_{\mathrm{CM}} ^N &=   \frac{\int \ell^N D(\ell^N )D(v^N\vert\ell^N) D(y\vert v^N) d\ell^N  dv^N }{\int D(\ell ^N)D(v^N\vert \ell ^N ) D(y\vert v^N) d\ell^N dv^N }\\
&=  \frac{\int \ell^N(u)D(u)D(v^N\vert\ell^N ( u)) D(y\vert v^N) du dv^N }{\int D(\ell^N(u))D(v^N\vert \ell^N(u)) D(y\vert v^N) du dv^N }. 
\end{split}
\end{equation*}
For the estimation of $\ell^N$, we  use the so-called Metropolis-within-Gibbs algorithm \cite{Markkanen2016}, where we make Gibbs type sampling, but component-wise, we make Metropolis-Hastings algorithm.
Alternatives for Metropolis-within-Gibbs include e.g.\ pseudo-marginal approach to MCMC studied by Filippone and Girolami 2014 \cite{Filippone2014}. 
The estimation of $v^N$ can be used with standard Gibbs sampling techniques.

\section{Mat\'ern field priors and hyperpriors}
\label{sec:Priors}

Mat\'ern fields are often defined as stationary Gaussian random field with a covariance function 
\begin{equation} \label{eqn:matern}
\mathrm{Cov}(x,x') = \mathrm{Cov}(x-x') = \frac{2^{1-\nu}}{\Gamma(\nu)}\left(\frac{\vert x-x' \vert}{\ell}\right)^\nu K_\nu\left(\frac{\vert x-x' \vert}{\ell}\right), \quad x,x'\in \mathbb{R}^d,
\end{equation}
where $\nu>0$ is the smoothness parameter, and $K_\nu$ is modified Bessel function of the second kind or order $\nu$.
The parameter $\ell$ is called length-scaling. 
Correlation length, that is to say where correlation is 0.1, corresponds approximately $\delta = \ell \sqrt{8\nu}$. 
From now on, we suppose that the Mat\'ern fields have a  zero mean.

Let us  recall the stochastic partial differential equation for the Mat\'ern fields \cite{Lindgren2011,Roininen2014}.
The  Fourier transform of the covariance function in Equation \eqref{eqn:matern}, gives a power spectrum
\begin{equation*}
S(\xi) = \frac{2^d\pi^{d/2}\Gamma(\nu+d/2)}{\Gamma(\nu)\ell^{2\nu}}\left(\frac{1}{\ell^2}+\vert\xi\vert^2\right)^{-(\nu+d/2)}.
\end{equation*}
As first mentioned by Rozanov 1977 \cite{Rozanov1977}, only fields with spectral density given by the reciprocal of a polynomial have a Markov representation.
 For our applications, we fix $\nu= 2-d/2$.

If we let $w$ be white noise, and by using 'hat'-notation for a Fourier-transformed object, then we may define the basic 
Mat\'ern field  $v$ through  the equation $\widehat{v} = \sigma \sqrt{S(\xi)}\widehat{w}$ in the sense of distributions. 
By using inverse Fourier transforms, we may write a stochastic partial differential equation
\begin{equation*} 
\left(1-\ell^2 \Delta\right) v = \sigma\sqrt{\ell^d}w.
\end{equation*}
Here we have an elliptic operator equation.
We note that the constructed field  $v$ is isotropic, i.e.\ the field has constant correlation length-scaling $\ell$ to every coordinate direction.

We modify the isotropic formulation to be inhomogeneous by allowing a spatially variable  length-scaling field $\ell(x)$, for which we write a stochastic partial differential equation%  this reduces to
\begin{equation} \label{eqn:isotropicmodulated}
\left(1 - \ell(x)^2 \Delta \right) v = \sigma \sqrt{\ell(x)^d}  w.
\end{equation}
In order to have a well-defined elliptic equation, we require that $\inf_{x\in D} \ell(x) >0$. 
The condition will be fulfilled with the help of an auxiliary transformation. Moreover, $\ell$ needs to be regular enough. 
We consider the cases where $\ell$ has $L^\infty(D)$-sample paths. In addition, we consider the  problems that arise for  rougher sample paths.  
We will construct the $\ell(x)$-model in the following sections.

Let us consider now the discretisation of Equation \eqref{eqn:isotropicmodulated}. It is well-known that white noise can be seen as a distributional derivative of the Brownian sheet $B$. 
Moreover, the measurable linear functionals of white noise can be identified with stochastic integrals. 
Hence, it is natural to discretise white noise as 
\begin{equation} \label{eqn:whitenoisediscrete}
w^N  (x) = \sum_{k=1}^{K_n}   \left( \frac{1}{|A_k|}\int 1_{A_k} (x) dB_x\right) 1_{A_k}(x), 
\end{equation}
where $\cup_{k=1}^{K_n}  A_k  = D$.
The   variance of white noise  $w^N$  at  $x\in A_k$,   is  
\begin{equation*}
w^N \big\vert_{x\in A_k} \sim \mathcal{N}\left(0,|A_k|^{-1 } \right) = \mathcal{N}\left(0,h^{-d} \right),
\end{equation*}
where $h$ is the discretisation step, which we choose to be same along all the coordinate directions.

The only thing we have left to discretise in Equation \eqref{eqn:isotropicmodulated}, is the operator part, for which we can use any standard finite difference methods.
Hence, we can write e.g.\ the one-dimensional discretisation of \eqref{eqn:isotropicmodulated} as 
\begin{equation} \label{eqn:1Dappr}
\begin{split}
\left(1-\ell(x)^2 \Delta \right)v|_{x=jh} & \approx  v^N_j-\ell_j^2\frac{v^N_{j-1}-2v^N_j+v^N_{j+1}}{h^2}  \\ &=  \sigma\sqrt{\ell_j}w^N_j \sim   \mathcal{N}\left(0,\sigma^2\ell_jh^{-1}\right),
\end{split}
\end{equation}
where $jh\in h\mathbb{Z}$ is the discretisation lattice, and $\ell_j:=\ell(jh)$.
This model can be given as a matrix equation $L(\ell^N) v^N=w^N$, where $L(\ell^N)$ is a symmetric sparse matrix.

\subsection{Hypermodels}
\label{subsec:hypermodels}

Let us consider modelling the length-scaling $\ell(x;u)$ with Gaussian random fields.  To achieve an elliptic equation, we  apply  a log-normal model 
\begin{equation} \label{eqn:exp}
\ell(x;u)=\exp(u(x)),
\end{equation}
 where $u$ is a Gaussian random field. For the discrete model on lattice points $x=jh$, we set similarly $\ell^N_j =\exp\left(u^N_j\right)$.

For example, we may choose another Mat\'ern field as a hypermodel for $u$. 
Hence, let $u$ be a zero-mean Mat\'ern field, with constant length-scaling $\ell_0$, and consider its discretisation  $u^N$  by   \eqref{eqn:1Dappr}.  
Then $u^N$ has  covariance matrix $\widetilde C^N$ and we write the discrete hypermodel as  
\begin{equation*}
\begin{split}
 D\left(u^N\right)D\left(v^N \vert \ell^N(u^N)\right) \propto &   \exp\left(-\frac{1}{2}u^T (\widetilde C^N)^{-1} u^N\right)\times \dots \\ & \vert L (\ell^N) \vert\exp \left( -\frac{1}{2}(v^N)^TL(\ell^N) ^TL(\ell^N)v^N\right).
\end{split}
\end{equation*}

\section{Discretisation and convergence of the hypermodel}
\label{sec:discretisation}

Let us now choose a Mat\'ern field hyperprior for $u$ as well as Mat\'ern prior for $v$, and use Equation \eqref{eqn:exp} for length-scaling $\ell$. 
In this section, we will first consider  discretisation and convergence of this hypermodel, and then make some notes, when we modify the hypermodel to have Cauchy walk hyperprior.

As above, let  the realisations of the  random field $u$  be  tempered distributions that satisfy  
\begin{equation}\label{eq:GMRF}
\left(1 - \ell_0 ^2 \Delta \right) u  = \sigma_0 \sqrt{\ell_0^d} \widetilde { w},
\end{equation}
where $ \widetilde  w$ is  $\mathcal S'(\mathbb R^d)$-valued Gaussian white noise on $\mathbb R^d$ and $\ell_0,\sigma_0>0$ are given constants.   We assume   that    $\widetilde  w$ is statistically independent from $w$.

It is easy to verify that $u$ is a measurable transformation of $\widetilde w$ and it  has  almost surely continuous sample paths for $d=1,2$. 
Hence the right hand side of \eqref{eqn:isotropicmodulated} is a well-defined generalised random field with distribution
\begin{equation*}
\mu_{\sqrt{\ell(\cdot ;u)^d}   w } (A) =\int \mu_ { \sqrt{\ell(\cdot; f)^d } w} (A)  \mu_{u}(df)
\end{equation*}
on Borel fields  $A$ (with respect to the weak$^*$-topology) of $\mathcal S'(\mathbb R^d)$.

In the first step, we  approximate the random  field $v$  without approximating $u$ or $\ell$. We  discretised the Laplacian with    finite differences  in  \eqref{eqn:isotropicmodulated}  and, furthermore,  discretised   the white noise $w$.

 Let us define  a suitable mesh space, where the convergence is studied. We equip the space of all real-valued functions $v^N$ on the 
 mesh $h\mathbb Z^d\cap \overline{D}$ with norm 
\begin{equation*}
\left \Vert v^N\right\Vert _{L^2(D_h)} =  \left( \sum_{kh\in  \overline{D} }  h^2 v^N(kh)^2 \right)^\frac{1}{2}.
\end{equation*}
 
We will use the following notations: we denote by $B$ a suitable boundary operator that stands  e.g.\ for the periodic or 
the Dirichlet boundary condition. 
In the next theorem 
\begin{equation*}
w^N= T_h w
\end{equation*}
is the Steklov mollified version of the white noise. That is,  the radonifying transformation 
$T_h$ is defined as 
\begin{equation*}
T_h f =S_1^2 S_2^2 f, 
\end{equation*}
where 
\begin{equation*}
S_1  f(x_1,x_2)=\frac{1}{h}\int_{x_1 -h/2} ^{x_1+h/2} f(t,x_2)d t
\end{equation*}
and
\begin{equation*}
S_2  f(x_1,x_2)=\frac{1}{h}\int_{x_2 -h/2} ^{x_2+h/2} f(x_1,t)d t
\end{equation*}
for all $f\in L^2(D)$.

\begin{lemma}\label{lemma:bound}
Let $u$ be a Gaussian random field that satisfies 
\eqref{eq:GMRF}. 
Let  $v(x;u)$ satisfy 
\begin{equation}\label{eq:limit1}
\left(\ell(x;u)^{-2} -  \Delta \right) v  = \sigma_0 \ell(x;u)^{d/2-2}   w \text{ in } D 
\end{equation}
with the periodic boundary condition, where $\ell(x;u)= g(u(x))$ and 
\begin{equation*}
g(s)  =\exp(s)
   \end{equation*}
  Let $v^N(x;u)$ be    \begin{equation}\label{eq:approx1}
\left(\ell(x;u) ^{-2} -  \Delta_N \right) v^N (x;u)   = \sigma_0  \ell(x;u)^{d/2-2}   w^N ,  
\end{equation}
on $h\mathbb Z^d \cap D$,  with the boundary condition $B v^N =0$ on $h Z^d\cap \partial D$.

Then $L^2(L^2(D_h), P)$-norm  of  $v^N -v$ converges to zero   
  as $h\rightarrow 0$.  
\end{lemma}
\begin{proof}

Conditioning with $u$  inside $L^2(L^2(D_h),P)$-norm gives us 
$$
\mathbb E \left[  \left\Vert v^N(\cdot;  u )- v(\cdot; u) \right\Vert_{L^2(D_h)}^2    \right] =
\mathbb E \left[  \mathbb E \left[  \left\Vert v^N (\cdot; u)  - v (\cdot;  u) \right\Vert _{L^2(D_h)}^2
| u \right] \right].  
$$
Recall that  $u$ is a radonifying transformation of $\widetilde w$, where $\widetilde w$  is statistically independent from $w^N$. 
Then also $u$ and $w^N$  are statistically independent. 
Moreover, $v^N$ is a Carath\'eodory function of  $(u,w^N)$ (which is radonifying with respect to the second variable).  
This means that  conditioning   $v^N$ with $u=u_0$, where $u_0\in C(\overline D)$,  only replaces the random  coefficient $\ell(\cdot;u)$ in \eqref{eq:approx1} with a fixed continuous function $\ell(\cdot ;u_0)$.  
The same holds for $v^N, w^N$ replaced with $v,w$. 

Let us  denote  with $v(\cdot; u,f)$  and $v^N(\cdot ; u,f)$   the solutions of 
\eqref{eq:limit1} and \eqref{eq:approx1}, respectively, when the white noise load 
$w$ is replaced with a function $f\in L^2(D)$. 

By applying adjoints of the solution operators, it is easy to verify that 
$$
\mathbb E \left[ \left\Vert v^N - v \right\Vert _{L^2(D_h)}^2  \big \vert u \right]  =       h^2 \sum_{kh\in D } \sup_{\Vert f\Vert_{L^2}\leq 1 } \left(v^N (kh; u, f)- v (kh; u,f)\right)^2    
$$
due to linearity of the elliptic problem. 

  By the usual convergence results for the finite-difference scheme (see p.\ 214 in \cite{MR3136501}, with straightforward 
  changes for the periodic case), we obtain
\begin{equation}\label{eq:upperbound}
 \left(v^N (kh; u, f)- v (kh; u,f)\right)^2  \leq  C h^2 \Vert v(u,f)\Vert_{W^2_2(D)}^2.
\end{equation}
By inserting elliptic estimates  (see Lemma \ref{lemma:ell} below) into \eqref{eq:upperbound}, we get the upper bound
$$
\mathbb E \left[ \left\Vert v^N - v \right\Vert _{L^2(D_h)}^2  \big \vert u \right]   \leq  C_u  \sum_{kh\in D } \sup_{\Vert f\Vert_{L^2}\leq 1 } h^4 
  \Vert v(u,f)\Vert_{L^2(D)}^2 \leq  C_u |D| h^2.  
$$
where the constant $C_u\in L^2(P)$.
\end{proof}

\begin{remark}
The above lemma can be easily generalised for the case when $u$ has almost surely bounded 
sample paths  and  $g$ is bounded from above and below with positive constants.
Hence, we obtain similarly the convergence for the Cauchy walk. 
\end{remark}
For completeness, we recall the following elliptic estimate. 
\begin{lemma}\label{lemma:ell}
$\Vert v(u,f) \Vert _{W^2_2(D)}\leq C \Vert f \Vert _{L^2(D)}$, 
where the constant $C\in L^p(P)$ for all $p\geq 1$.
\end{lemma}
\begin{proof}
Take $\sigma_0=1$ for simplicity.
Let 
$$
\left(\ell(x;u)^{-2} -  \Delta \right) v  =  \ell(x;u)^{d/2-2}   f 
$$
with periodic boundary conditions. Let us write a corresponding 
integral equation with the help of the operator $ G_{c_0}=(-\Delta + c_0(u))^{-1}$, where 
$c_0(u)=\inf_{x\in D}(\ell(x,u))^2/2$. Then 
$$
v + G_{c_0}(\ell(\cdot; u)- c_0)v =  G_{c_0}\ell(\cdot;u)^{d/2-2} f. 
$$
With Fourier techniques on the torus, we can show that $G_{c_0}: L^2(D)\rightarrow H^2(D)$.
Moreover, the norm of the mapping is bounded by the maximum of $1$ and $\sqrt{c_0^{-1}}$. 
Therefore,
\begin{eqnarray*}
\Vert v \Vert_{H^2}^2 &\leq & \max(1, c_0^{-1}) \sup_{x\in D }(\ell(x;u)^{d-4}) \Vert f\Vert_{L^2}^2 + \sup_{x\in D }(\ell(x; u)^2 - c_0)^2 \Vert v\Vert_{L^2}^2 \\
&\leq& \max(1, c_0^{-1})  \sup_{x\in D }(\ell(x;u)^{d-4}) \Vert f\Vert_{L^2}^2 +  \frac{\sup(\ell(\cdot; u)^2- c_0)^2}{
\inf (\ell(\cdot; u)^2)}\Vert f\Vert_{L^2}^2
\end{eqnarray*}
by Babu\^ska-Lax-Milgram theorem.  The multipliers of  the norm belong to $L^p(P)$ for all $p\geq 1$ \cite{charrier}.
\end{proof}

Next, also  Equation \eqref{eq:GMRF} is  discretised on the equidistant mesh $h\mathbb Z^d$ by   finite differences and
discretisation of the white noise \eqref{eqn:whitenoisediscrete}. We further modify 
the  equation   
\begin{equation*}%\label{eq:discreteGMRF}
\left(1 - \ell_0 ^2 \Delta_h  \right)  u^N (hk)  = \sigma_0 \sqrt{\ell_0^d} w^N (hk),  \;  k\in\mathbb Z^d, 
\end{equation*}
for discrete $u^N(hk)$ by expressing the discrete white noise $w^N$
as  the measurable  transformation 
$
w^N (hk)  = (T_h  w) (hk)  
$
of the continuous-parameter white noise (for details on measurable transformations, see \cite{Bogachev1998}). 

\begin{theorem}\label{th:converg}
Let  $v(x;u)$ satisfy 
\begin{equation}\label{eq:limit2}
\left(1 - \ell(x;u) ^2 \Delta \right) v  = \sigma_0 \sqrt{\ell(x;u)^d}   w \text{ in } D 
\end{equation}
with the periodic boundary condition where $\ell(x;u)= g(u(x))$ and 
\begin{equation*} 
g(s)  = \exp(s)
\end{equation*}
  Let $v^N(x;u^N)$ satisfy
\begin{equation*}
\left(1 - \ell(x;u^N) ^2 \Delta_N \right) v^N (x;u^N)   = \sigma_0 \sqrt{\ell(x;u^N)^d}   w^N ,  
\end{equation*}
on $h\mathbb Z^d \cap D$,  with the periodic boundary.

Then $v^N (\cdot; u^N)$ converges to  $v$ in 
$L^2(L^2(D_h), P)$  as $N\rightarrow \infty$.  
\end{theorem}
\begin{proof}
Consider the  norm 
\begin{equation}\label{eq:converg} 
\begin{split} 
\mathbb E \left[ \left \Vert v^N(\cdot;  u^N  )- v(\cdot; u) \right\Vert_{L^2(D_h)}^2    \right] \leq & 2 
 \mathbb E \left[  \left\Vert v^N(\cdot; u ^N)- v(\cdot; u^N ) \right\Vert_{L^2(D_h)}^2  \right] \\
   & +  2 \mathbb E \left[ \left\Vert v(\cdot; u^N )- v(\cdot; u) \right\Vert_{L^2(D_h)}^2    \right],  
 \end{split}
\end{equation}
where $v(\cdot; u^N)$ solves \eqref{eq:limit2} for some continuous  pointwise convergent  interpolation of  $u^N$ in place of $u$.  By Lemma \ref{lemma:bound}, the 
first term of \eqref{eq:converg} vanishes when $h\rightarrow 0$. We show  that the 
second 
term of \eqref{eq:converg} vanishes as $h\rightarrow 0$.  Indeed,
$$
\mathbb E \left[ \left\Vert v (\cdot,u^N)  - v(\cdot, u)  \right\Vert _{L^2(D_h)}^2  \big \vert u \right]  =       h^2 \sum_{kh\in D } \sup_{\Vert f\Vert_{L^2}\leq 1 } \left(v (kh; u^N , f)- v (kh; u,f)\right)^2 .   
$$
If we can show that $v (x; u^N , f)- v (x; u,f)$ converges to zero uniformly with respect to $x $ and  $f$ , we 
are done. The term can be considered with the help of Sobolev spaces   
\begin{equation*}
\begin{split}
 v (x; u^N , f)- v (x; u,f) &= \left(  v (\cdot ; u^N , f)- v (\cdot ; u,f), \delta_ x  \right)_{H^{d/2+\delta, -d/2 -\delta}}  \\
 &\leq 
  C  \Vert  v (\cdot ; u^N , f)- v (\cdot ; u,f) \Vert _{H^{d/2+\delta}}
\end{split}
\end{equation*}
We denote the Green's operator for $\ell(x;u) ^{-2}  - \Delta $ with $G_{\ell(\cdot; u)}$ and 
for  $  \ell(x;u^N) ^{-2}  -\Delta $ with $G_{\ell(\cdot; u^N)}$. Then 
\begin{equation*}
\begin{split}
\left \Vert  v (\cdot ; u^N , f)- v (\cdot ; u,f) \right\Vert _{H^{d/2+\delta}} & \leq 
  \left \Vert \left( G_{\ell(\cdot; u)}   - G_{\ell(\cdot; u^N)} \right)  \sqrt{\ell(\cdot;u)^{d-1}} f \right\Vert _{H^{d/2+\delta}}
  \\  &+   \left \Vert G_{\ell(\cdot; u^N )}  \left(  \sqrt{\ell(\cdot;u)^{d-1}}  -   \sqrt{\ell(\cdot;u^N)^{d-1}} \right)  f \right\Vert 
   _{H^{d/2+\delta}}. 
\end{split}
\end{equation*}
 By using integral equation techniques, we can show that 
$\Vert G_{\ell(\cdot; u^N )}\Vert_ {L^2, {H^{d/2+\delta}}} $ are uniformly bounded with respect to $N$, and the 
upper bound belongs to $L^2(P)$. The uniform boundedness of  $\sup_{x\in D } \ell(x;u^N)$ follows 
from Sobolev space estimates and convergence of $u^N$ to $u$ in discrete Sobolev norms. 
Hence, the second term converges.  The first term converges  similarly since 
 $$
 G_{\ell(\cdot; u)}   - G_{\ell(\cdot; u^N)}  =  G_{\ell(\cdot; u)} \left(\ell(\cdot;u^N)- \ell(\cdot;u)\right) G_{\ell(\cdot; u^N)}. 
 $$
 \end{proof}

\begin{remark}
We can show that the Steklov mollification can be replaced with the weaker mollification \eqref{eqn:whitenoisediscrete} with  similar technique as in  the proof of Theorem \ref{th:converg}.
\end{remark}

\begin{remark}
Theorem \ref{th:converg} holds also for the Cauchy walk, when $\ell$ is uniformly bounded from above and below.
\end{remark}

\section{Rough hypermodels}
\label{section:roughhypermodels}

In addition to the models mentioned above, we may wish to use a Cauchy noise, a Cauchy walk, or, white noise, which have rougher sample paths than the previously discussed examples.
For the discrete white noise, see Equation \eqref{eqn:whitenoisediscrete}. 

Let us consider a one-dimensional Cauchy walk and its discretisation \cite{Markkanen2016}. 
The Cauchy walk 
$u(x)$ is  an $\alpha$-stable L\'evy motion with $\alpha=1$ defined  by
$$
u(x)=M([0,x]),
$$ 
where $M(A)$ has Cauchy probability density function 
$$
f(x)= \frac{|A|}{ \pi (|A|^2+ x^2)}
$$
for all Borel sets $A\subset \mathbb R_+$. We denote $u(x)\sim \mathrm{Cauchy}(|x|,0)$. The Cauchy walk 
is right-continuous and has finite left limits.  The discretisation of Cauchy walk is based on independent increments 
\begin{equation*}
u(hj)-u(h(j-1)) \sim \mathrm{Cauchy}(h,0).
\end{equation*}
We note that $\ell(x;u)$ is not stationary.
To control the length-scaling in the elliptic equation \eqref{eqn:isotropicmodulated}, we make a transformation 
$\ell(x;u)= g(u(x)),$
where 
\begin{equation} \label{eqn:cauchymap}
g(s)=\frac{a}{b+c \vert s\vert} +d,
\end{equation}
and $d>0$ is a fixed small constant, and $a,b,c>0$ are suitably chosen constants. 
The corresponding discretised hypermodel can then be constructed as
\begin{equation*}
\begin{split}
 D(u^N)D(v^N \vert \ell^N(u^N))  \propto& \prod_{j=1}^N \frac{h}{h^2+(u_j^N-u_{j-1}^N)^2} \times\dots \\ &{\vert  L (\ell^N) \vert}\exp\left(-\frac{1}{2}(v^N)^TL(\ell^N) ^TL(\ell^N)v^N\right).
 \end{split}
\end{equation*}
We note again that we have included the normalisation constant as $L(\ell^N)$-matrix depends on the length-scaling $\ell^N$, i.e.\ it is not a constant.
Similarly, we could use discrete Cauchy noise, which is an iid process with rougher sample paths. For  discussion of the Cauchy noise for Bayesian statistical inverse problems, see  Sullivan 2016 \cite{Sullivan2016}.
Using again Equation \eqref{eqn:cauchymap}, discrete Cauchy noise $u^N$ has typically values around zero, or 'sporadically' some big values. 
Hence, $\ell^N$ is typically either long, or 'sporadically' short,  i.e.\ the chosen  hypermodel promotes either long or short length-scaling.
We note that  transformation \eqref{eqn:cauchymap} is  constructed to be symmetric with respect to zero.

\subsection{Realisations}

In Figure \ref{fig:foto1}, we have plotted realisations of Cauchy and Gaussian process hyperparameters $\ell_\omega^N$, the resulting covariance matrices and realisation of $v^N$. 
The realisations $v_\omega^N$ clearly have non-stationary features, parts where we have highly oscillatory features and edges, and then smoother parts.
In the Bayesian inversion analysis itself, i.e.\ in the posterior distribution, the $\ell^N$ is a parameter vector to be estimated, and hence to find where the non-stationarities, like the edges are located.

\begin{figure}[htp]
      \centering
      \subcaptionbox{Length-scale realisation $\ell^N_\omega$ given Cauchy walk hyperparameter realisation $u^N_\omega$.}
              {\includegraphics[width=0.32\textwidth]{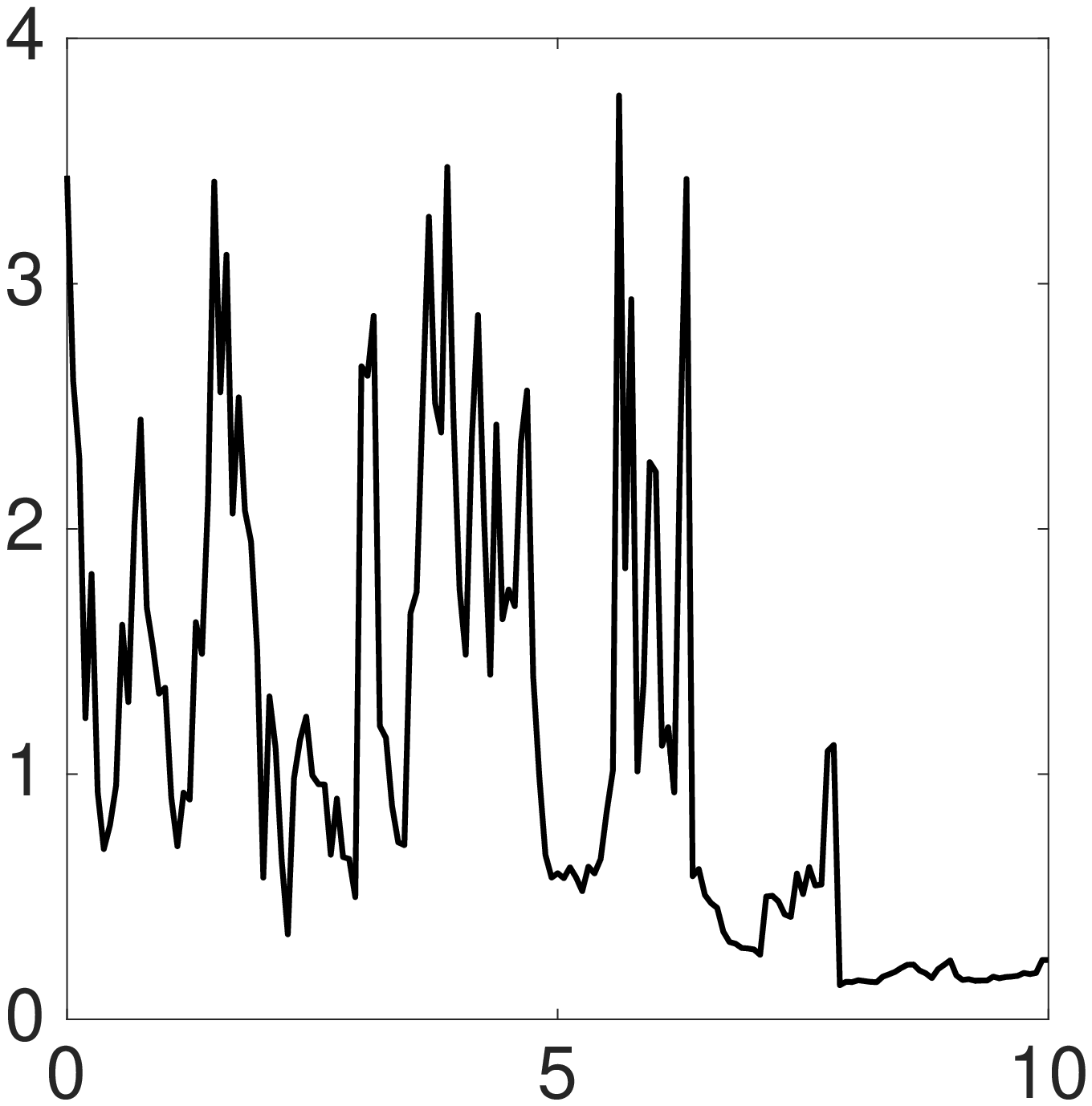}}
      \subcaptionbox{Mat\'ern covariance, given Cauchy walk hyperparameter realisation $u^N_\omega$.}
        {\includegraphics[width=0.32\textwidth]{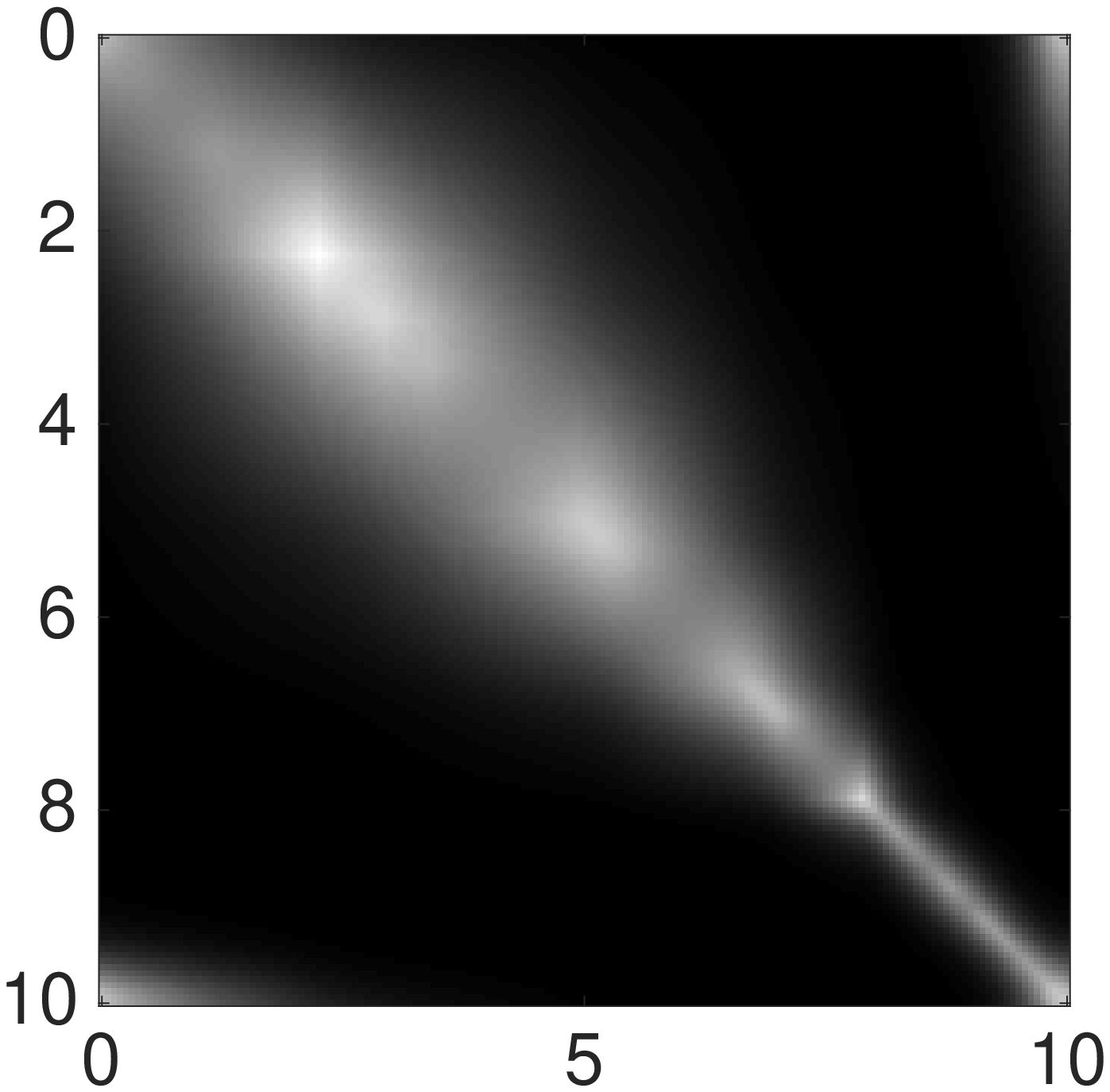}} 
      \subcaptionbox{Realisations $v^N_\omega$ given Cauchy walk hyperparameter $u^N_\omega$.}
        {\includegraphics[width=0.32\textwidth]{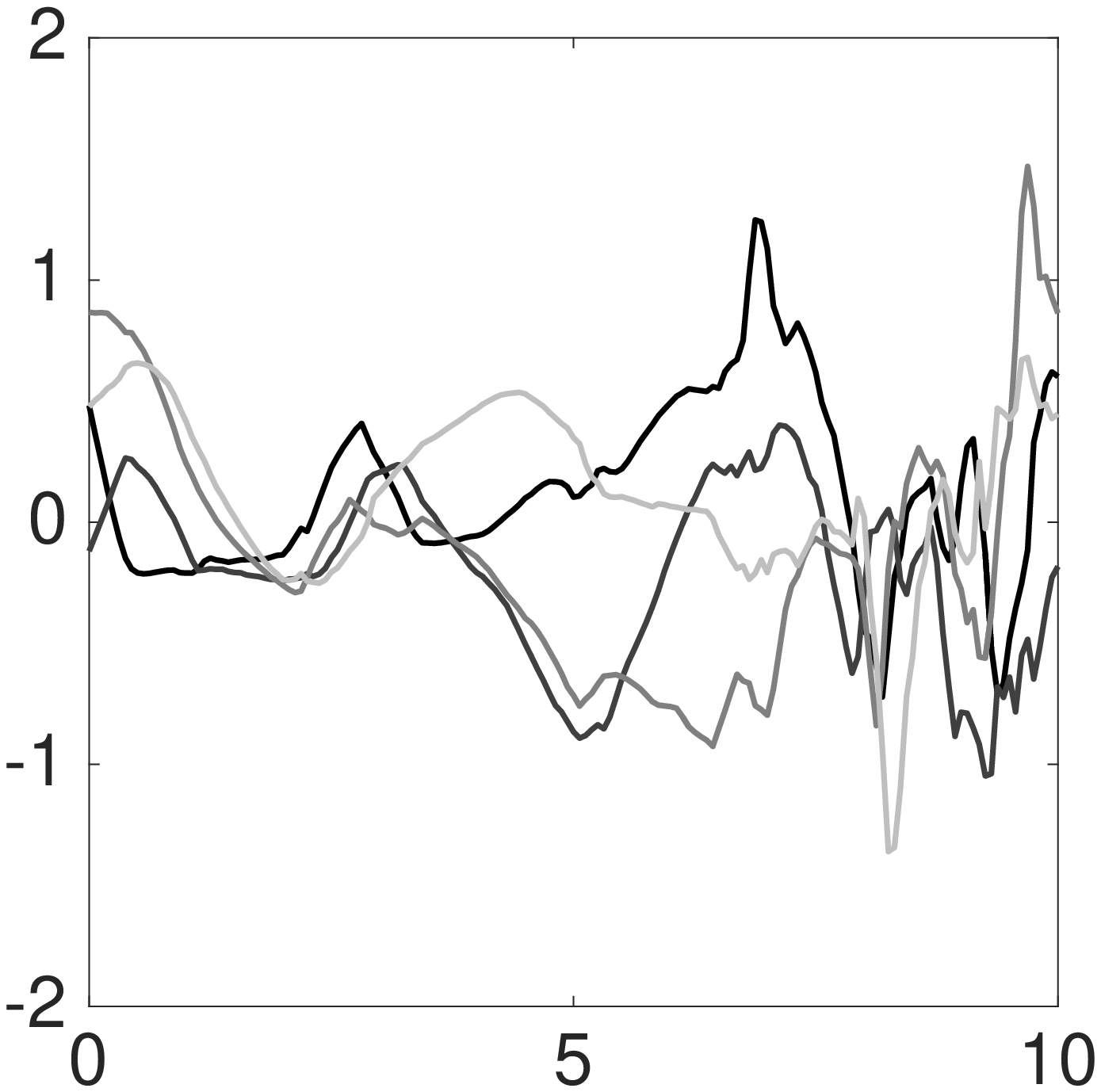}} \\
      \subcaptionbox{Length-scale realisation $\ell^N_\omega$ given Gaussian process hyperparameter realisation $u^N_\omega$.}
              {\includegraphics[width=0.32\textwidth]{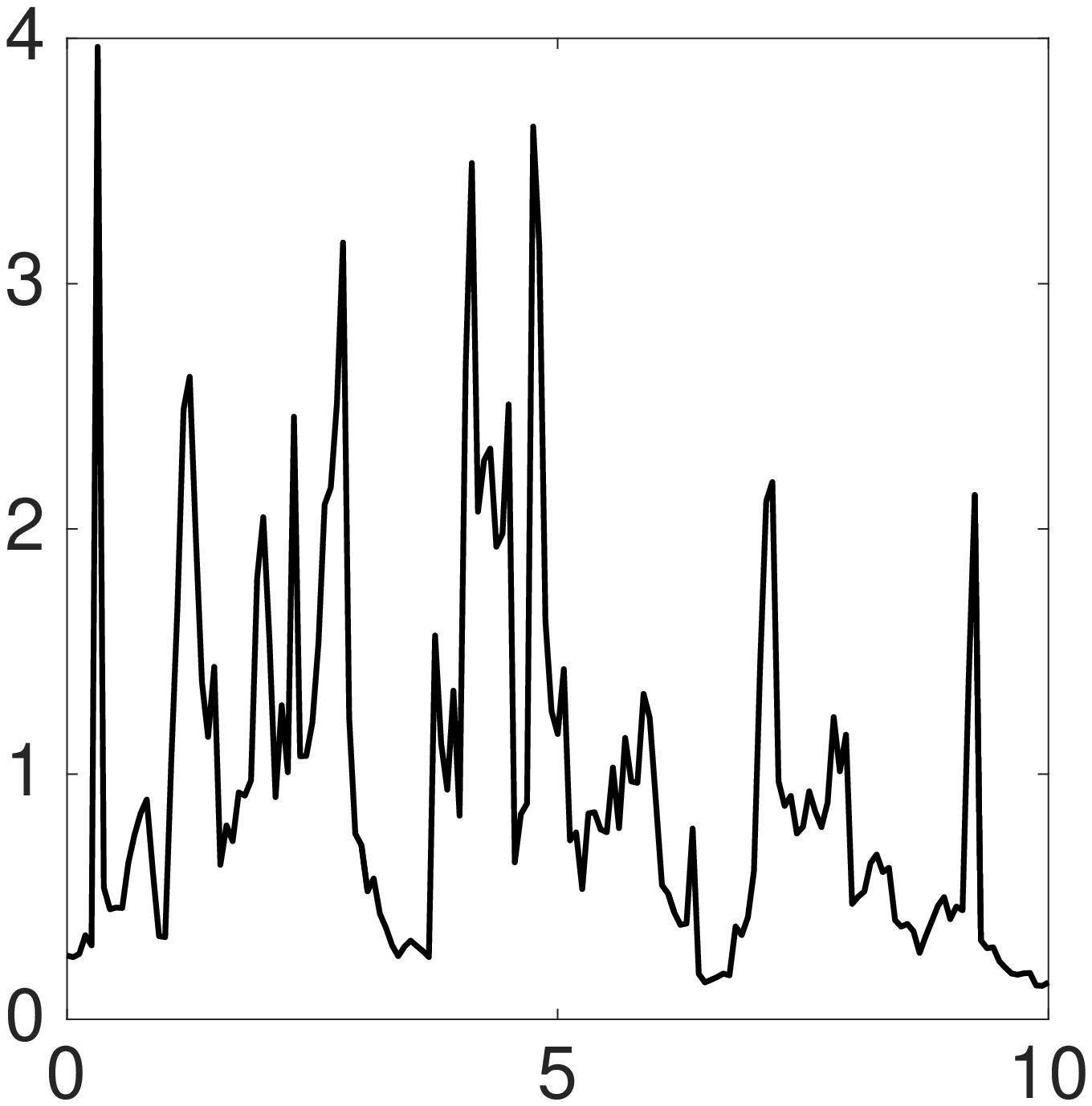}}
      \subcaptionbox{Realisations $v^N_\omega$ given Gaussian process hyperparameter $u^N_\omega$.}
        {\includegraphics[width=0.32\textwidth]{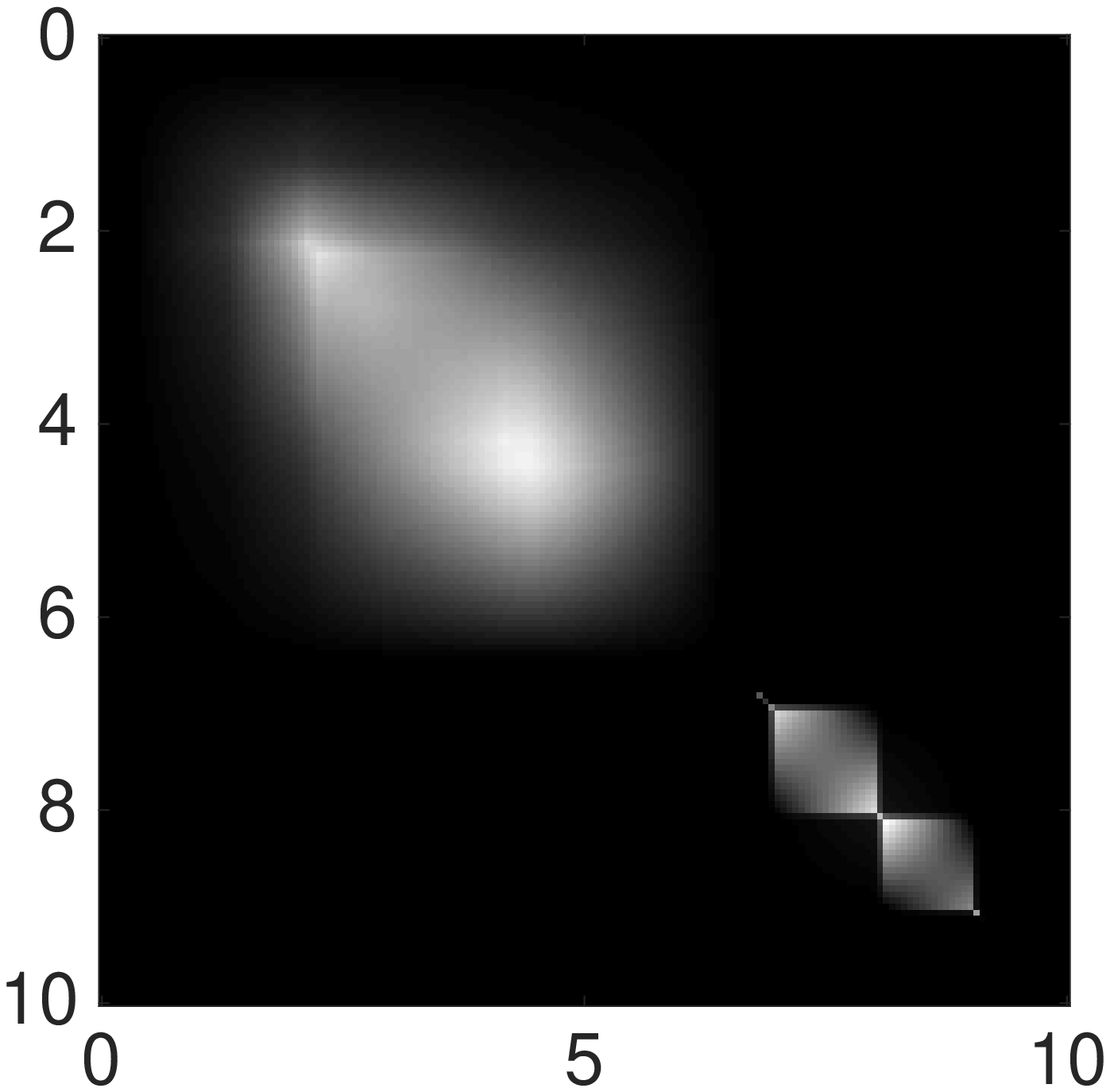}} 
      \subcaptionbox{Realisations $v^N_\omega$ given Gaussian process hyperparameter $u^N_\omega$.}
        {\includegraphics[width=0.32\textwidth]{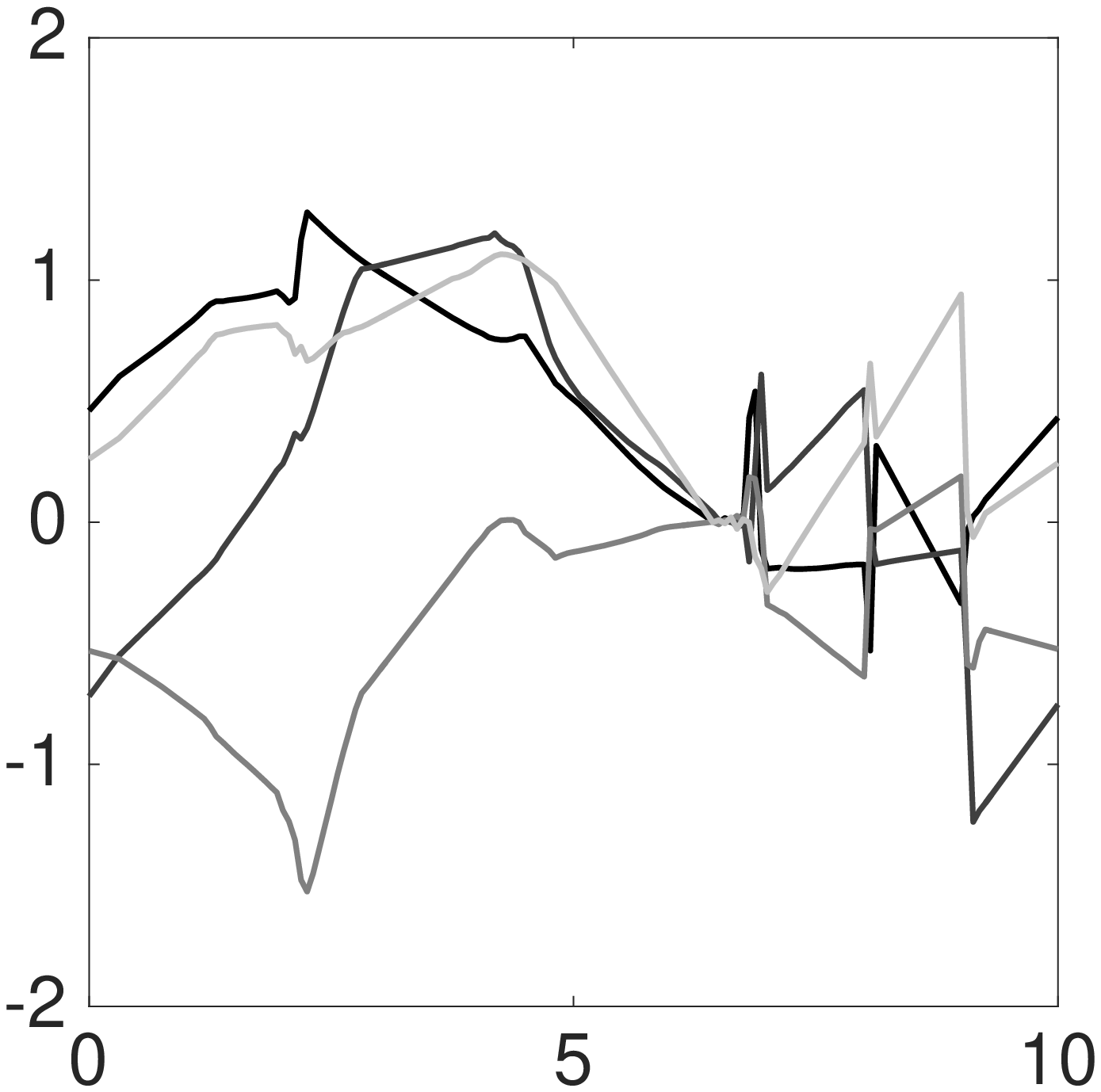}} 
        
  \caption{Examples of constructing non-stationary Mat\'ern realisations with hypermodels. Top panel -- from left to right: Realisation $\ell^N_\omega$ given Cauchy walk as $u^N_\omega$, resulting covariance matrix, and four realisations. Bottom panel: Same as above, but with a Gaussian process hyperprior.} \label{fig:foto1}
\end{figure}

In Figure \ref{fig:foto2}, we have realisations of the constructed two-dimensional hypermodels.
As hyperprior realisations, we have a constant-parameter Mat\'ern field realisation, and, an inhomogeneous Mat\'ern field realisations obtained by using different values of length-scaling and tilt-angle in different parts of the domain. 
Here we have used an extended domain when constructing the realisation, and in order to remove the boundary effects, we have  cropped the images suitably.
By varying the Mat\'ern field models, we have three different length-scaling fields.
These fields are then used as input for $g(s)$, and in the bottom row, we have plotted realisations of the prior.
In the bottom panel, second image from the right, we have used $\ell(u(x)) = \ell_1(x) = 2\ell_2(x)$ and non-zero constant tilt-angle theta.
In this way, we can make also anisotropic features.

\begin{figure}[htp]
      \centering
\begin{tikzpicture}
\node[inner sep=0pt] (foto1) at (4.5/2,0)
    {\includegraphics[width=.22\textwidth]{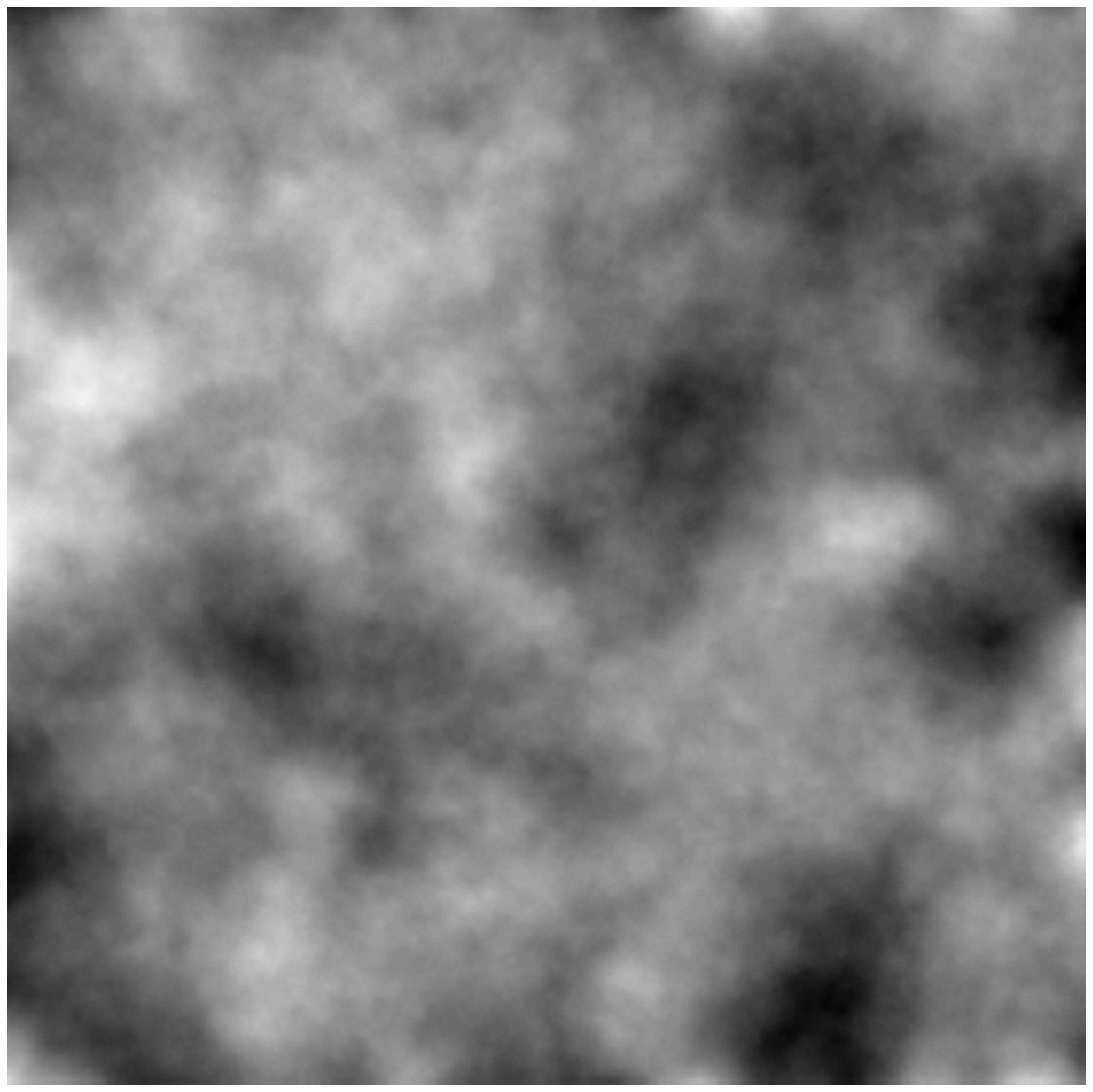}};

\node[inner sep=0pt] (foto2) at (0,-4)
    {\includegraphics[width=.22\textwidth]{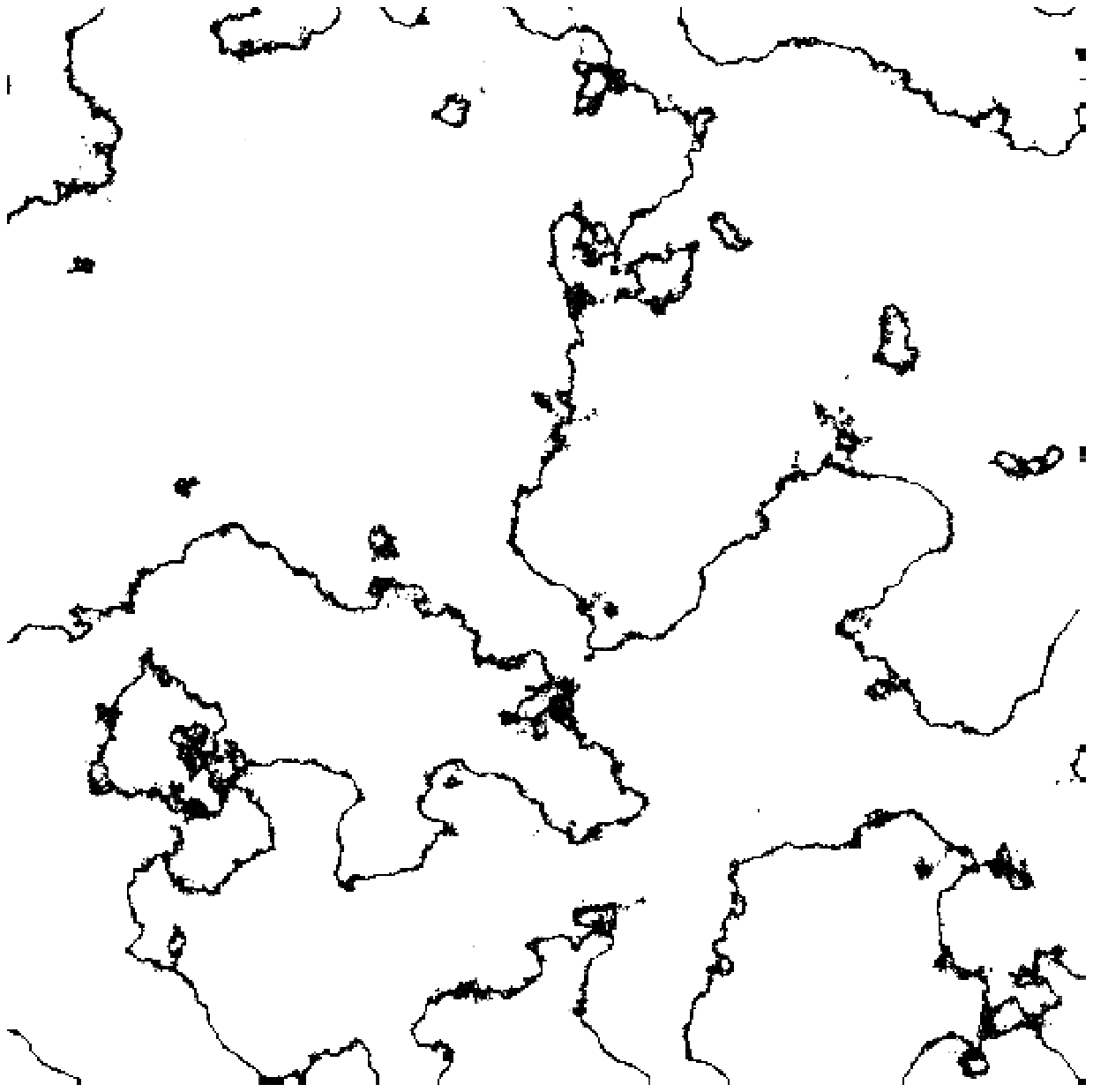}};

\node[inner sep=0pt] (foto2c) at (4.5,-4)
    {\includegraphics[width=.22\textwidth]{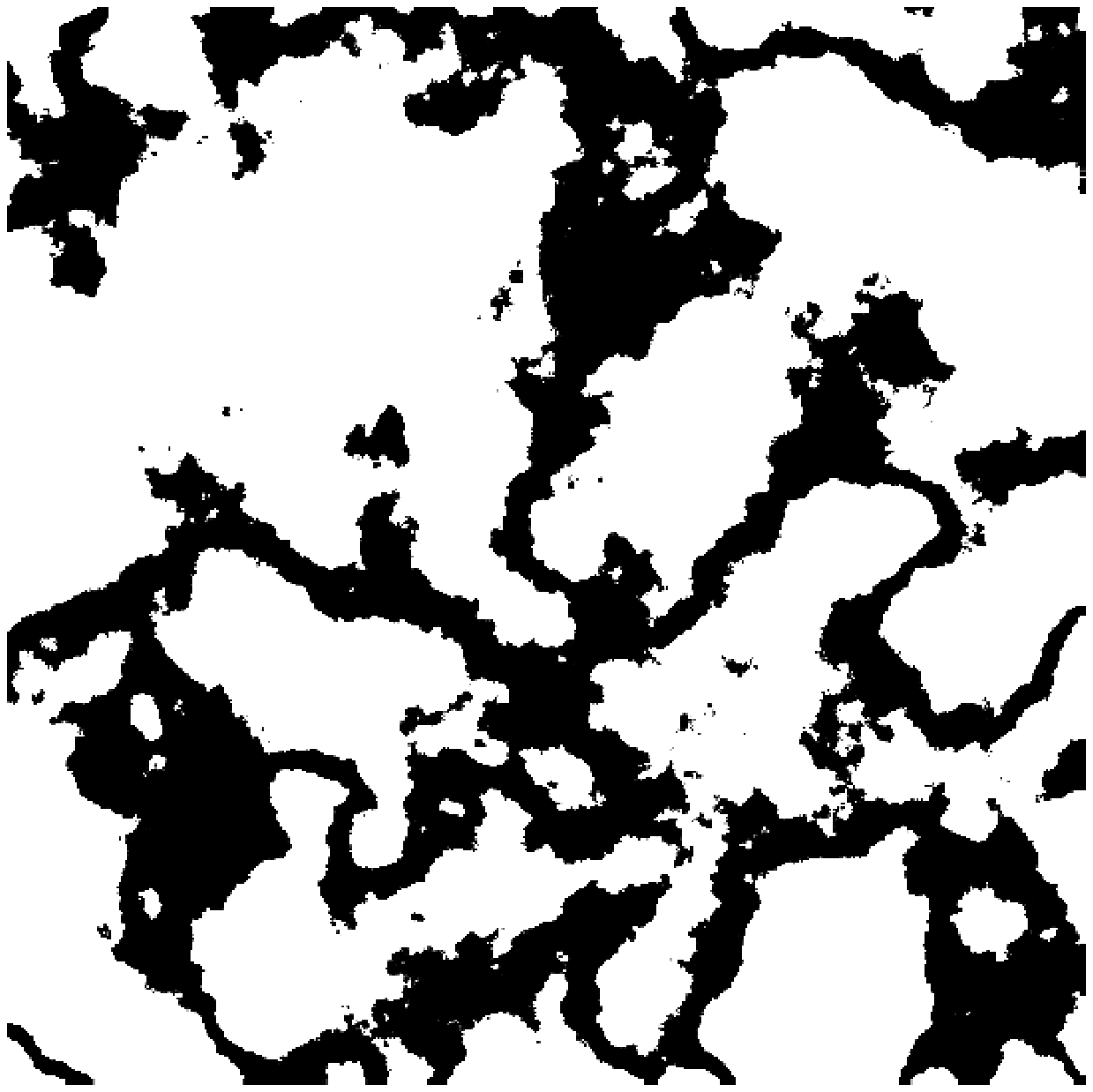}};

\node[inner sep=0pt] (foto3) at (0,-8)
    {\includegraphics[width=.22\textwidth]{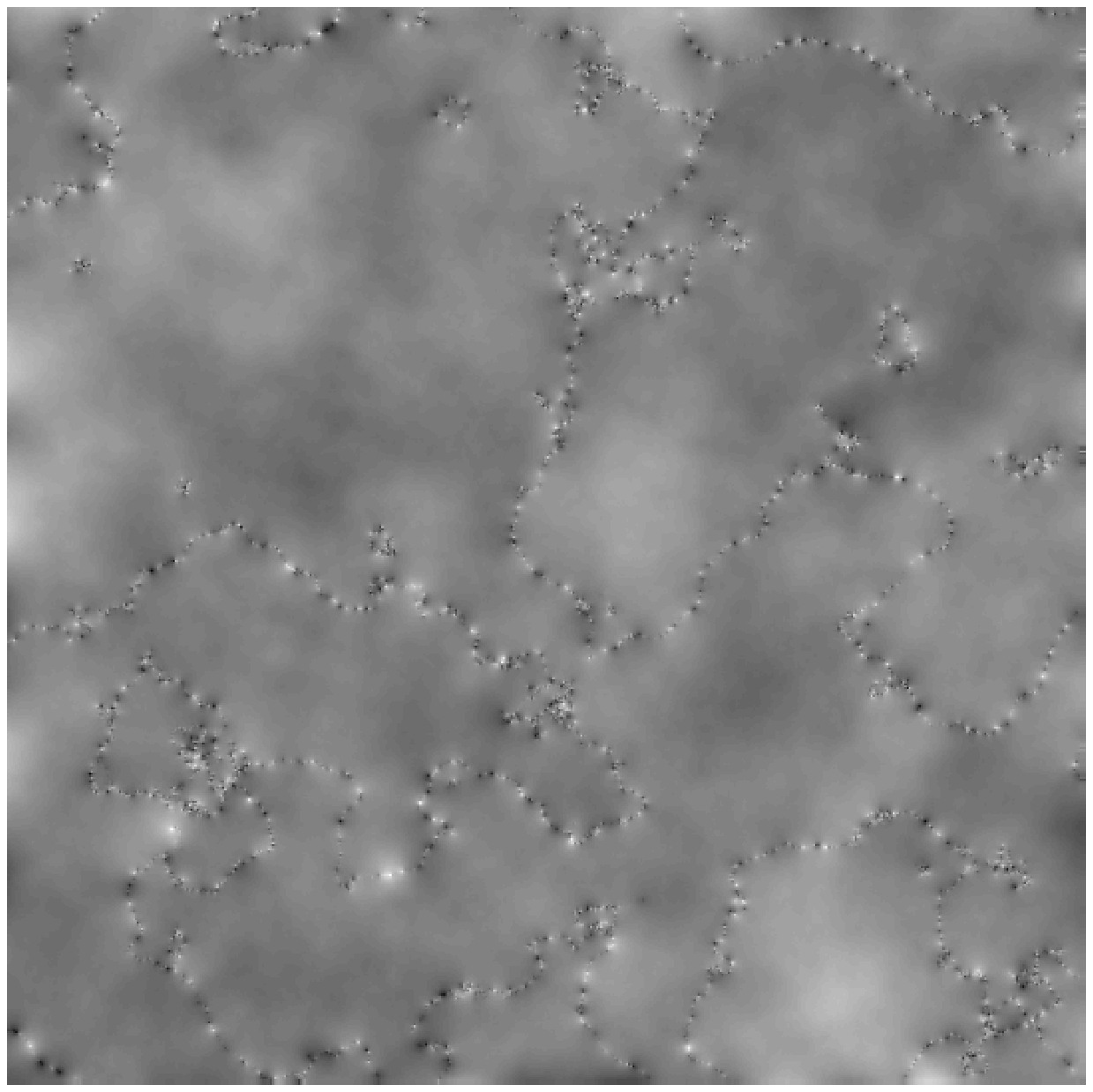}};

\node[inner sep=0pt] (foto3c) at (3,-8)
    {\includegraphics[width=.22\textwidth]{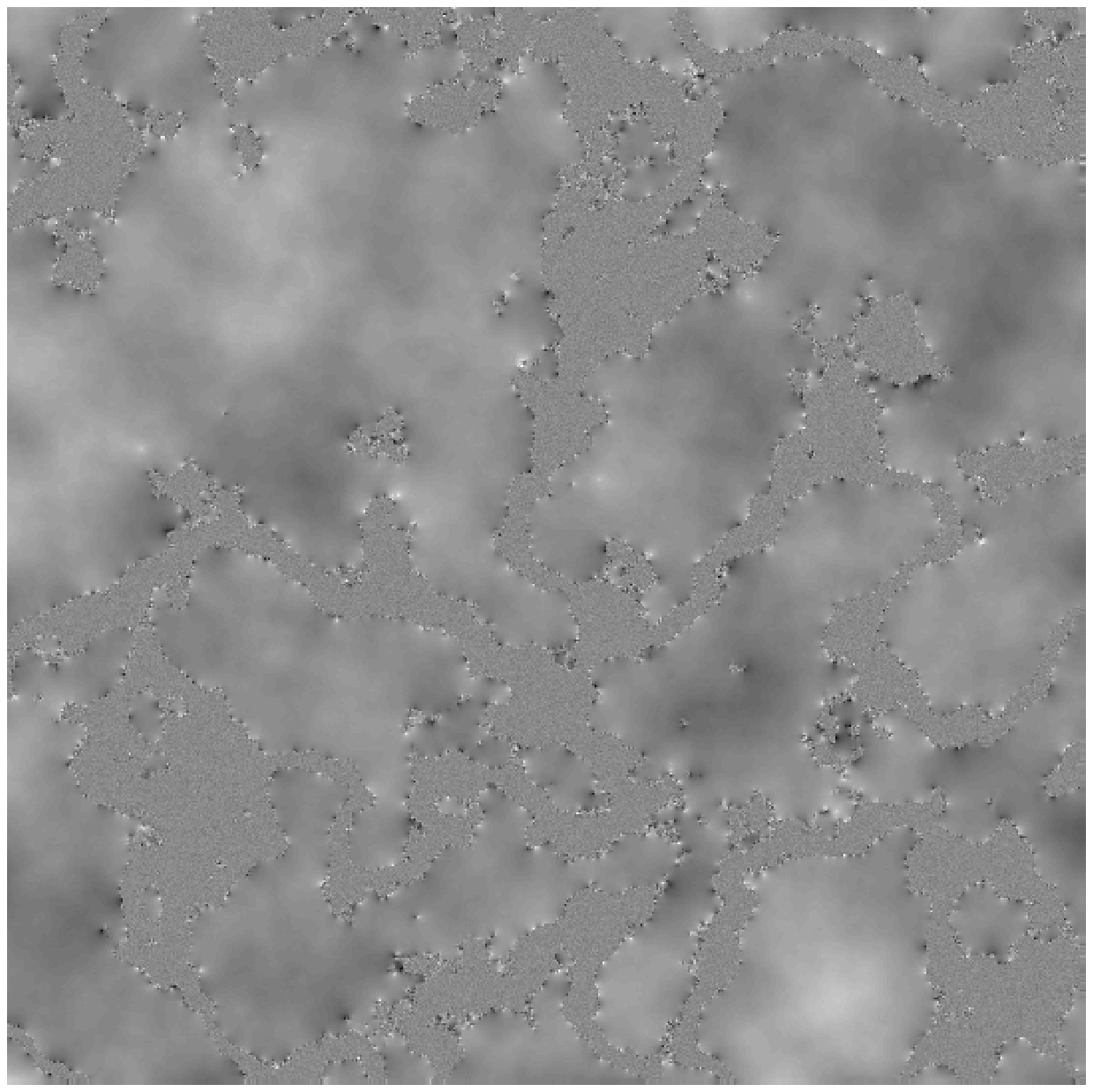}};

\node[inner sep=0pt] (foto3d) at (6,-8)
    {\includegraphics[width=.22\textwidth]{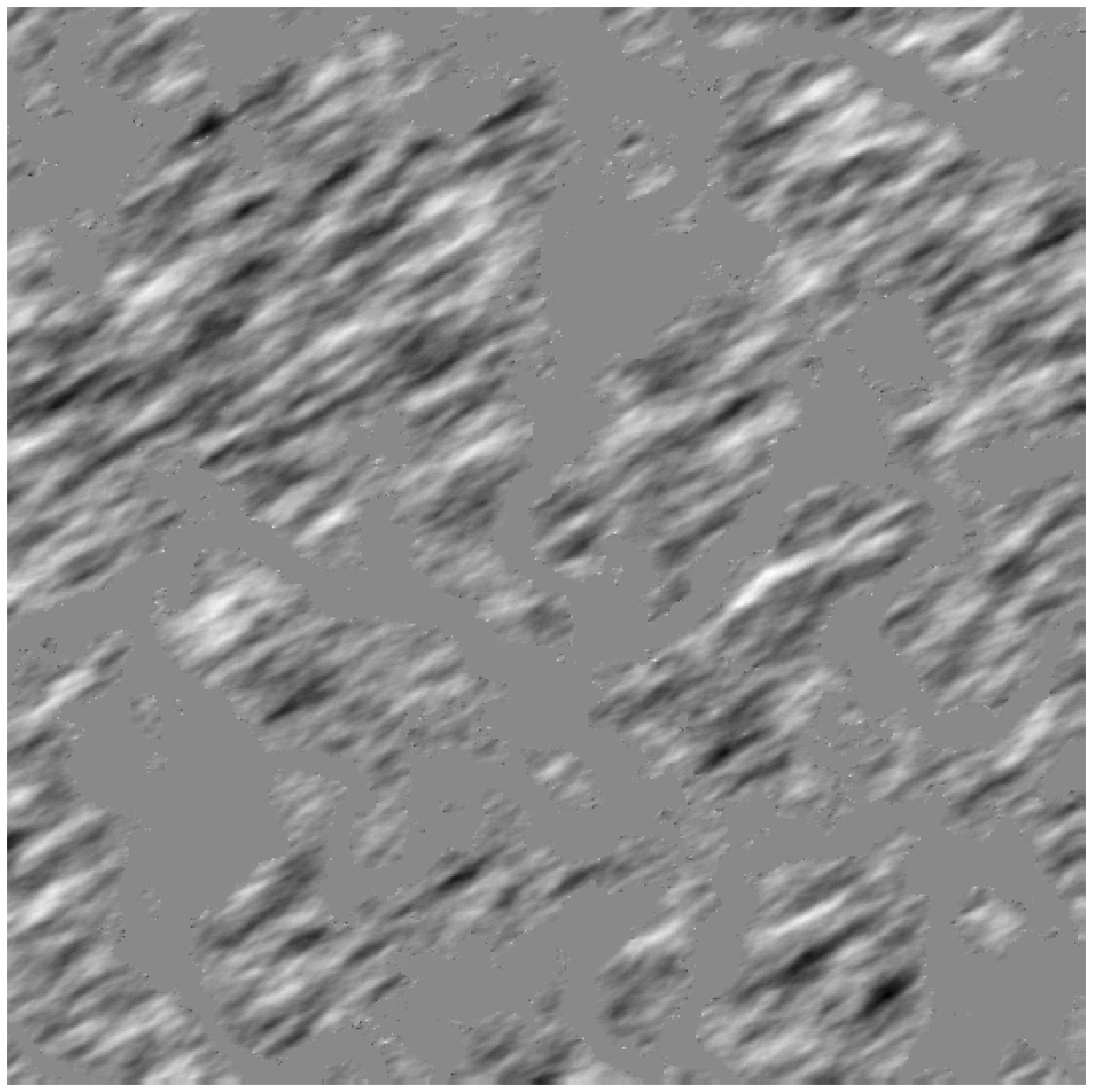}};

\node[inner sep=0pt] (foto1f) at (9,0)
    {\includegraphics[width=.22\textwidth]{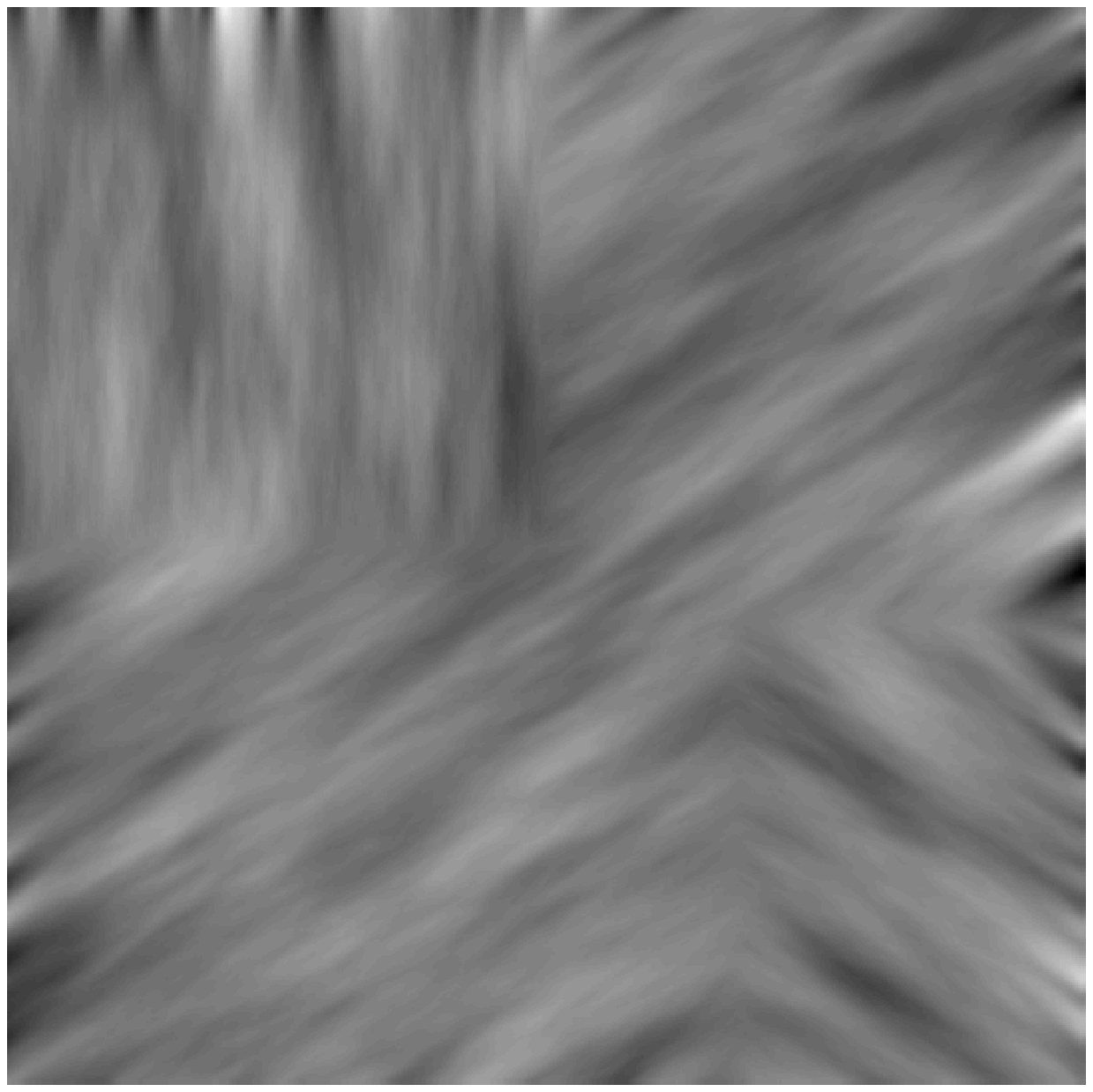}};

\node[inner sep=0pt] (foto2f) at (9,-4)
    {\includegraphics[width=.22\textwidth]{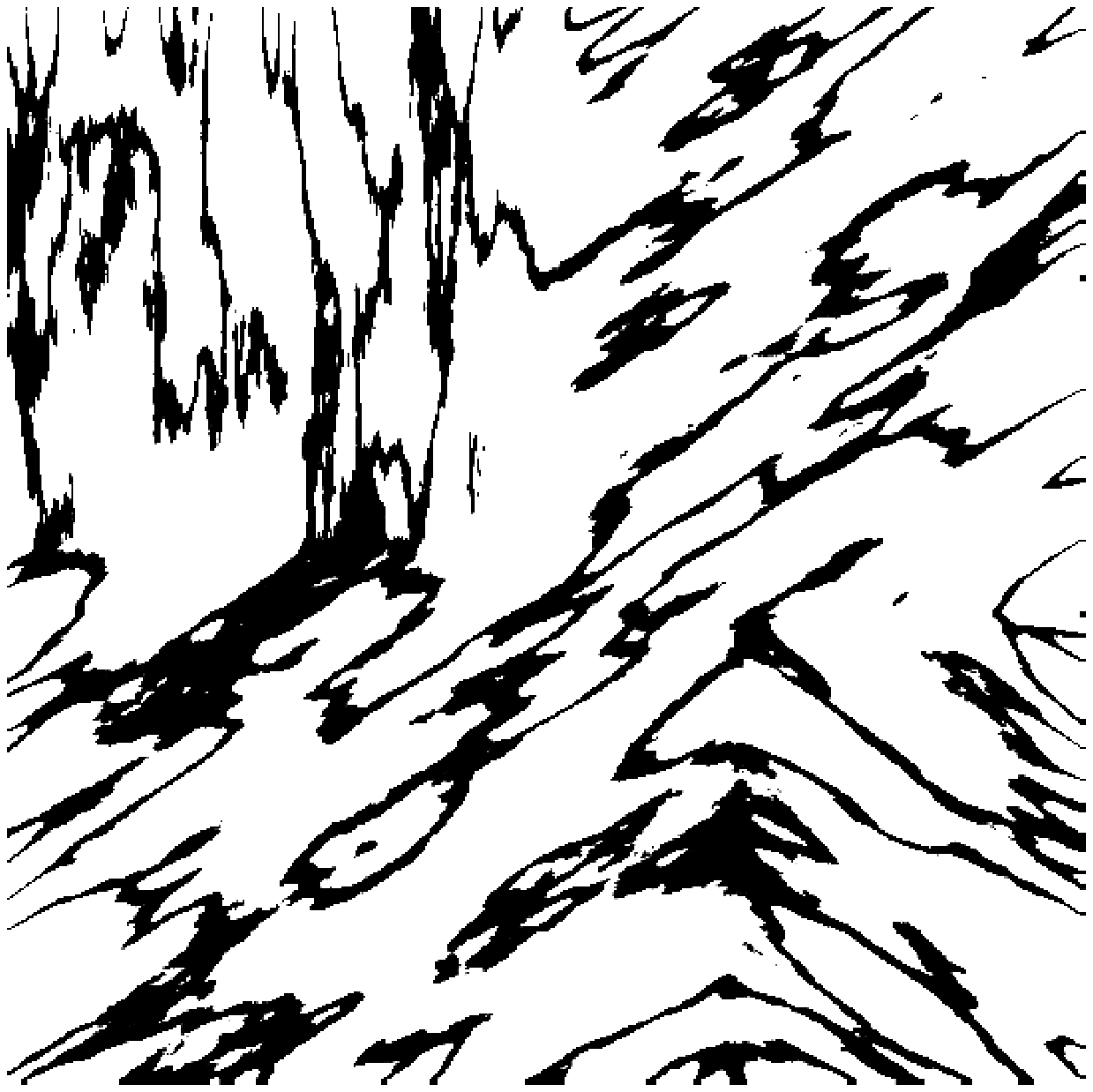}};

\node[inner sep=0pt] (foto3f) at (9,-8)
    {\includegraphics[width=.22\textwidth]{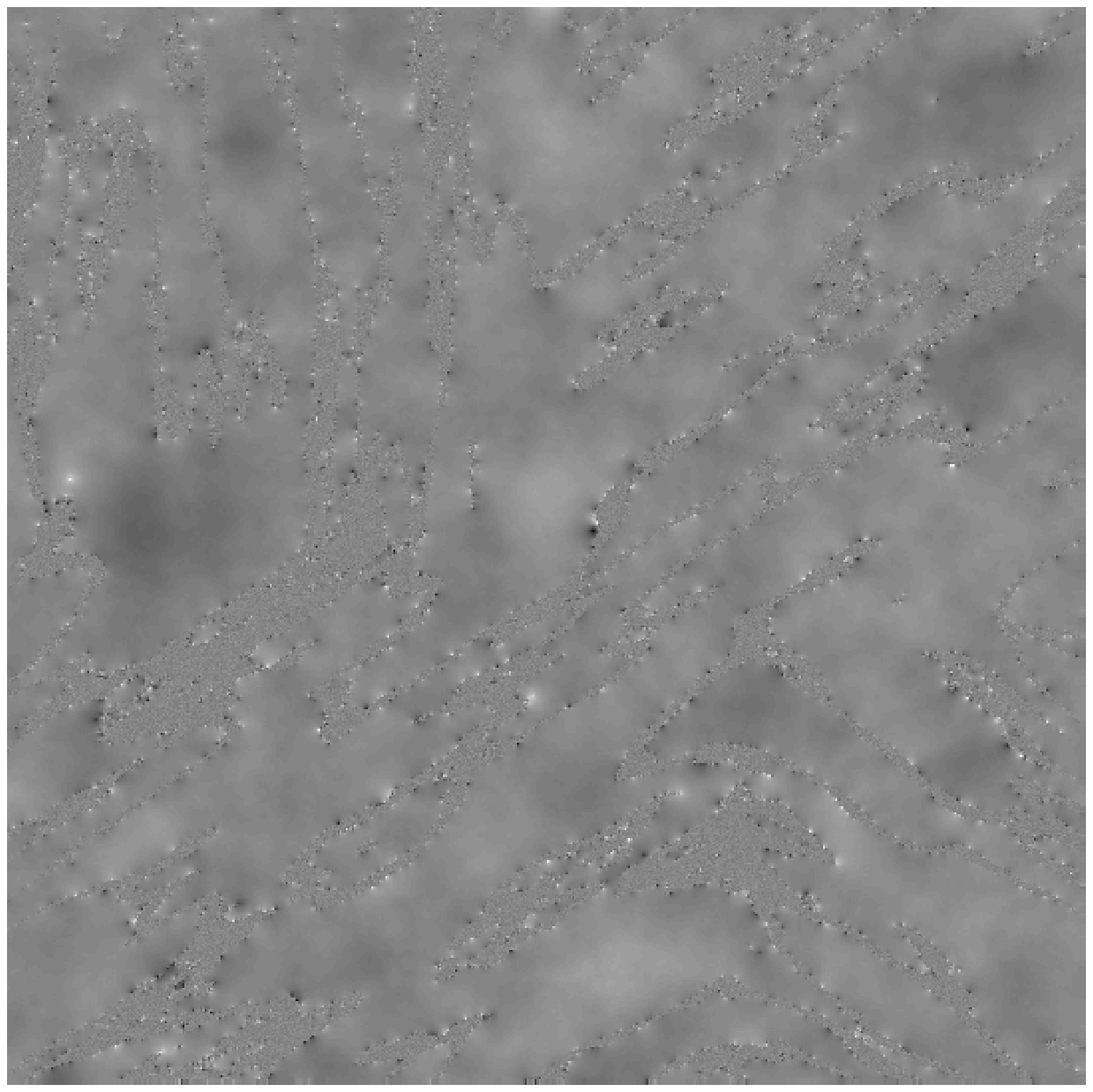}};

\draw[->,thick] (foto1.south) -- (foto2.north)
    node[midway,fill=white] {$\ell(x)$};

\draw[->,thick] (foto1.south) -- (foto2c.north)
    node[midway,fill=white] {$\ell(x)$};

\draw[->,thick] (foto1f.south) -- (foto2f.north)
    node[midway,fill=white] {$\ell(x)$};
    
\draw[->,thick] (foto2f.south) -- (foto3f.north);

\draw[->,thick] (foto2.south) -- (foto3.north);

\draw[->,thick] (foto2c.south) -- (foto3c.north);

\draw[->,thick] (foto2c.south) -- (foto3d.north);

\end{tikzpicture}
  \caption{Non-stationary structures obtained by starting from a constant-parameter or inhomogeneous Mat\'ern field realisation (upper  panel), after which have been mapped to correlation length-scaling fields (middle). In the bottom panel, we have corresponding realisations with isotropic and anisotropic structures. This kind of structure can detect regions within which the behaviour of the random field is smooth, but the regions are distinct.} \label{fig:foto2}
 \end{figure}

\section{MCMC}
\label{section:MCMC}

In order to draw estimates from the posterior distribution, we will use a combination of  Gibbs sampling and Metropolis-within-Gibbs algorithms.
We summarise the algorithm as follows:
\begin{enumerate}
\item Initiate $v^{N,{(0)}}$ and $\ell^{N,(0)}$. 
\item For $k = 1\dots K$
\begin{enumerate}
	\item Update $v^{N,(k)}$ given fixed $\ell^{N,(k-1)}$ and draw $\eta\sim \mathcal{N}(0,I)$, and set 
	\begin{equation*}
		v^{N,(k)}  = \begin{pmatrix} \sigma^{-1} A \\ L(\ell^{N,(k-1}))\end{pmatrix}^\dag\left(\begin{pmatrix}\sigma^{-1}y \\ 0 \end{pmatrix} + \eta \right),
	\end{equation*}
	where $^\dag$ denotes the matrix pseudoinverse.
	\item Update $\ell^{N,(k)}$ using Metropolis-within-Gibbs  given fixed $v^{N,{(k)}}$. \\ For $n = 1\dots N$
	\begin{enumerate}
		\item Draw a candidate sample $\tilde \ell^{N}_n$ from a proposal distribution $Q(\cdot | \ell^{N,(k-1)}_n)$ 
		\item Denote $\ell^{N,(k)}_{j\neq n}:=(\ell^{N,(k)}_1,\dots, \ell^{N,(k)}_{n-1},\ell^{N,(k-1)}_{n+1},\dots, \ell^{N,(k-1)}_{N})^T$, and 
			 accept with probability
			\begin{equation*}
				p_n = \min \left(1, \frac{D\left(\tilde \ell^{N}_n \vert v^{N,{(k)}}, \ell^{N,(k)}_{j\neq n}\right)Q\left(\ell_n^{N,(k-1)} \vert \tilde \ell^{N}_n\right)}{D\left(\ell^{N,(k-1)}_n \vert v^{N,{(k)}}, \ell^{N,(k)}_{j\neq n}\right)Q\left(  \tilde \ell^{N}_n  \vert  \ell_n^{N,(k-1)}\right)}\right)
			\end{equation*}
		\item If accepted, we set $\ell^{N,(k)}_n=\tilde \ell^{N}_n$. Otherwise we set $\ell^{N,(k)}_n=\ell^{N,(k-1)}_n$.
		\item Set $n \leftarrow n + 1$, and repeat from step (i) until $n=N$
	\end{enumerate}
	\item Set $k \leftarrow k + 1$, and repeat from step (a) until the desired sample size $K$ is reached.
\end{enumerate}
\end{enumerate}
Metropolis-within-Gibbs is explained for example in \cite{Dunlop2016,Markkanen2016}. 
The latter uses the term single-component Metropolis-Hastings to emphasise the fact that we sample every single component separately with the Metropolis-Hasting algorithm.
We aim at acceptance ratio between 25-50 per cent, which is obtained by tuning the random walk proposal process.

In computing acceptance probability, it is a common practice, due to numerical reasons, to take logarithms instead of using ratios. 
For example, in the case of the Gaussian hyperprior,  the logarithm of the posterior is
\begin{equation} \label{eqn:alpha_n}
\begin{split}
\log\left(D(v^N,\ell^N \vert y)\right)
=&  R- \frac{1}{2}(u^N)^T (\tilde C^N)^{-1}u^N + \log(\vert L(\ell^N)\vert))-\dots \\ & \frac{1}{2}(v^N)^TL(\ell^N)^TL(\ell^N)v^N - \frac{1}{2}(y-Av^N)^T\Sigma^{-1}(y-Av^N), 
\end{split}
\end{equation}
where $R$ is some constant, which we may omit from the analysis.
We note that the normalisation constant computation, i.e.\ logarithmic determinant $\log(\vert L(\ell^N)\vert))$, is a numerically unstable and computationally expensive operation, especially in higher dimensions.
We need to compute altogether $N\times K$ logarithmic determinants in our estimation algorithm, hence we may wish to minimise the log-determinant computation time.

We note that in the Metropolis-Hastings part, when updating $\ell^{N,k}_n$, we are actually computing ratio of the proposed and old normalisation constant.
Let us denote the proposed and old covariances  by $C_{\mathrm{old}}$ and  $C_{\mathrm{prop}}$, respectively.
Then we should  actually calculate the ratio, as originally in Equation \eqref{eqn:alpha_n}, and  not take logarithms.
Then by simple algebra we have
\begin{equation*}
\frac{\sqrt{\vert C_{\mathrm{old}}\vert }}{\sqrt{ \vert C_{\mathrm{prop}} \vert }} = \frac{\sqrt{\vert (L_{\mathrm{old}}^TL_{\mathrm{old}})^{-1}\vert}}{\sqrt{\vert (L_{\mathrm{prop}}^TL_{\mathrm{prop}})^{-1}\vert }} = \frac{\vert L_{\mathrm{prop}}\vert }{\vert L_{\mathrm{old}}\vert } = \vert L_{\mathrm{prop}}L_{\mathrm{old}}^{-1}\vert .
\end{equation*}
Now, we note that as we update only one row at the time, the matrix-inverse-matrix-product is of the form:
\begin{equation*}
L_{\mathrm{prop}}L_{\mathrm{old}}^{-1} = \begin{pmatrix} I & 0 \\ \times & \times \\ 0 & I   \end{pmatrix}.
\end{equation*}
Hence, the product is diagonal, except for the $n^{(\mathrm{th})}$ updated row.
This means that we simply need to compute diagonal alue of the updated row, i.e.\ only one value.
It would seem that we would need to invert whole matrix $L_{\mathrm{old}}^{-1}$. 
However, we note that as $L$-matrices are sparse, so we will have only a limited amount of non-zero values. Consider e.g. the one-dimensional case. 
We could have e.g.
\begin{equation*}
\begin{pmatrix} 0 & \dots & 0 & a & b & a & 0 & \dots & 0\end{pmatrix} L_{\mathrm{old}}^{-1} = \begin{pmatrix} 0 & \dots & 0 & \times & \times & \times & 0 & \dots & 0\end{pmatrix},
\end{equation*}
where $a$  and $b$ are constants derived from the  approximations in Equation \eqref{eqn:1Dappr}.
By simple matrix operations and removing the 'zeroes' from the matrix equation, we  can rewrite this as 
\begin{equation*}
\begin{pmatrix} a & b & a \end{pmatrix}  = \begin{pmatrix}  \times & \times & \times \end{pmatrix}\tilde L_{\mathrm{old}},
\end{equation*}
where $\tilde L_{\mathrm{old}}$ is a $3\times 3$ matrix.
Hence, the computation of the determinant is then simply inverting a $3\times 3$ matrix and making one matrix-vector multiplication. In two-dimensional problems, we need to invert $5\times 5$ matrix.

\section{Numerical examples}
\label{sec:numerics}

Now, we shall apply the developed methodology to one-dimensional interpolation and numerical differentiation, and, to two-dimensional interpolation.

\subsection{One-dimensional interpolation}

We model  discrete noise-perturbed observations of a continuous object $v$ as
\begin{equation*}
y(jh) = v(j'h') + e(jh),
\end{equation*}
where $e(jh)$ is zero-mean white noise with known variance, and $j\in \mathbb{J}\subset\mathbb{Z}$ is the measurement mesh and $j'\in  \mathbb{J}'\subset\mathbb{Z}$ is the mesh of the discretised unknown $v^N$.
The discretisation steps $h,h'>0$.
This model, can be rewritten in the form given in Equation \eqref{eqn:discreteobservation}, i.e.\ as $y=Av^N+e$.
Hence we can write the whole posterior distribution with the help of hypermodels as discussed earlier.

Let us consider $v$ consisting of a $C^\infty$ mollifier function and  two boxcar functions 
\begin{equation}\label{eqn:interpolationunknown}
v(x) = 
\begin{cases}
\exp\left(4-\frac{25}{x(5-x)}\right), &  x\in(0,5) \\
1, & x\in[7,8] \\
-1, & (8,9] \\
0, & \text{otherwise}.
 \end{cases}
\end{equation}
This function has  smooth parts,  edges, and it is also piecewise constant for $x\in [5,10]$.
In Figure \ref{fig:CM_est_1}, we have simulation results with three different prior and hypermodels:
\begin{enumerate}
	\item Constant-parameter Mat\'ern prior, i.e.\ this model does not have hyperprior.
	\item Hypermodel with Cauchy walk $u$.
	\item Hypermodel with stationary Gaussian zero-mean process $u$ with exponential covariance
\end{enumerate}
The domain for both the measurements $y$ and unknown $v^N$ is $[0,10]$.
The measurement mesh is $j = \{0,1,\dots,80\}$, $h=1/8$ and the unknown mesh $j'=\{0,1,\dots,160\}$, $h'=1/16$.
Zero-mean measurement noise has standard deviation $\sigma=0.1$.
Mat\'ern prior has periodic boundary conditions.

With the constant-parameter Mat\'ern prior, we have plotted estimates with a long length-scaling (D),  length-scaling minimising maximum absolute error (G), and length-scaling minimising root mean square error (J).
These estimates capture the smoothness or edges, but not both at the same time.
With the Cauchy and Gaussian hypermodels, the algorithm finds short  and long length-scaling $\ell^N$ (subfigures (B) and (C)). 
Also, the corresponding $v^N$ estimates in (E) and (F) show that we can reconstruct both smooth and edge-preserving parts.
In subfigures (H) and (I), we have plotted $v^N$ on the measurement mesh, and in subfigures (K) and (L) in the interpolated points, i.e.\ between the measurement grid points.
This shows that the interpolated estimates are behaving as expected.

In order to relate this study to Paciorek's 2003 study  \cite{Paciorek2003}, we note that the (D,G,J) subfigures correspond to Paciorek's so-called single-layer model.
For deep learning, one should build a series of hyperpriors over hyperpriors. 
Here, instead, we  have a two-layer model, and we may note that it captures different properties with very good precision. 
Hence, the question remains whether two layers is actually often enough in deep Gaussian processes, and what is the actual gain using deep layers. 
We will leave this question open, and hope to address that in subsequent studies.

\begin{figure}[htp]
\begin{center}
  \subcaptionbox{81 noisy measurements}{\includegraphics[width=0.32\textwidth]{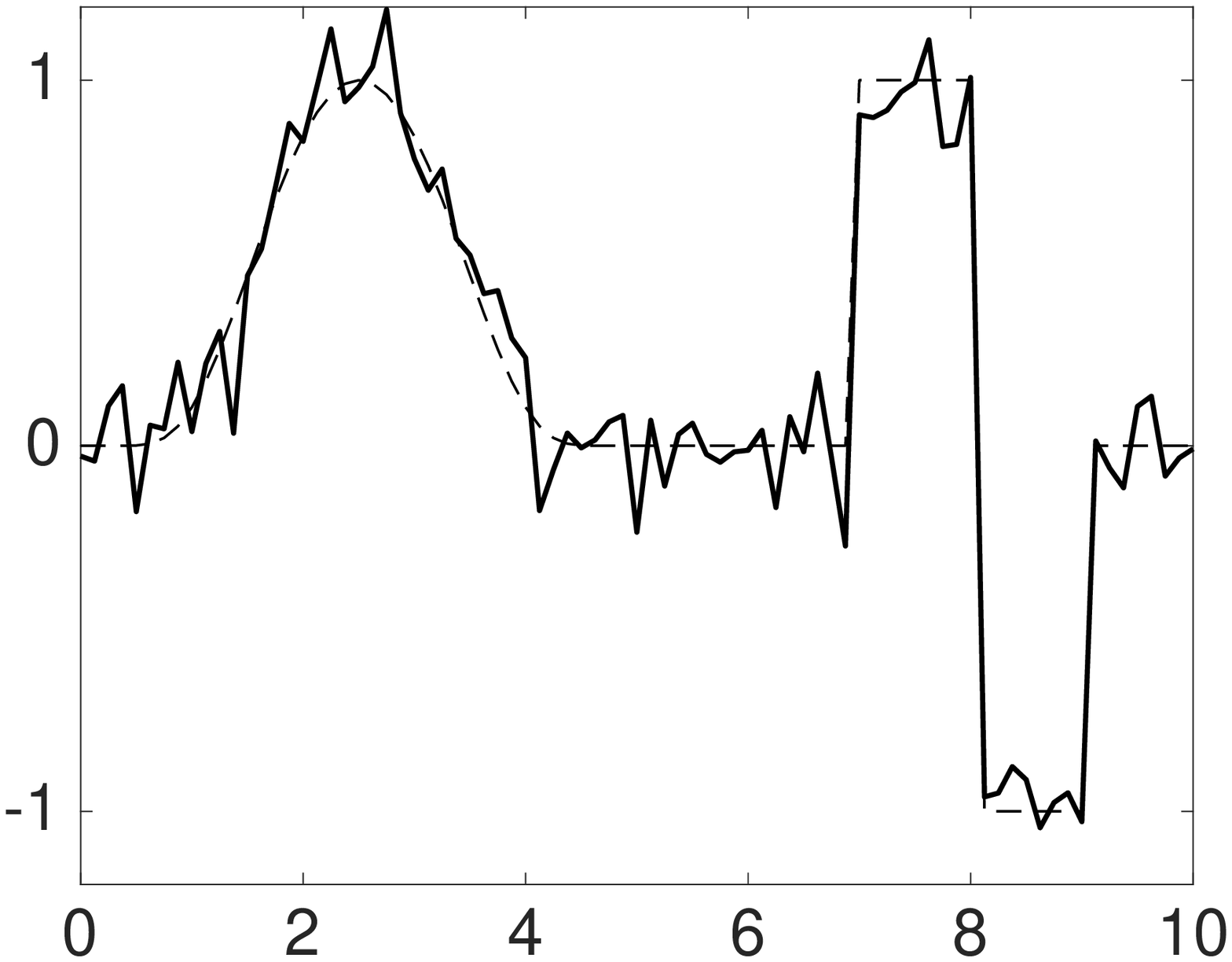}}
  \subcaptionbox{Cauchy hypermodel estimate  $\ell^N$ with $\sigma$ error bars}{\includegraphics[width=0.32\textwidth]{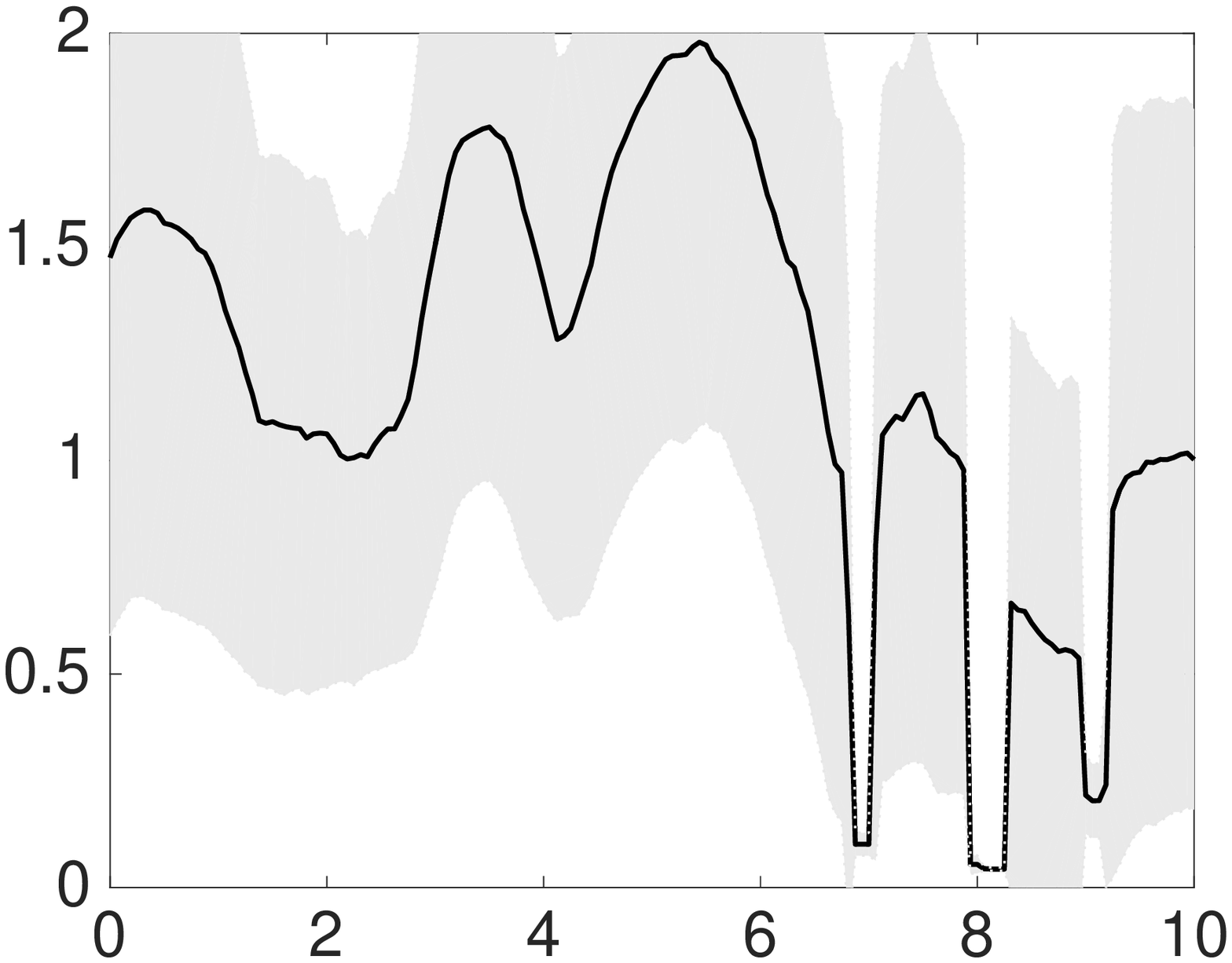}} 
  \subcaptionbox{Gaussian hypermodel estimate  $\ell^N$ with $\sigma$ error bars}{\includegraphics[width=0.32\textwidth]{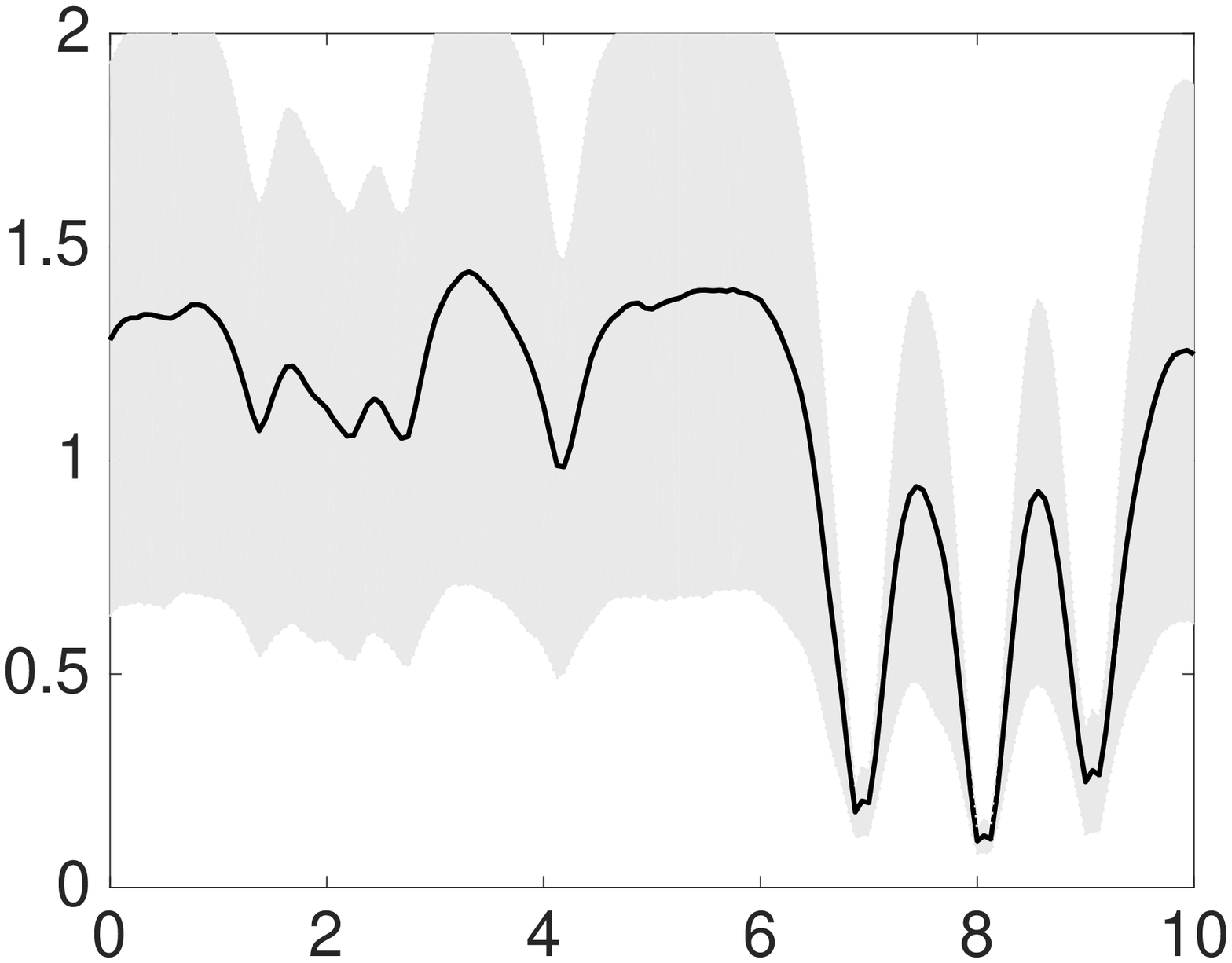}} \\

  \subcaptionbox{Estimate with Gaussian prior  $\ell=2$}{\includegraphics[width=0.32\textwidth]{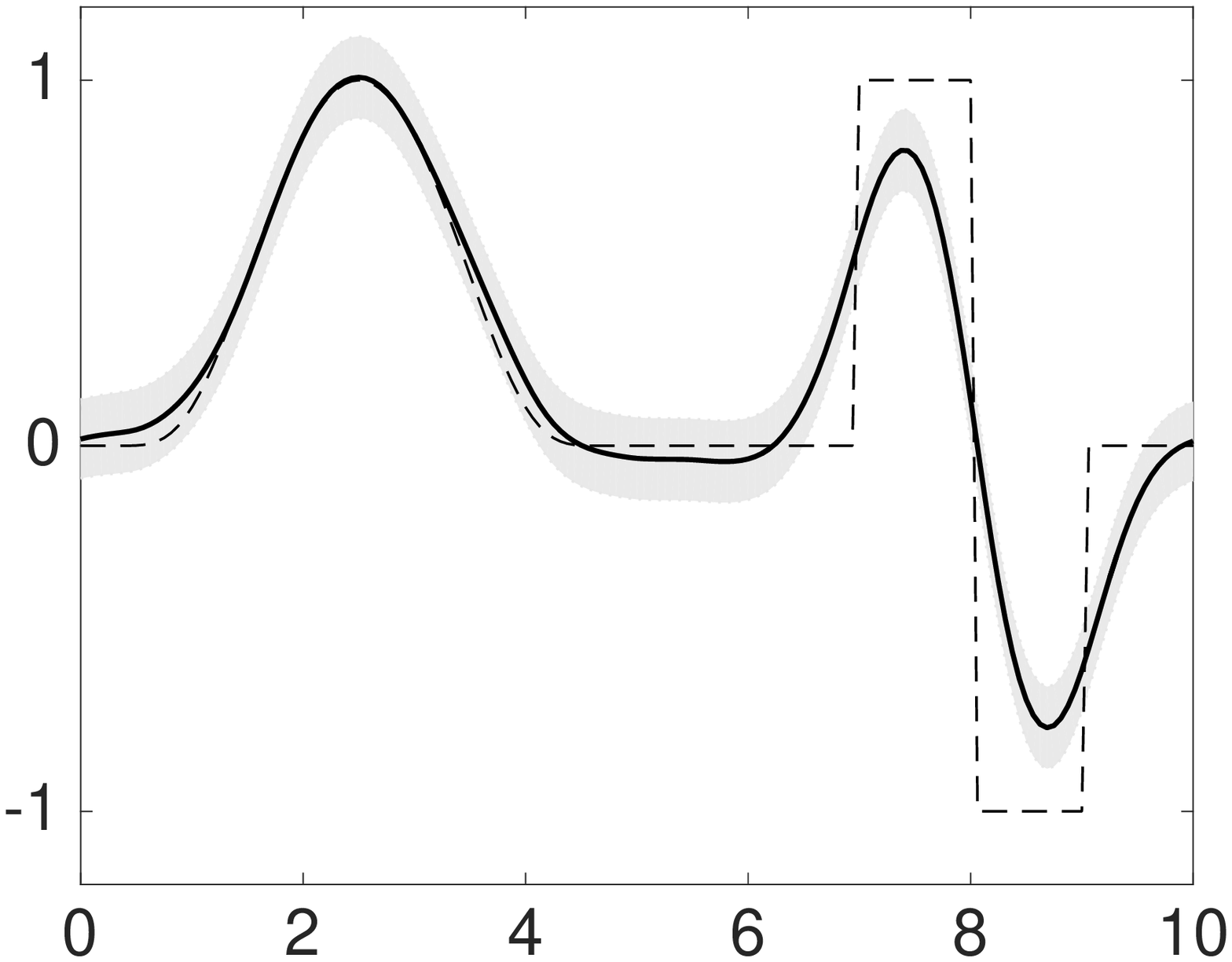}}
  \subcaptionbox{Cauchy hypermodel estimate $v^N$ with $3\sigma$ error bars}{\includegraphics[width=0.32\textwidth]{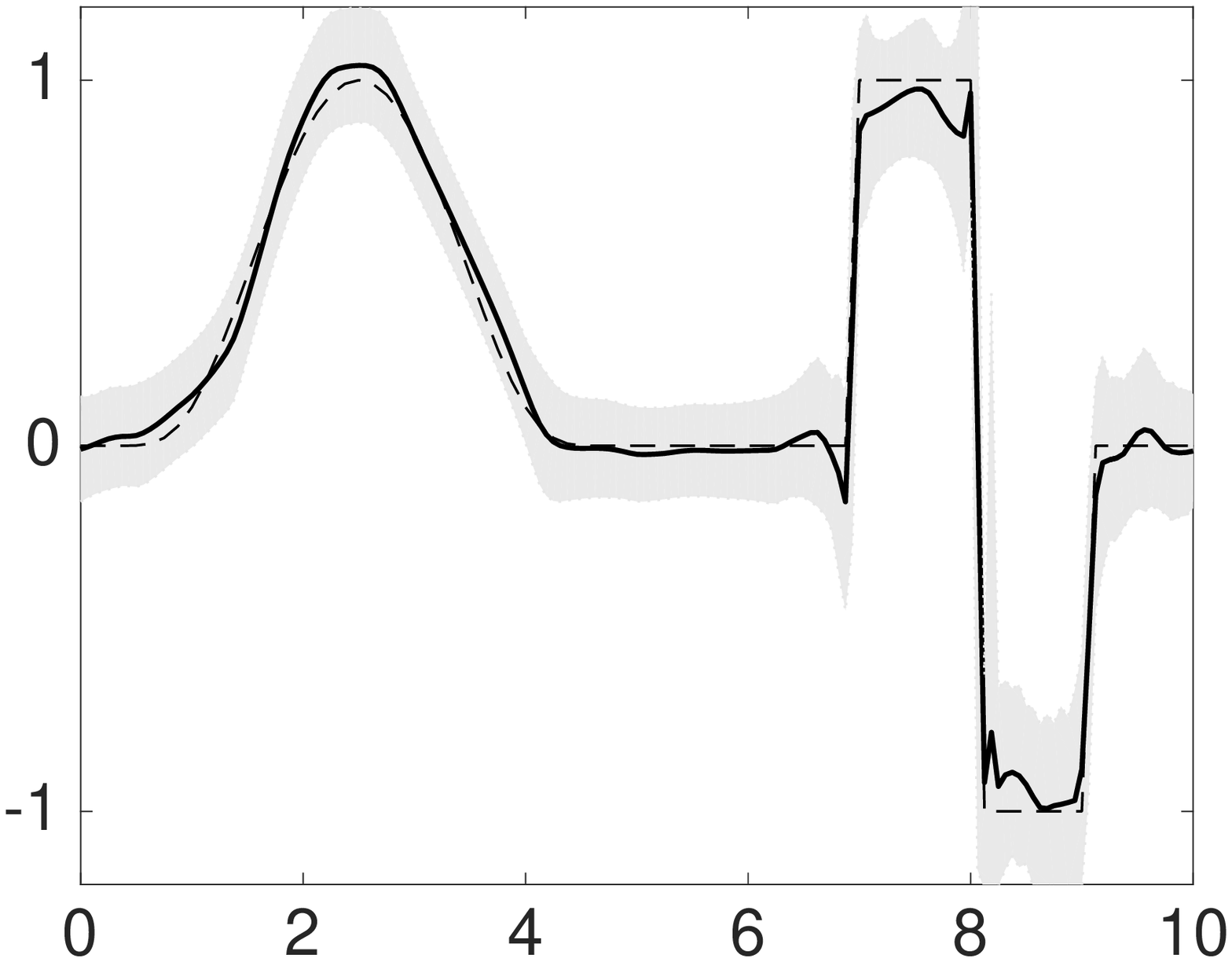}} 
  \subcaptionbox{Gaussian hypermodel estimate $v^N$ with $3\sigma$ error bars}{\includegraphics[width=0.32\textwidth]{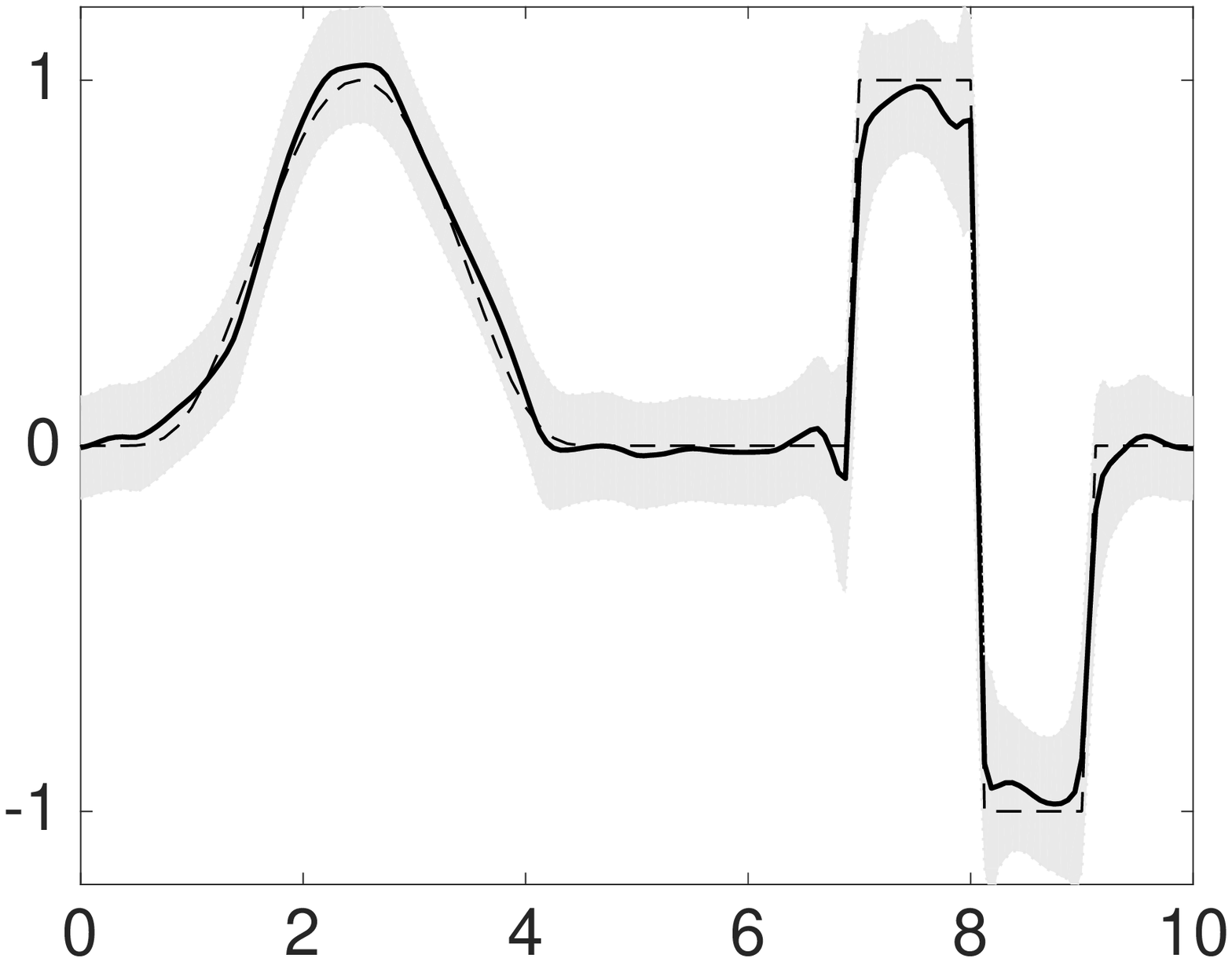}} \\
  
  \subcaptionbox{Minimum MAE estimate with Gaussian prior}{\includegraphics[width=0.32\textwidth]{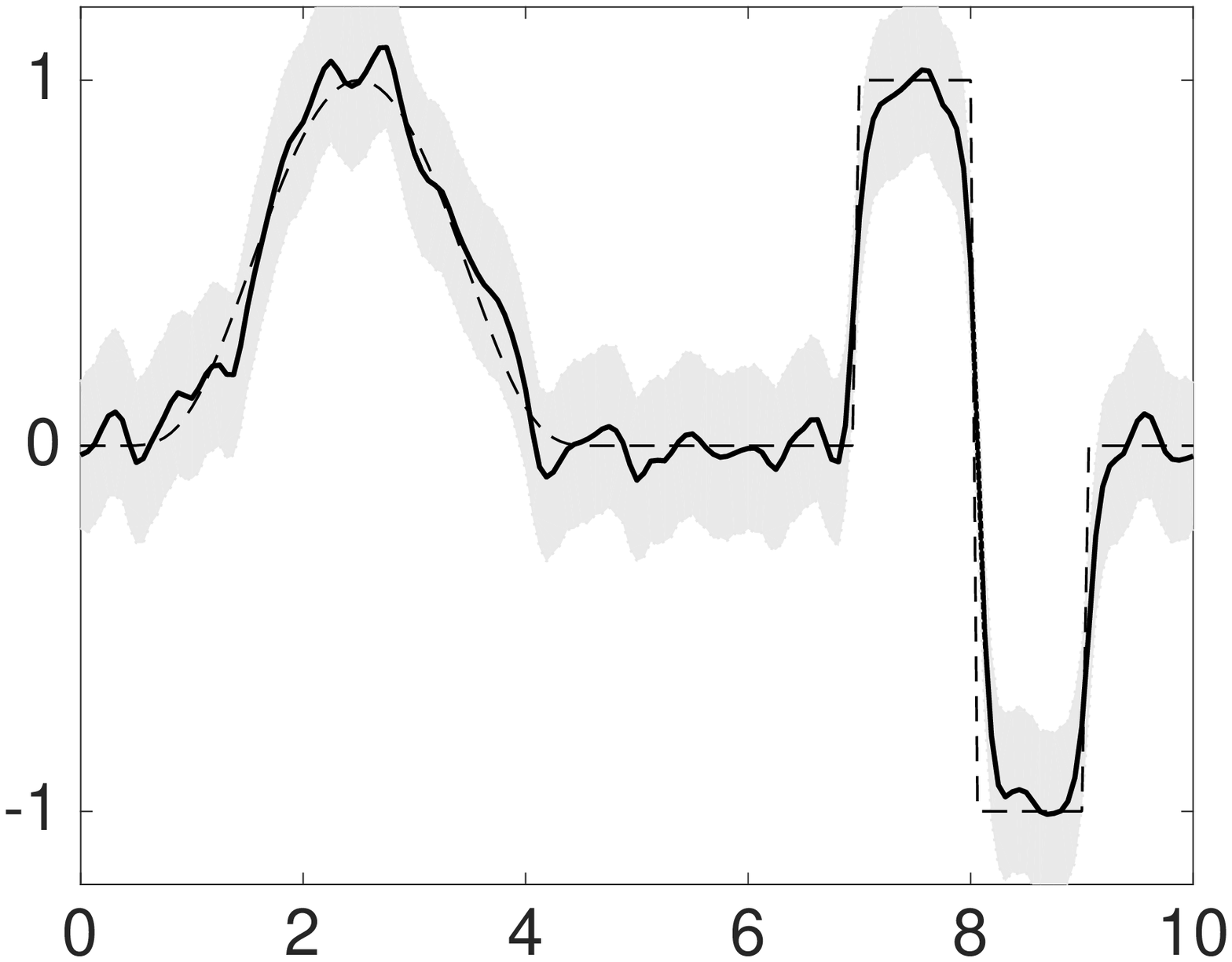}}
  \subcaptionbox{Cauchy hypermodel estimate $v^N$ on measurement grid}{\includegraphics[width=0.32\textwidth]{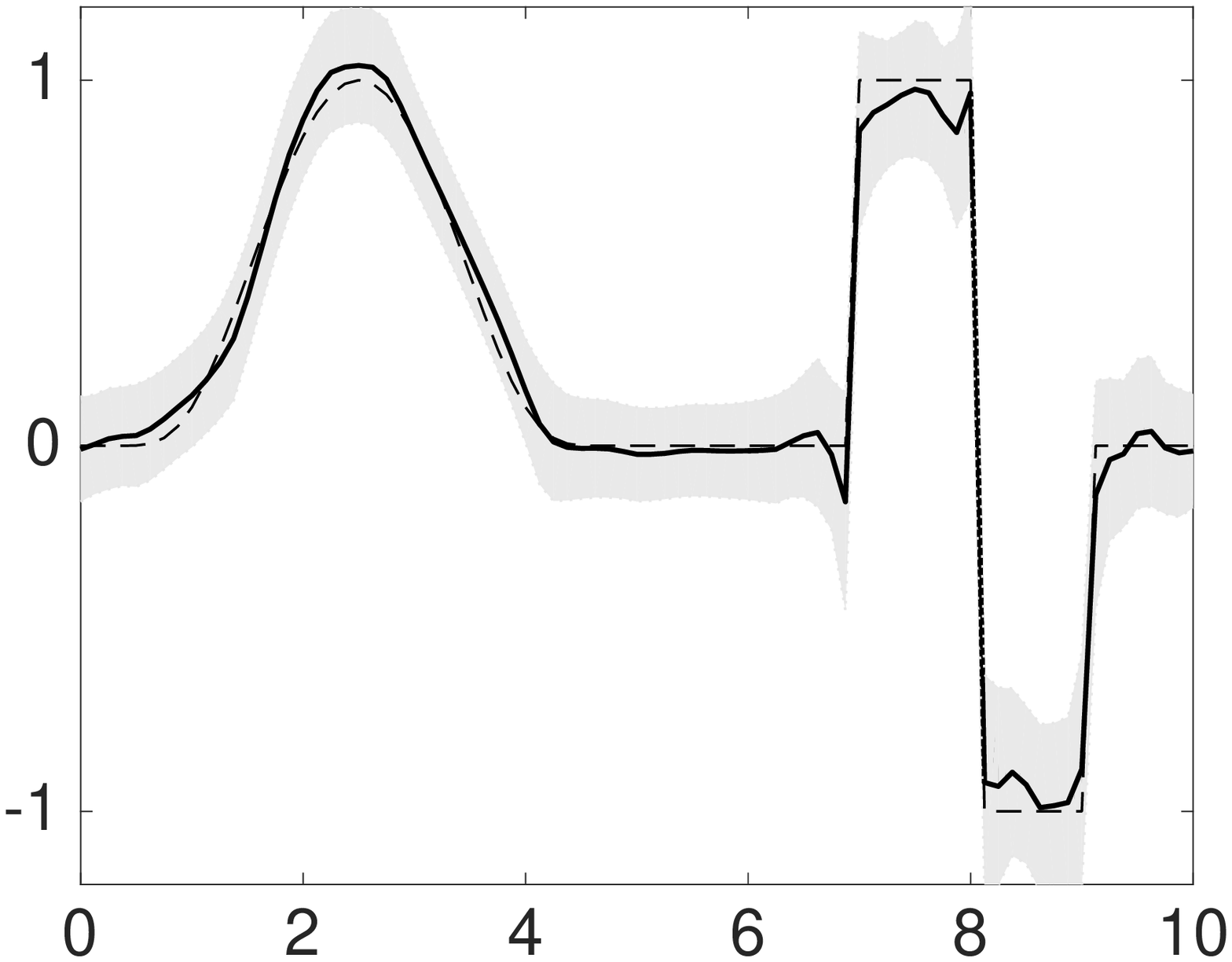}}
  \subcaptionbox{Gaussian hypermodel estimate  $v^N$ on measurement grid}{\includegraphics[width=0.32\textwidth]{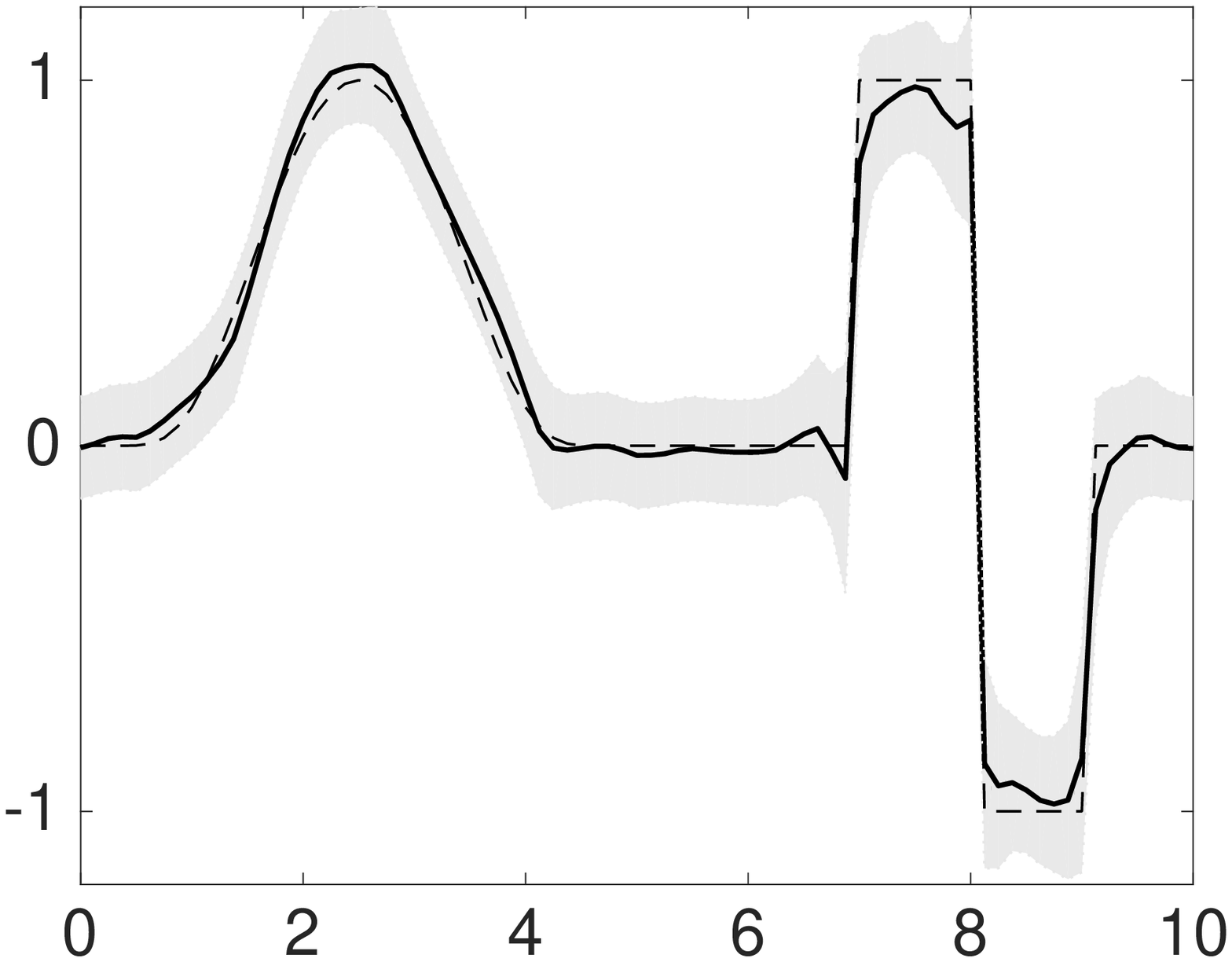}} \\

  \subcaptionbox{Minimum RMSE estimate with Gaussian prior}{\includegraphics[width=0.32\textwidth]{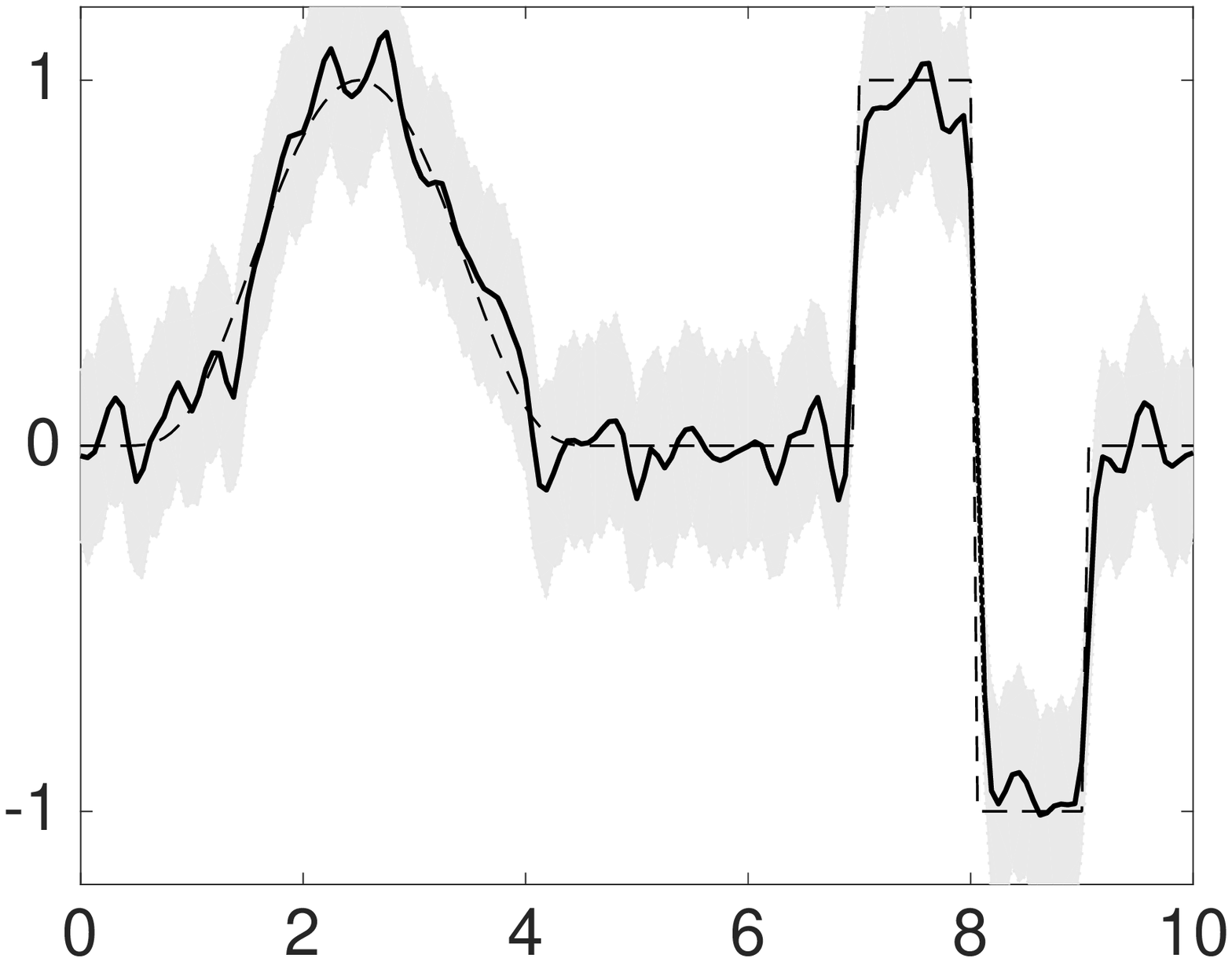}} 
  \subcaptionbox{Cauchy hypermodel estimate $v^N$ between measurement grid}{\includegraphics[width=0.32\textwidth]{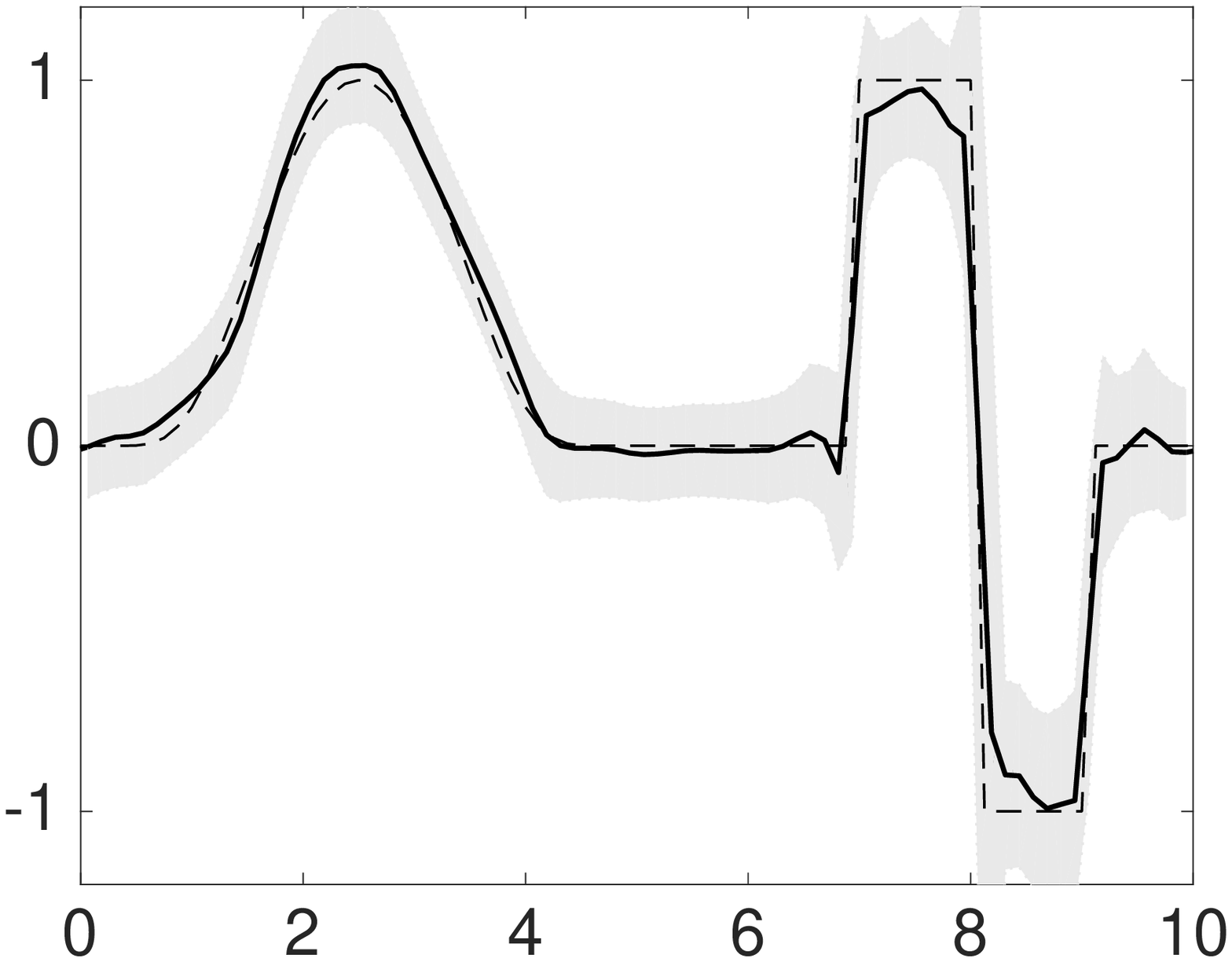}}
  \subcaptionbox{Gaussian hypermodel estimate  $v^N$ between measurement grid}{\includegraphics[width=0.32\textwidth]{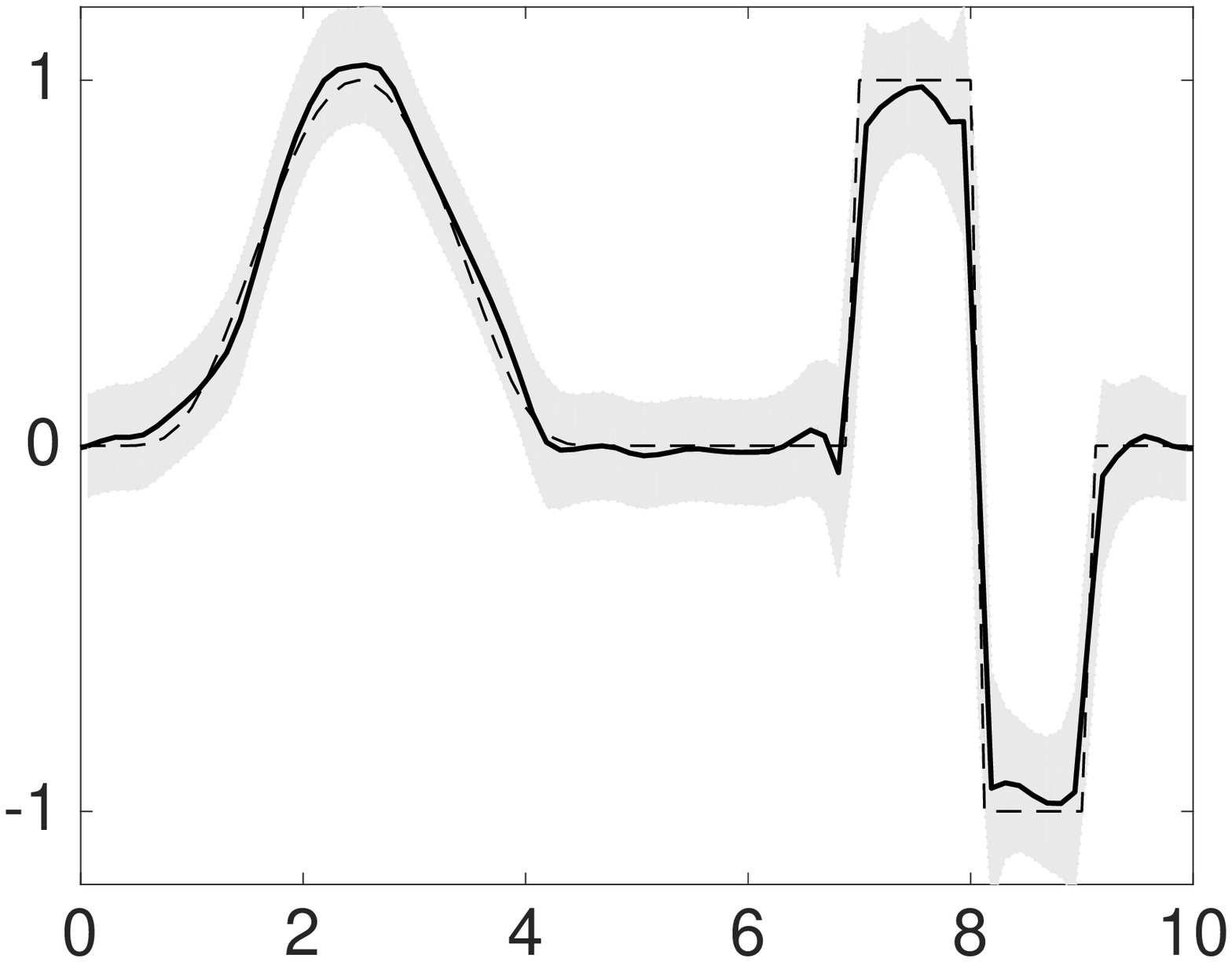}}

  \caption{Top panel: 81 noisy measurements and estimated $\ell^N$ ($N=161$) with Cauchy noise (B) and Gaussian hyperprior (C). (D,G,J) are conditional mean estimates of $v^N$ ($N=161$) with long length-scaling (D), $\ell^N$ minimising MAE (G), $\ell^N$ minimising RMSE. (E,H,K) and (F,I,L) are CM-estimates of $v^N$ on different meshes with Cauchy hypermodel and Gaussian hypermodels, respectively. 
  } \label{fig:CM_est_1}
  \end{center}
\end{figure}

In Figures \ref{fig:CM_est_discretisation_invariance} and \ref{fig:CM_est_GMRF_discretisation_invariance2}, we study the behaviour of the $\ell^N$ and $v^N$, when $N$ changes, i.e.\ we numerically study discretisation-invariance.
We choose $N = 81,\ 161,\ 321$.
The Cauchy and Gaussian hypermodels, as well as the forward theory, are the same as in example in Figure \ref{fig:CM_est_1}.
As we have constructed the hypermodels based on continuous-parameter processes, we assume that the finite-dimensional estimates essentially look the same.
This was covered theoretically in Section \ref{sec:discretisation}.
This behaviour can be verified from all the $v^N$ estimates visually rather easily, but the $\ell^N$ are not as well behaving.
The reason is mostly due to too short chains, but already with the chains here, with $K=100,000$, the essential features are in practice rather similar.

We have also plotted the MCMC chains and cumulative means for 15$^{(\mathrm{th})}$ and 66$^{(\mathrm{th})}$ elements of $\ell^N$ and $v^N$ with  $N = 81$. %
The elements are chosen in such a way the the 15$^{(\mathrm{th})}$ element is on a smoothly varying part of the unknown, hence long length-scaling. 
Around 66$^{(\mathrm{th})}$ element, we expect to detect an edge, hence short length-scaling.
The chains show both good mixing and convergence to the values expected. 
%
%We have used $10^5$ as MCMC chain length with 
We use $50,000$ as burning period.
As can be seen from the figures, we could use much shorter burning periods and chains, but here our objective is simply to demonstrate the estimators, not optimisation of the chains.
Hence, we leave optimisation of MCMC chains for future studies.

\begin{figure}[htp]
\begin{center}
  \subcaptionbox{$\ell^N,\ N=81$}{\includegraphics[width=0.32\textwidth]{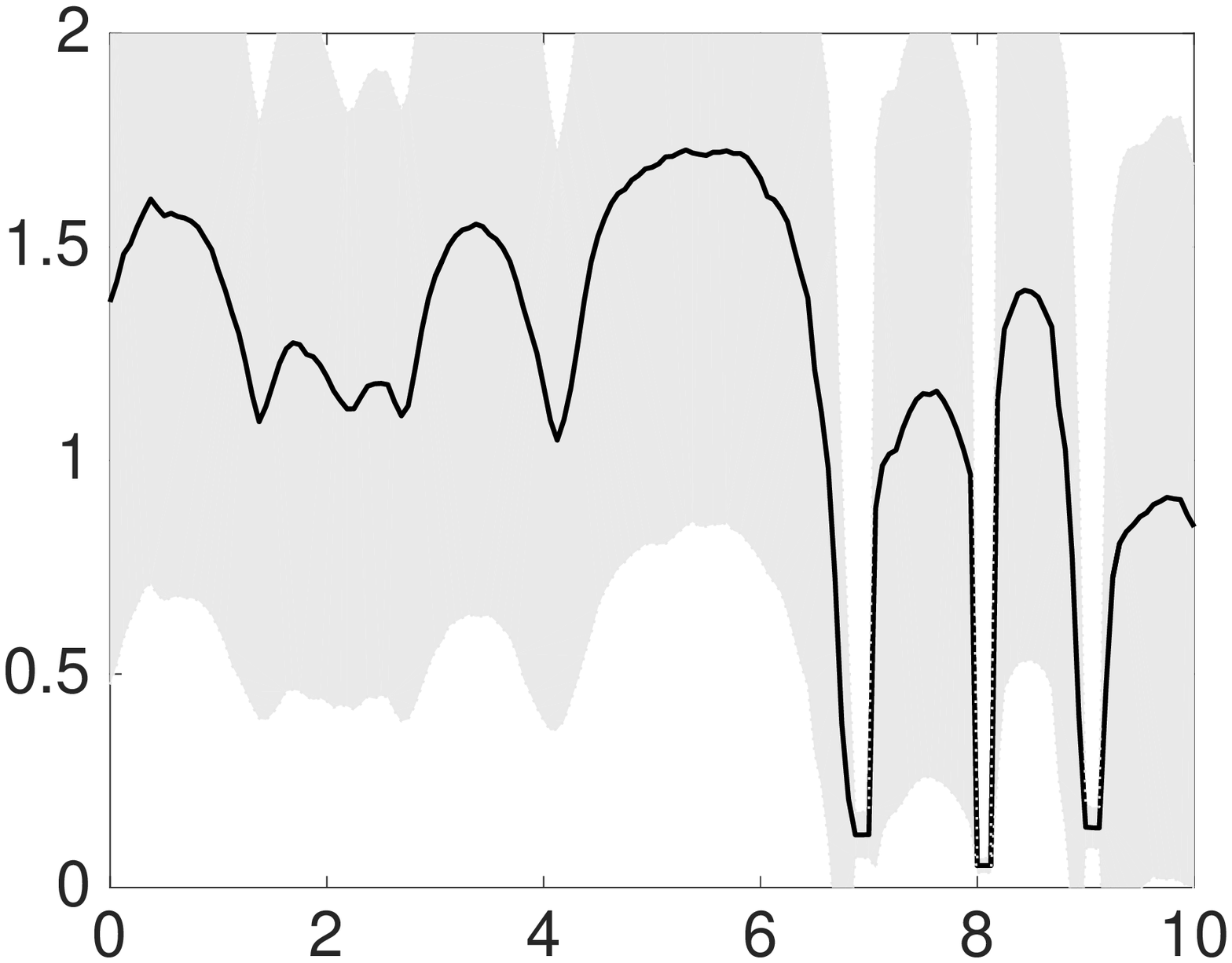}} 
  \subcaptionbox{$\ell^N,\ N=161$}{\includegraphics[width=0.32\textwidth]{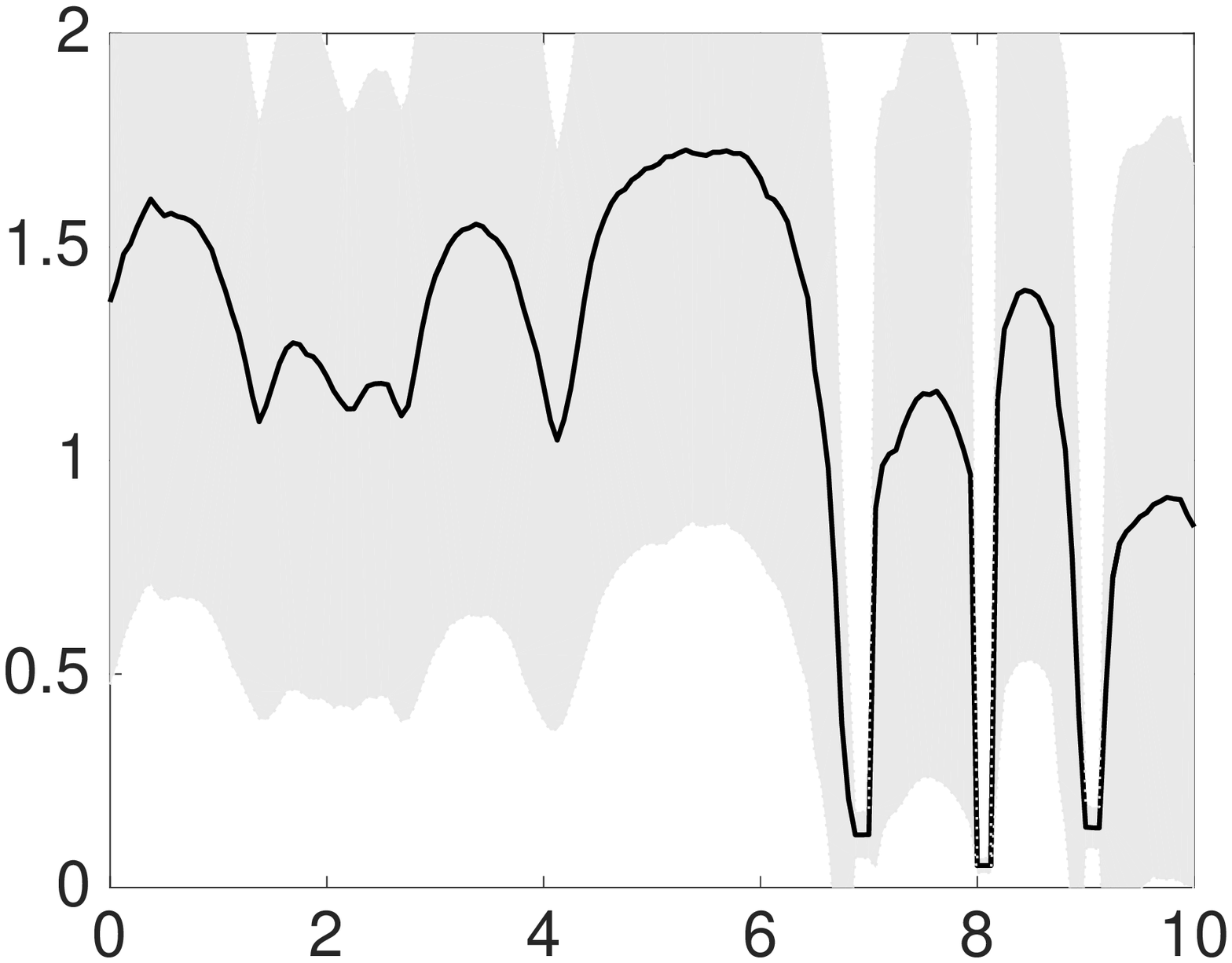}} 
  \subcaptionbox{$\ell^N,\ N=321$}{\includegraphics[width=0.32\textwidth]{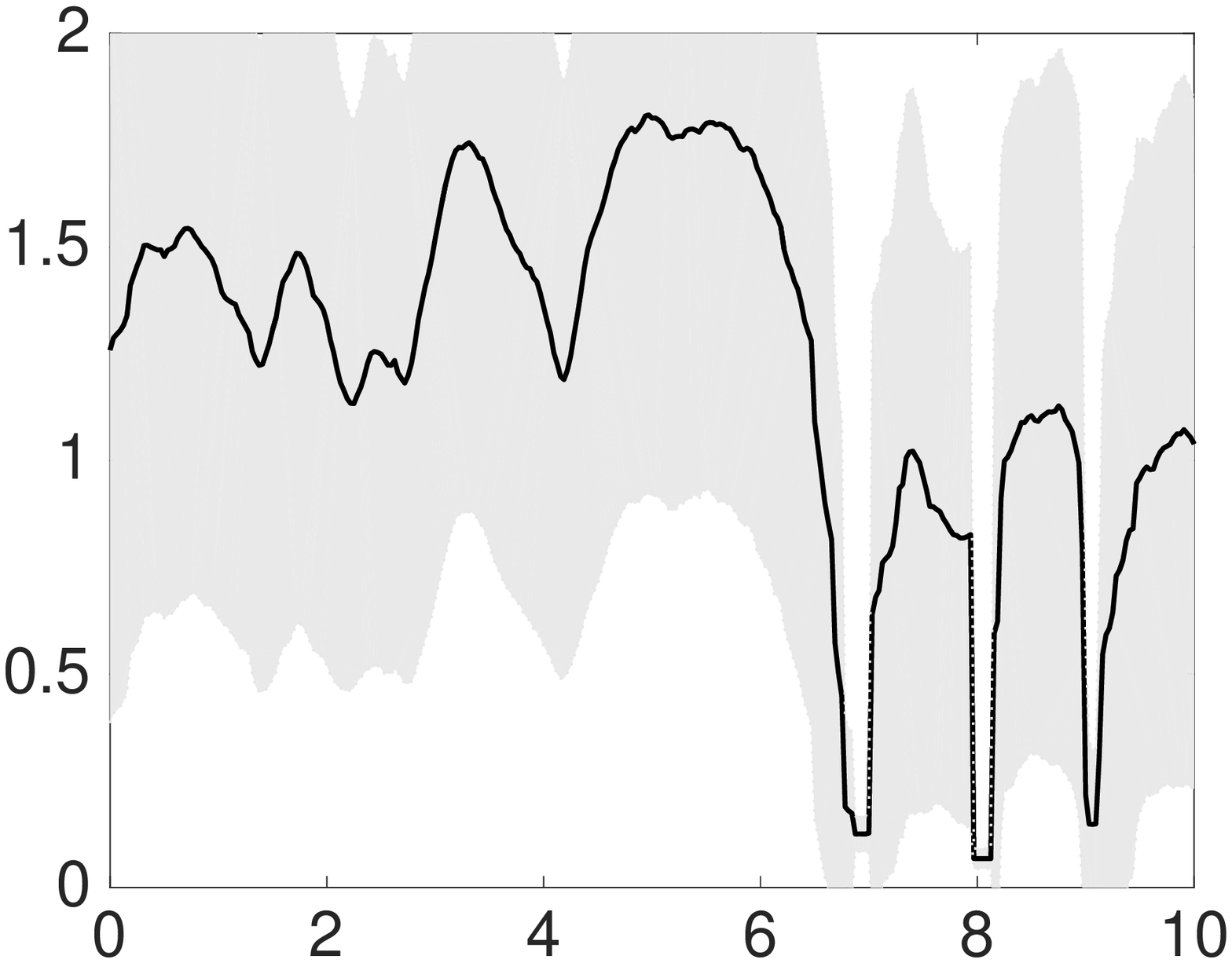}} \\
  
  \subcaptionbox{$v^N,\ N=81$}{\includegraphics[width=0.32\textwidth]{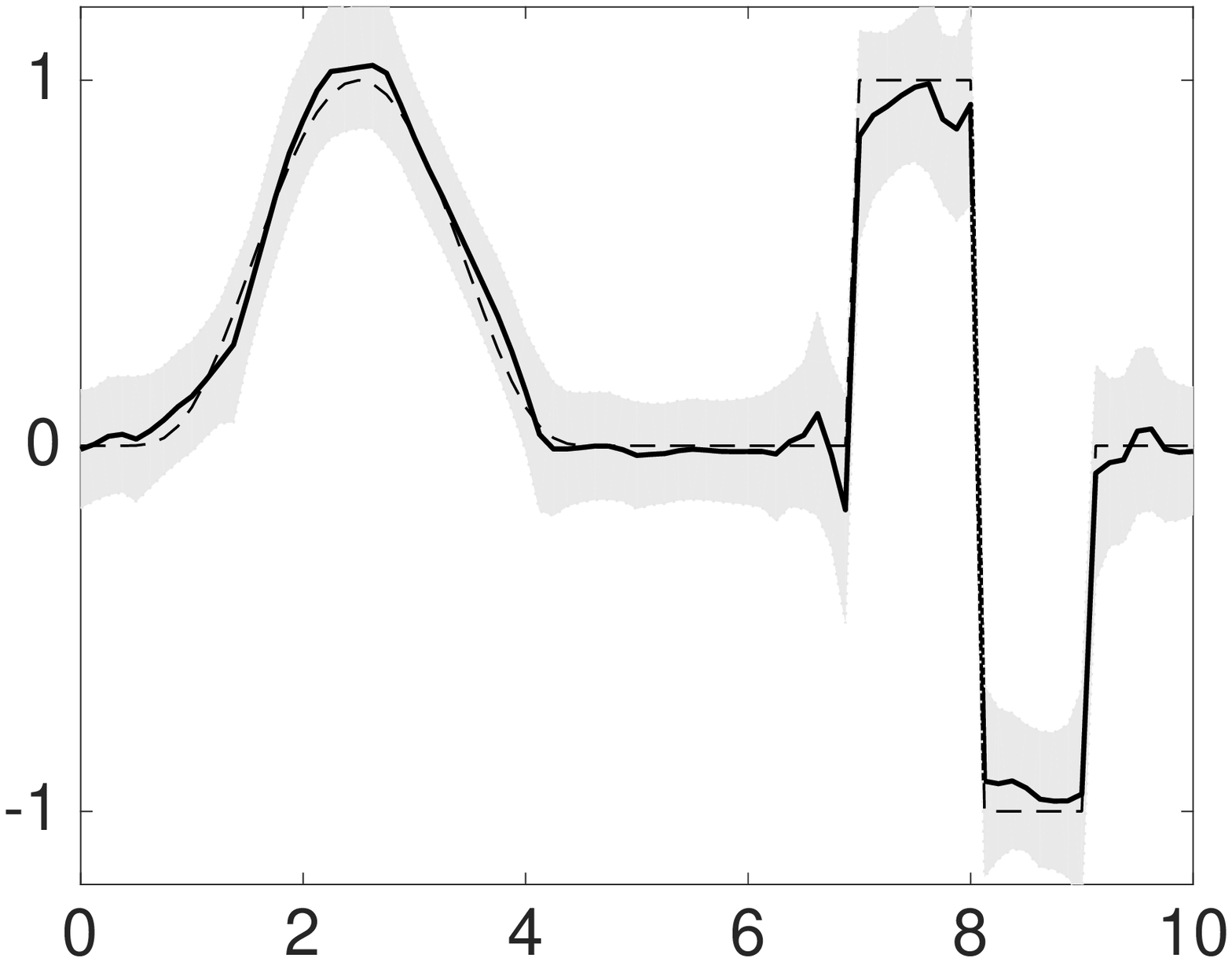}} 
  \subcaptionbox{$v^N,\ N=161$}{\includegraphics[width=0.32\textwidth]{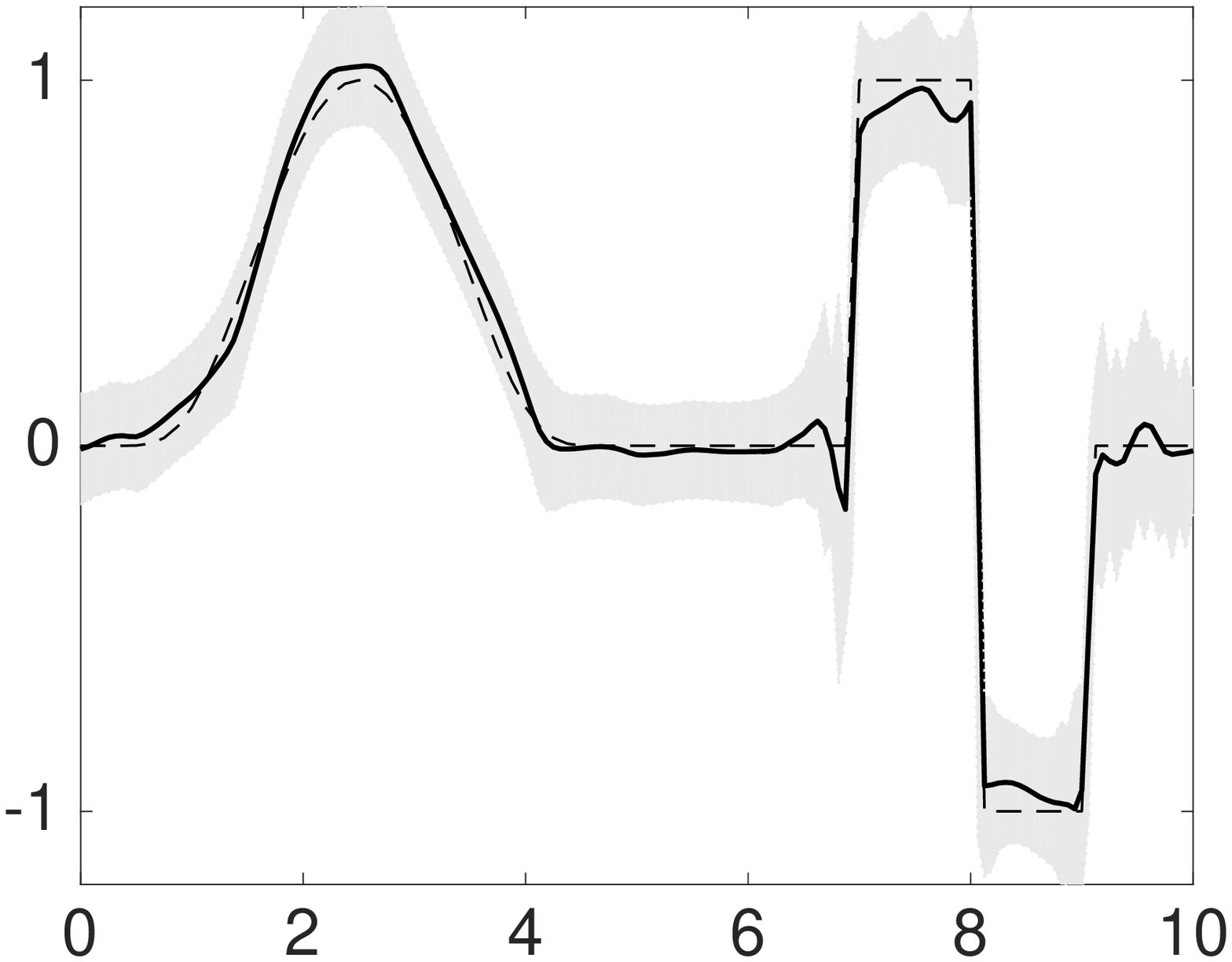}} 
  \subcaptionbox{$v^N,\ N=321$}{\includegraphics[width=0.32\textwidth]{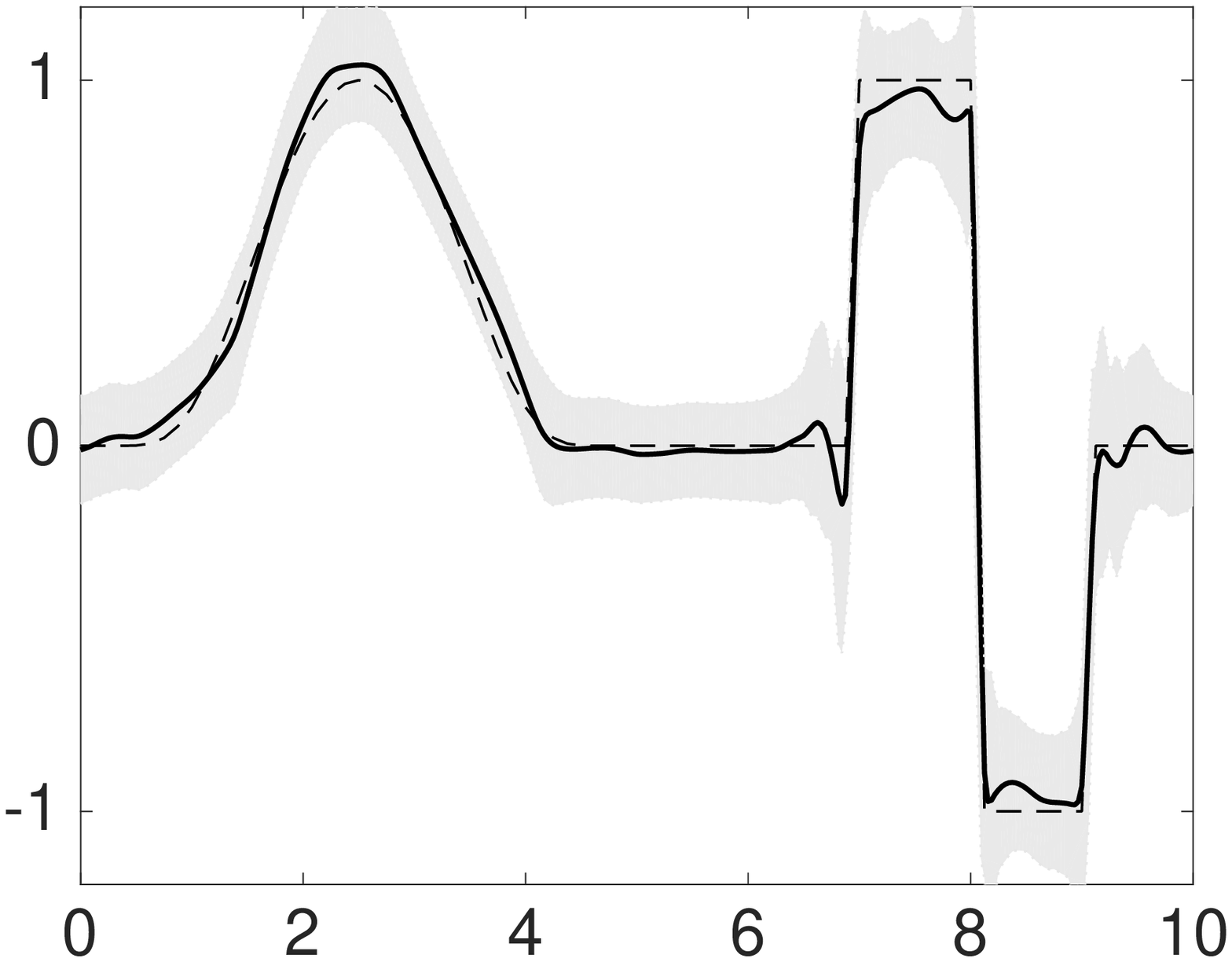}} \\
  
  \subcaptionbox{$\ell^N_{15},\ N=81$ chain and cumulative mean.}{\includegraphics[width=0.32\textwidth]{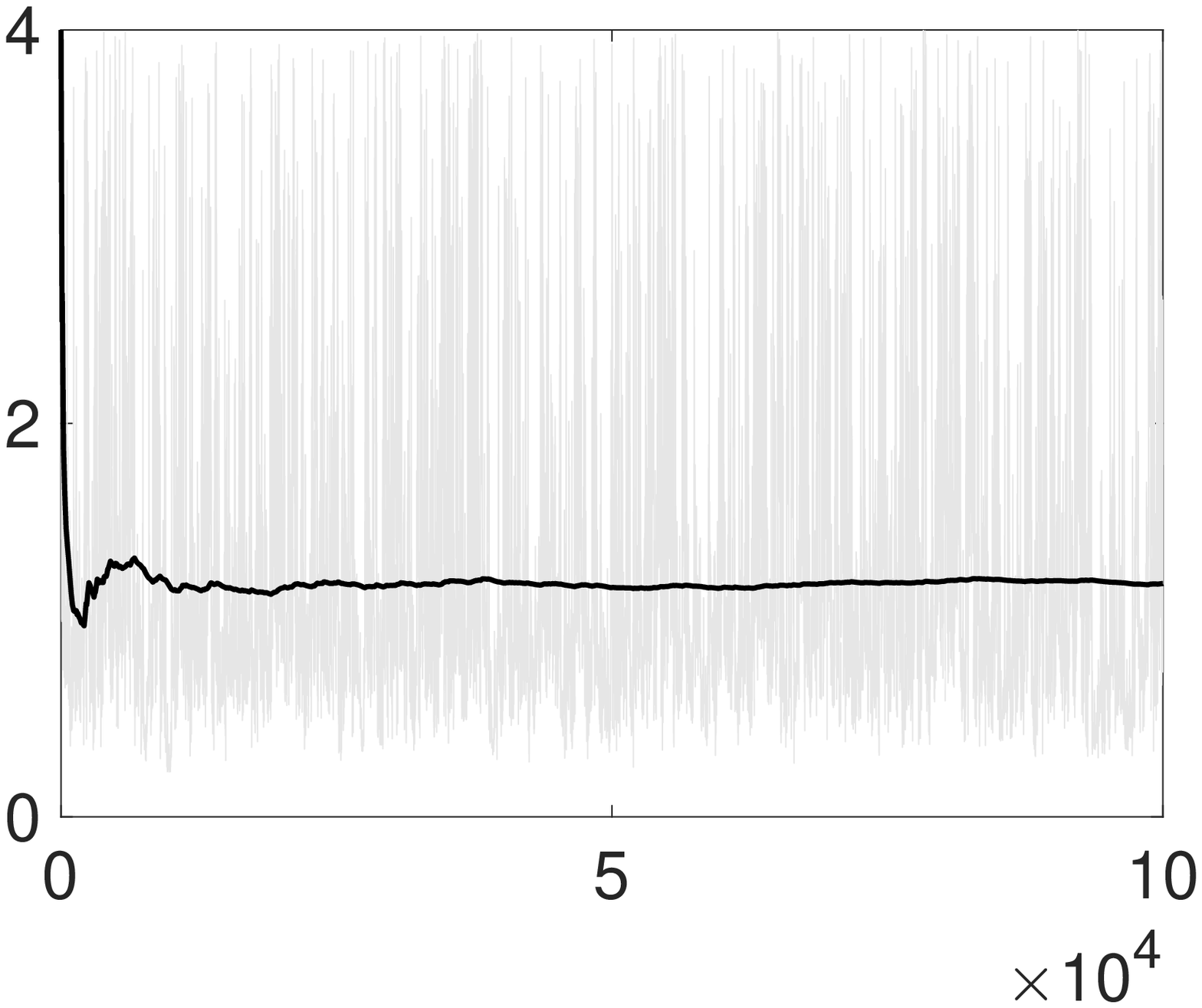}} 
  \subcaptionbox{$\ell^N_{66},\ N=81$ chain and cumulative mean.}{\includegraphics[width=0.32\textwidth]{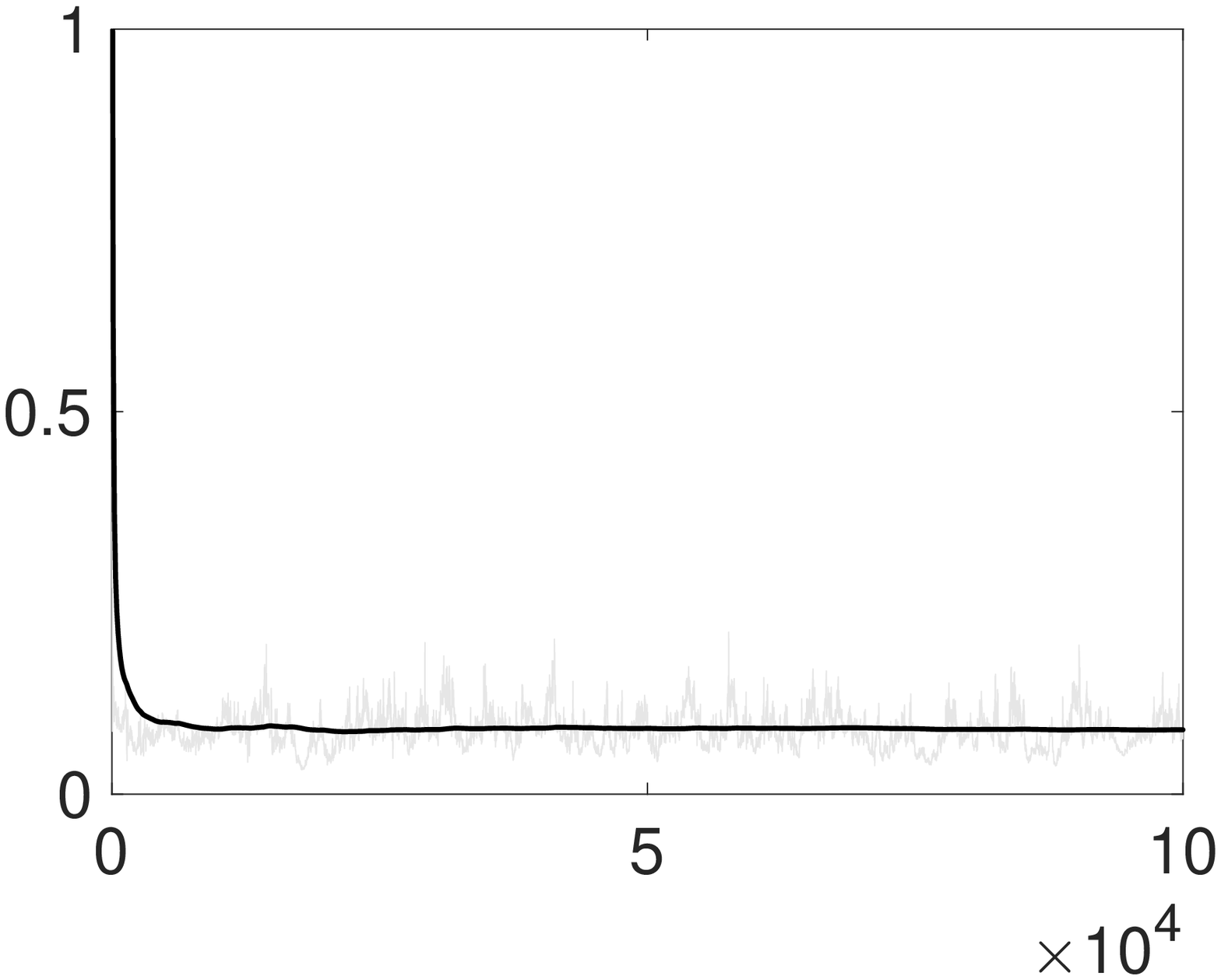}} \\
  \subcaptionbox{$v^N_{15},\ N=81$ chain and cumulative mean.}{\includegraphics[width=0.32\textwidth]{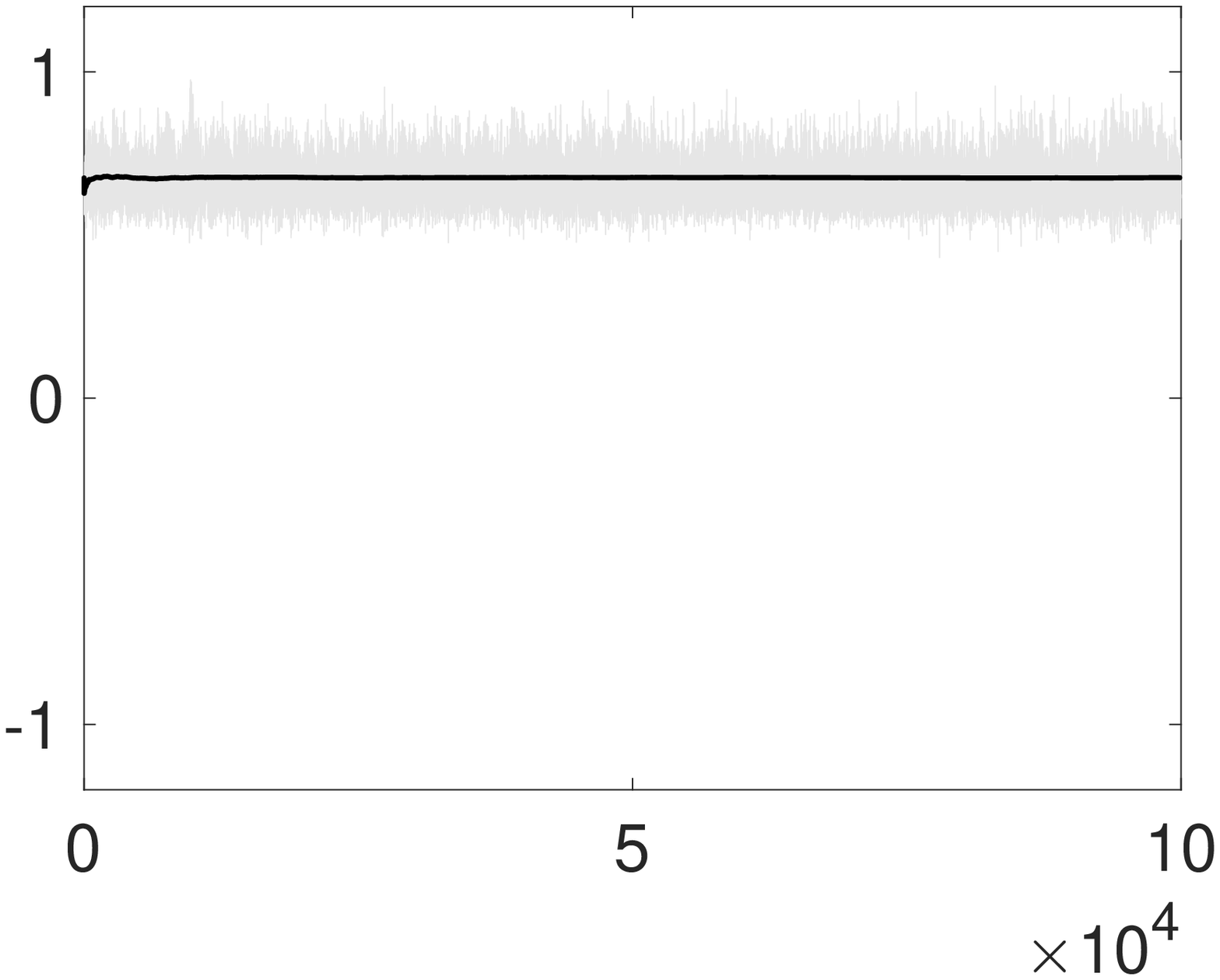}} 
  \subcaptionbox{$v^N_{66},\ N=81$ chain and cumulative mean.}{\includegraphics[width=0.32\textwidth]{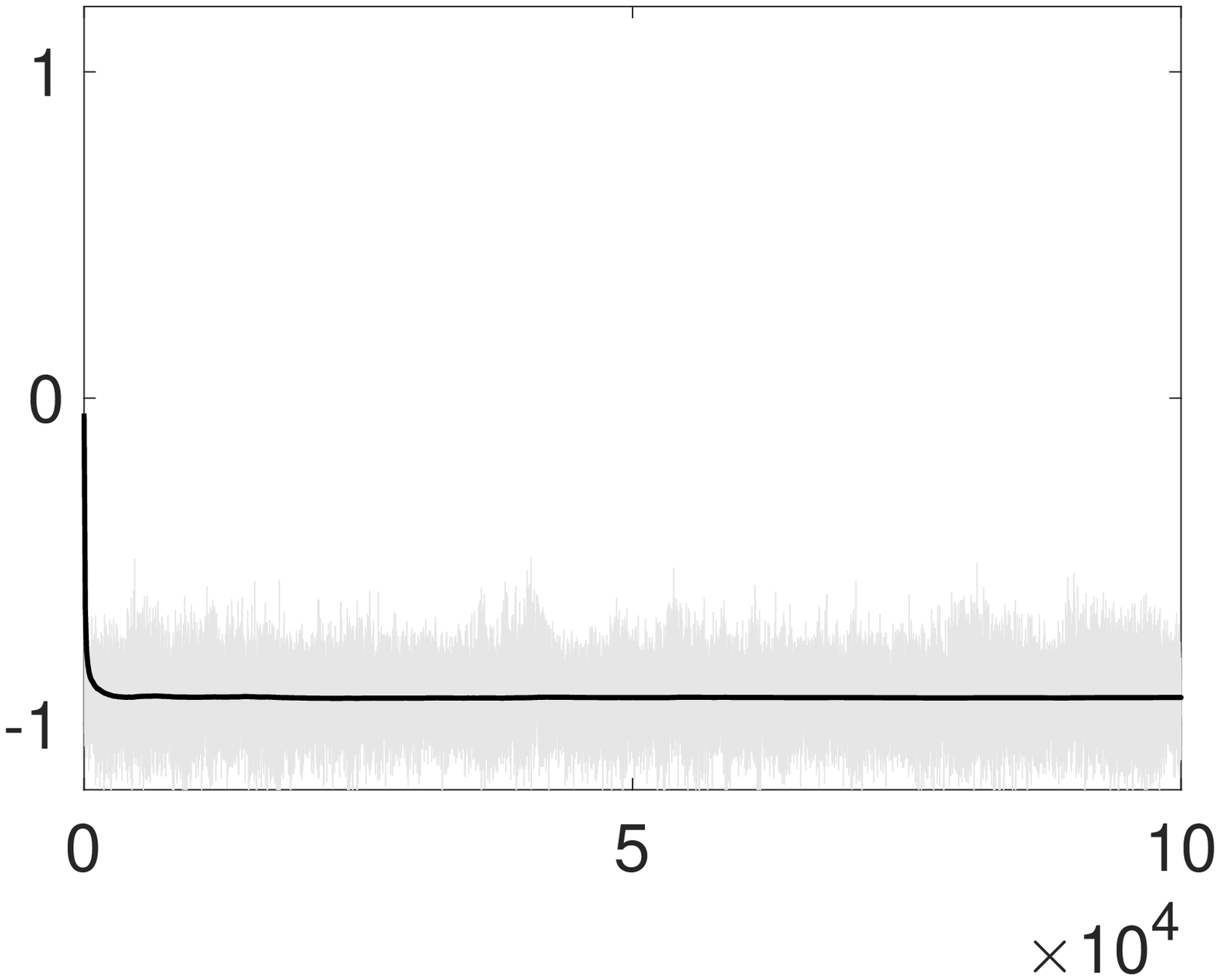}} 

  \caption{Estimates of $\ell^N$ and $v^N$ with a Cauchy walk hypermodel $u^N$ on different lattices with $81$ measurements, with the number of unknowns varying as in figures.  Bottom four subfigures (G-J) are chains and cumulative means of certain $\ell^N$ and $v^N$ elements.} \label{fig:CM_est_discretisation_invariance}
  \end{center}
\end{figure}

\begin{figure}[htp]
\begin{center}

  \subcaptionbox{$\ell^N,\ N=81$}{\includegraphics[width=0.32\textwidth]{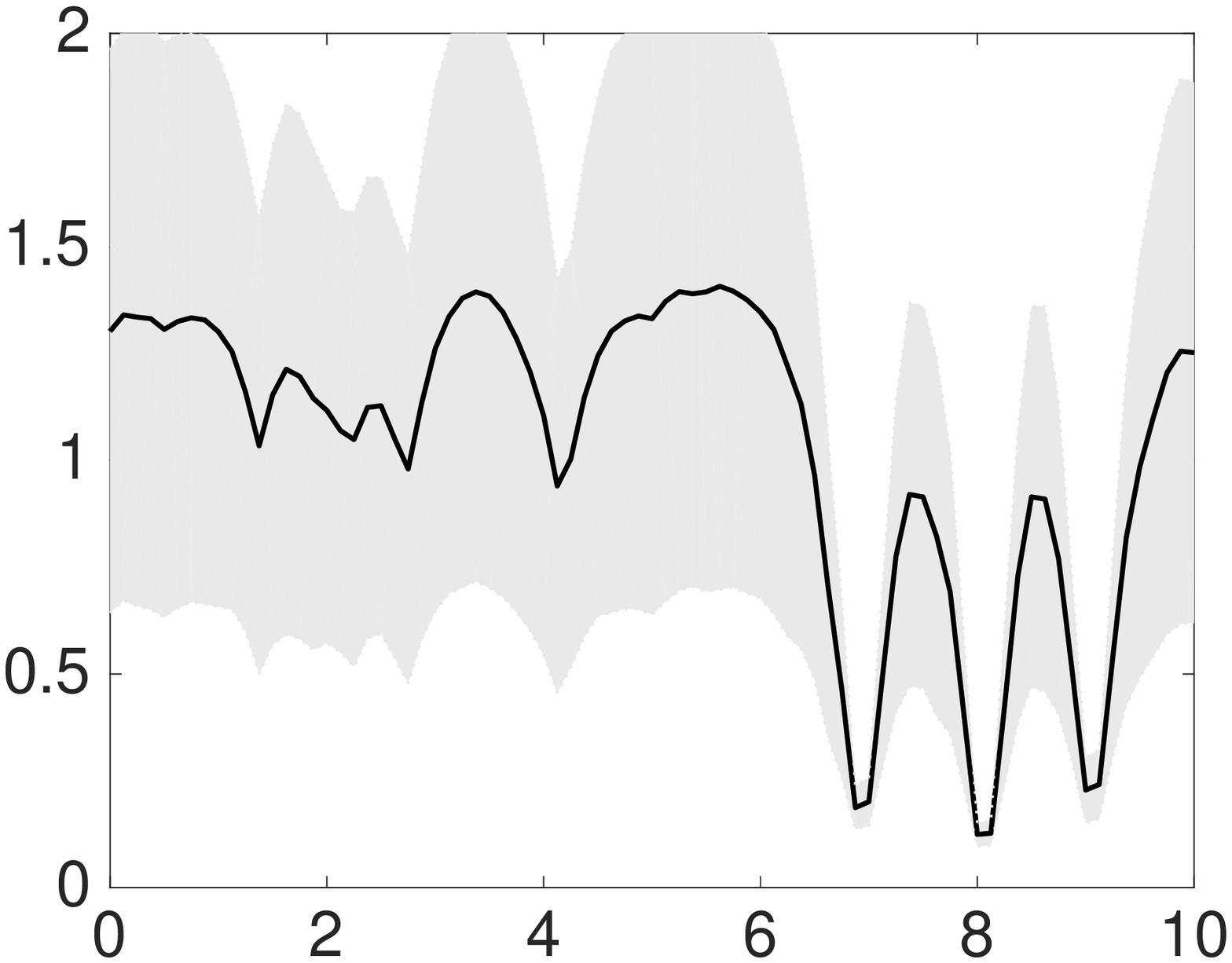}} 
  \subcaptionbox{$\ell^N,\ N=161$}{\includegraphics[width=0.32\textwidth]{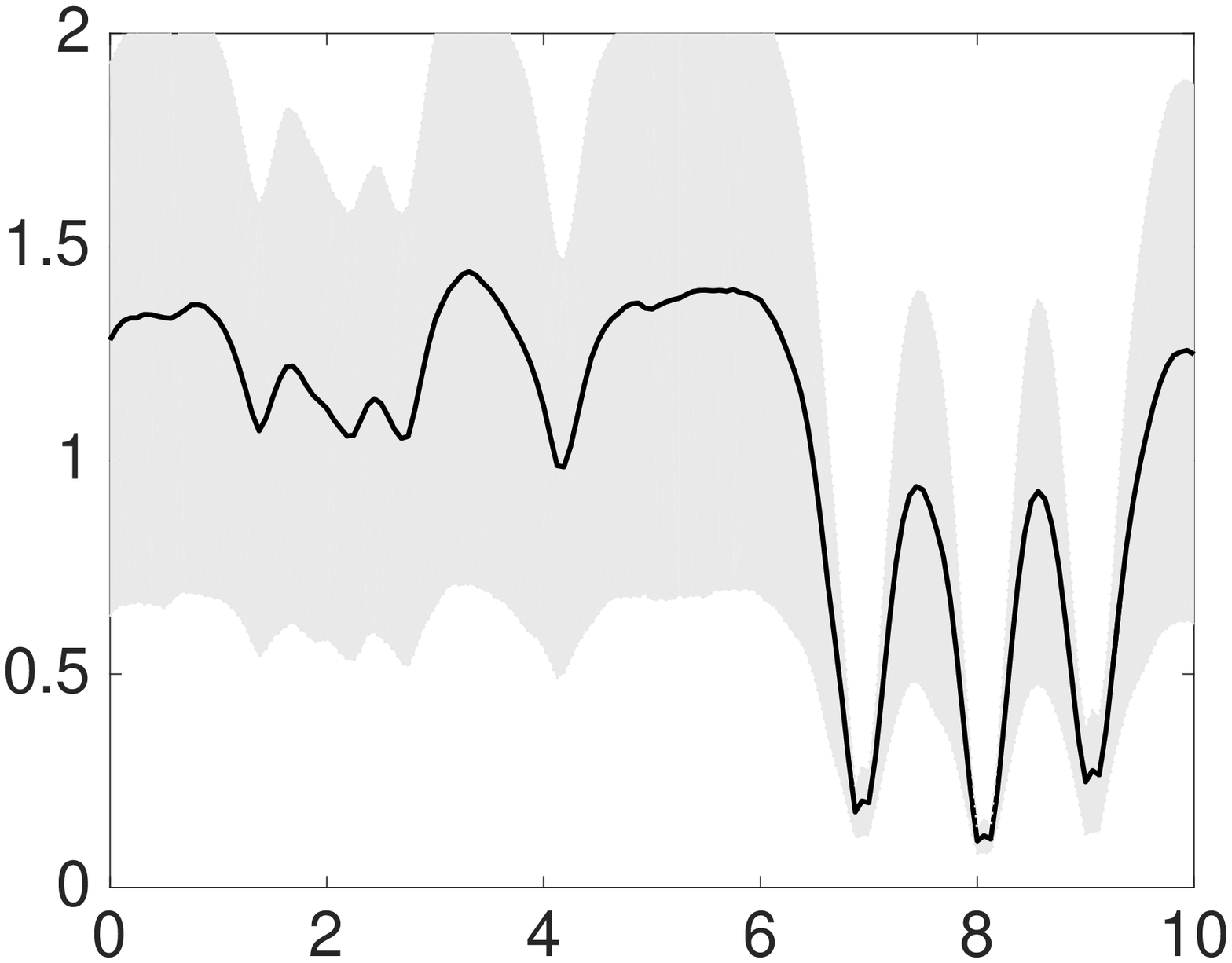}} 
  \subcaptionbox{$\ell^N,\ N=321$}{\includegraphics[width=0.32\textwidth]{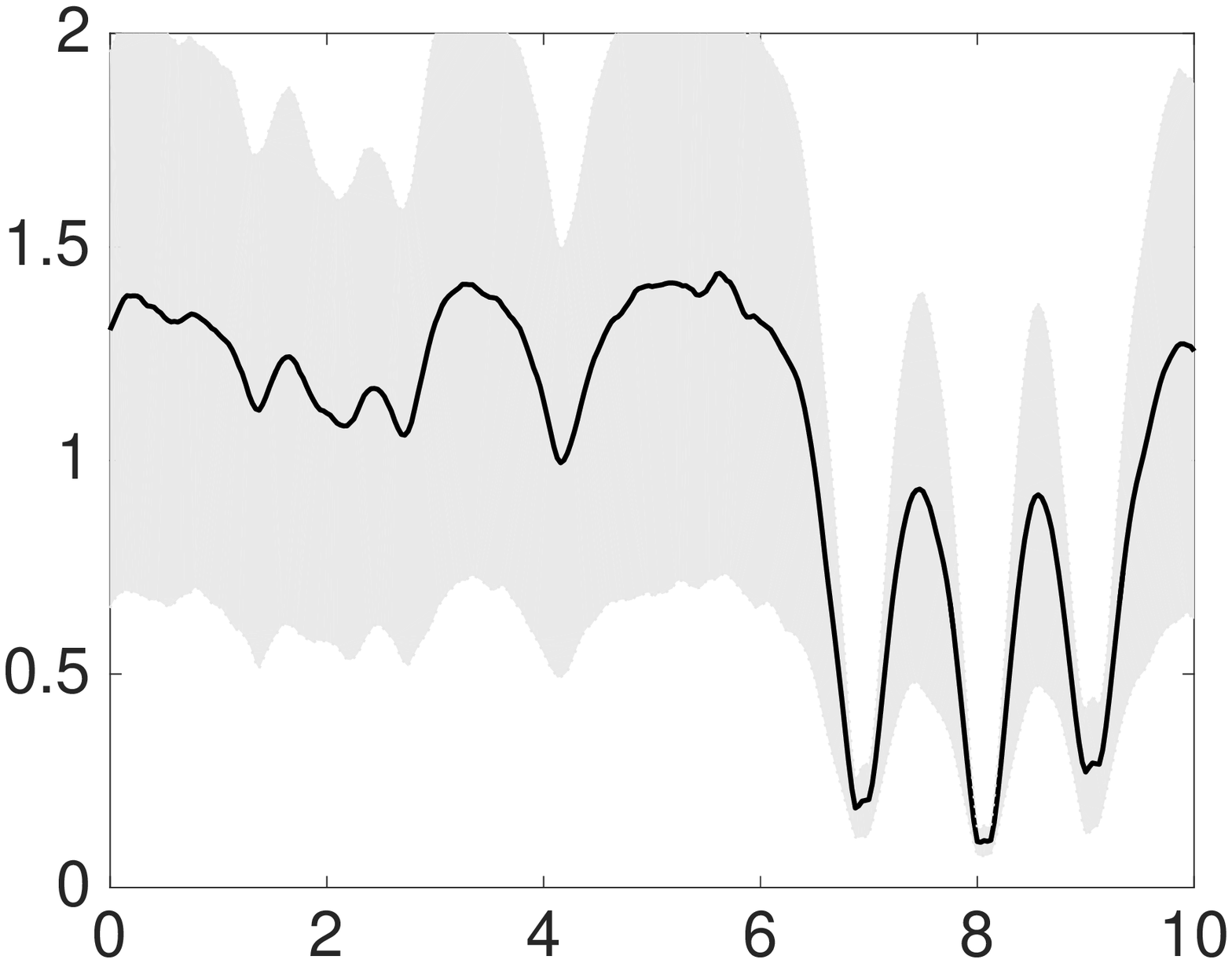}} \\
  
  \subcaptionbox{$v^N,\ N=81$}{\includegraphics[width=0.32\textwidth]{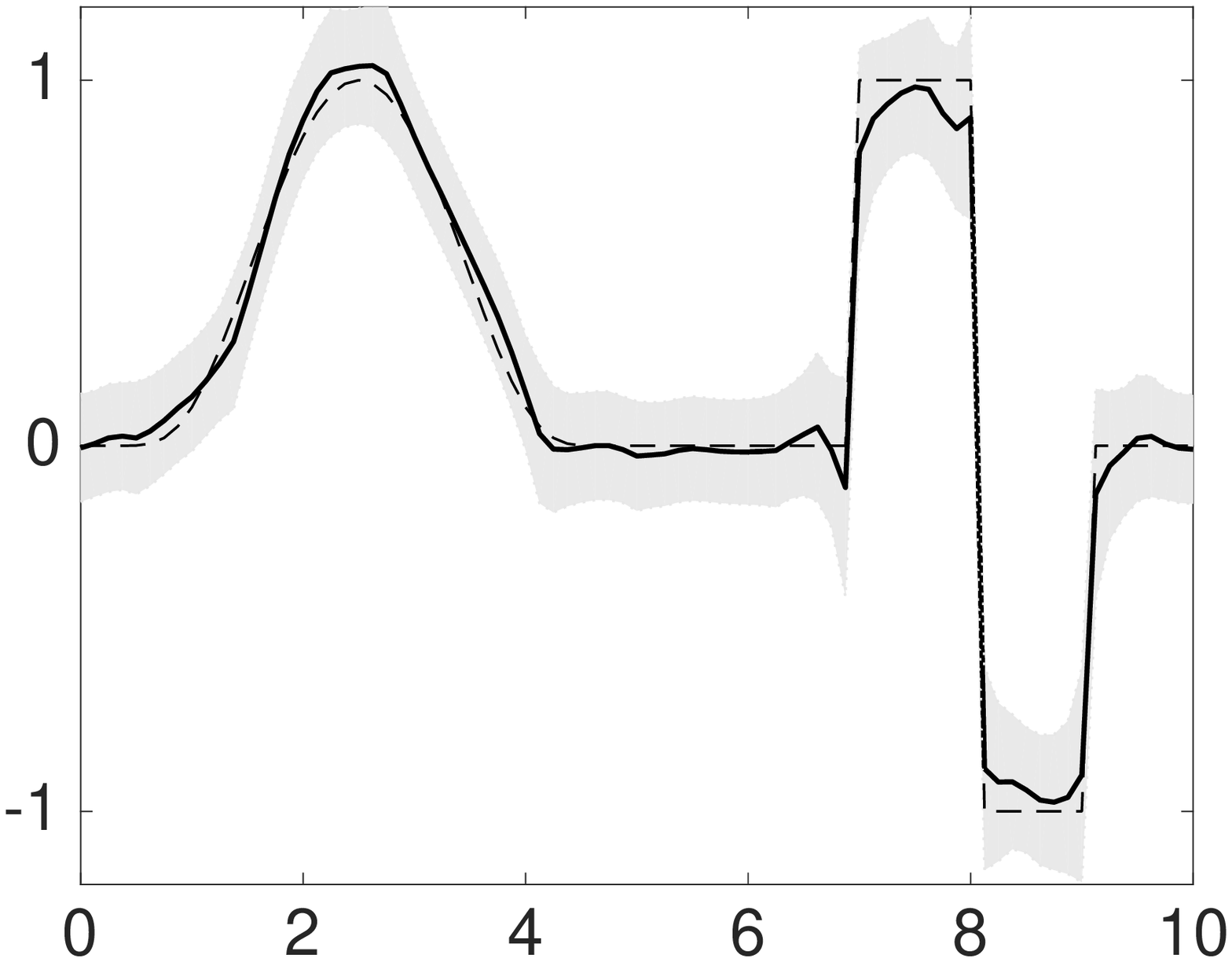}} 
  \subcaptionbox{$v^N,\ N=161$}{\includegraphics[width=0.32\textwidth]{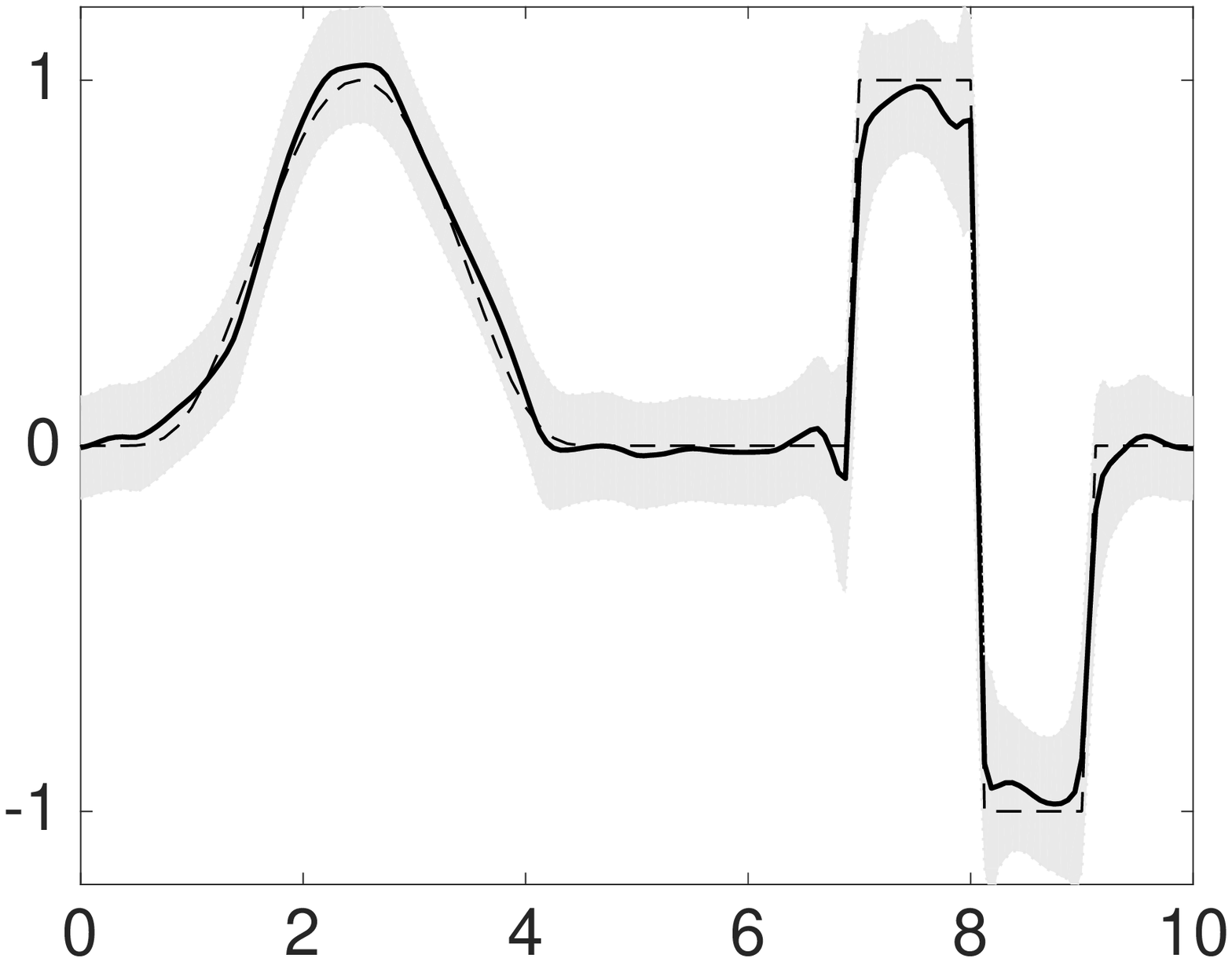}} 
  \subcaptionbox{$v^N,\ N=321$}{\includegraphics[width=0.32\textwidth]{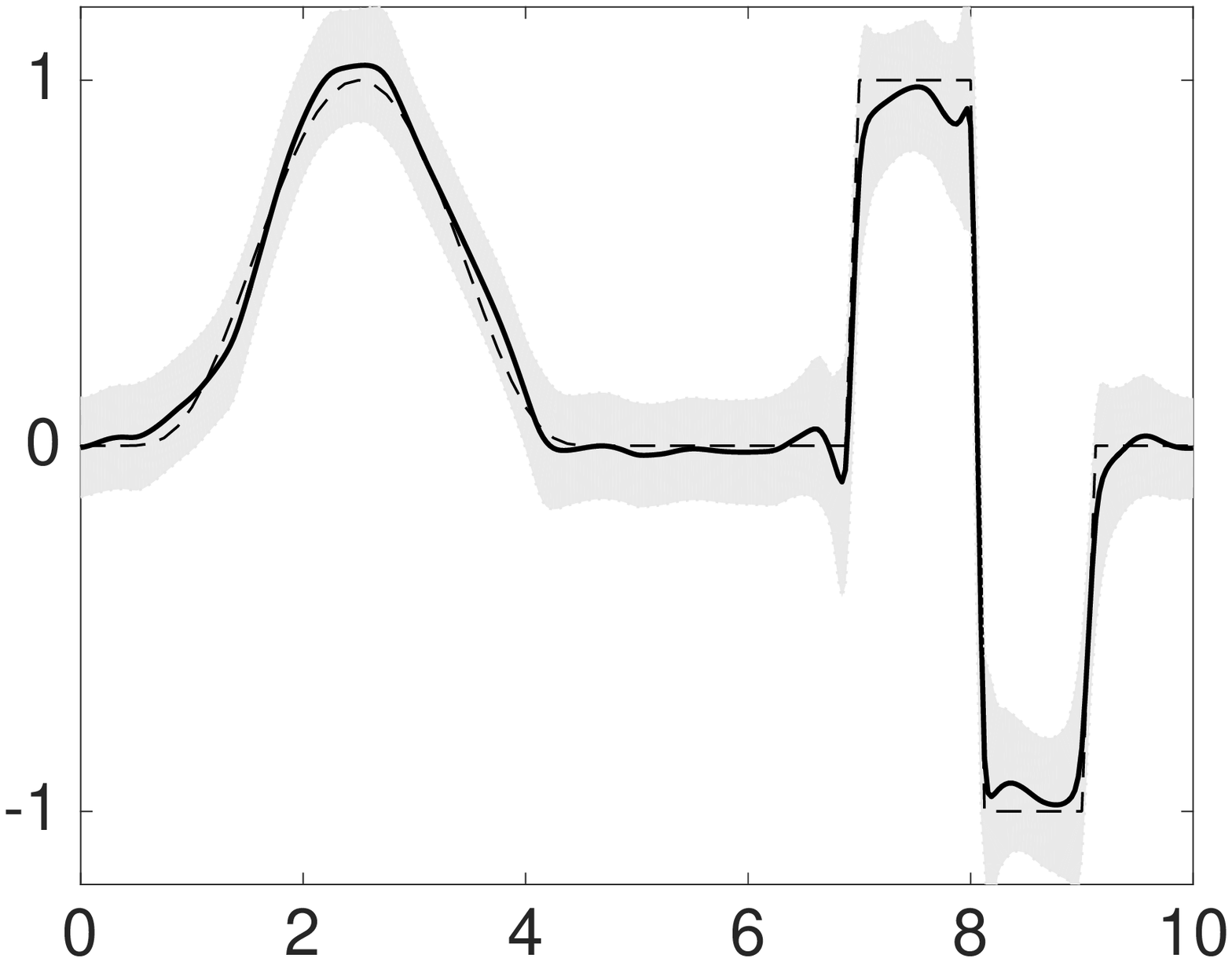}} \\

  \subcaptionbox{$\ell^N_{15},\ N=81$ chain and cumulative mean.}{\includegraphics[width=0.32\textwidth]{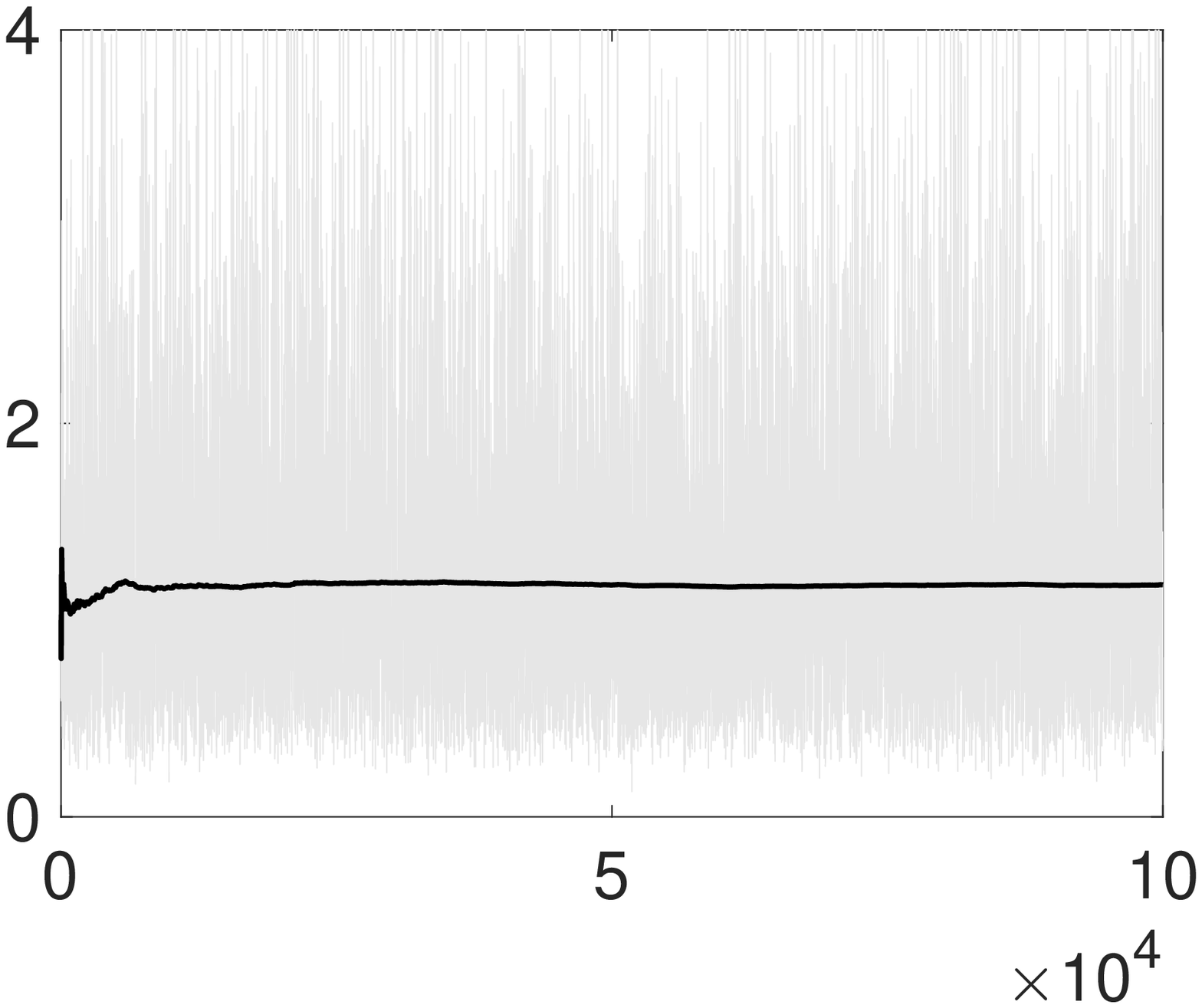}} 
  \subcaptionbox{$\ell^N_{66},\ N=81$ chain and cumulative mean.}{\includegraphics[width=0.32\textwidth]{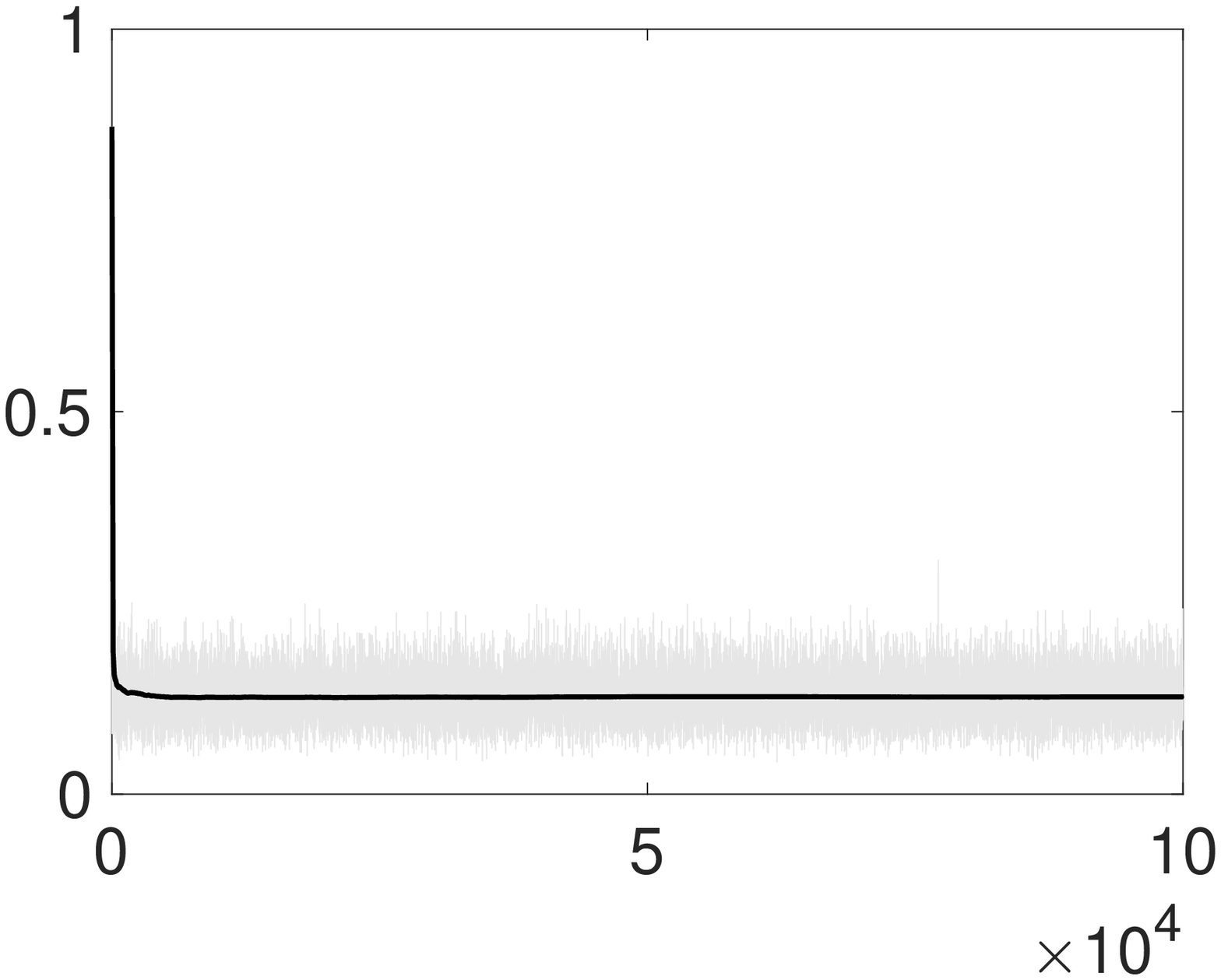}} \\
  \subcaptionbox{$v^N_{15},\ N=81$ chain and cumulative mean.}{\includegraphics[width=0.32\textwidth]{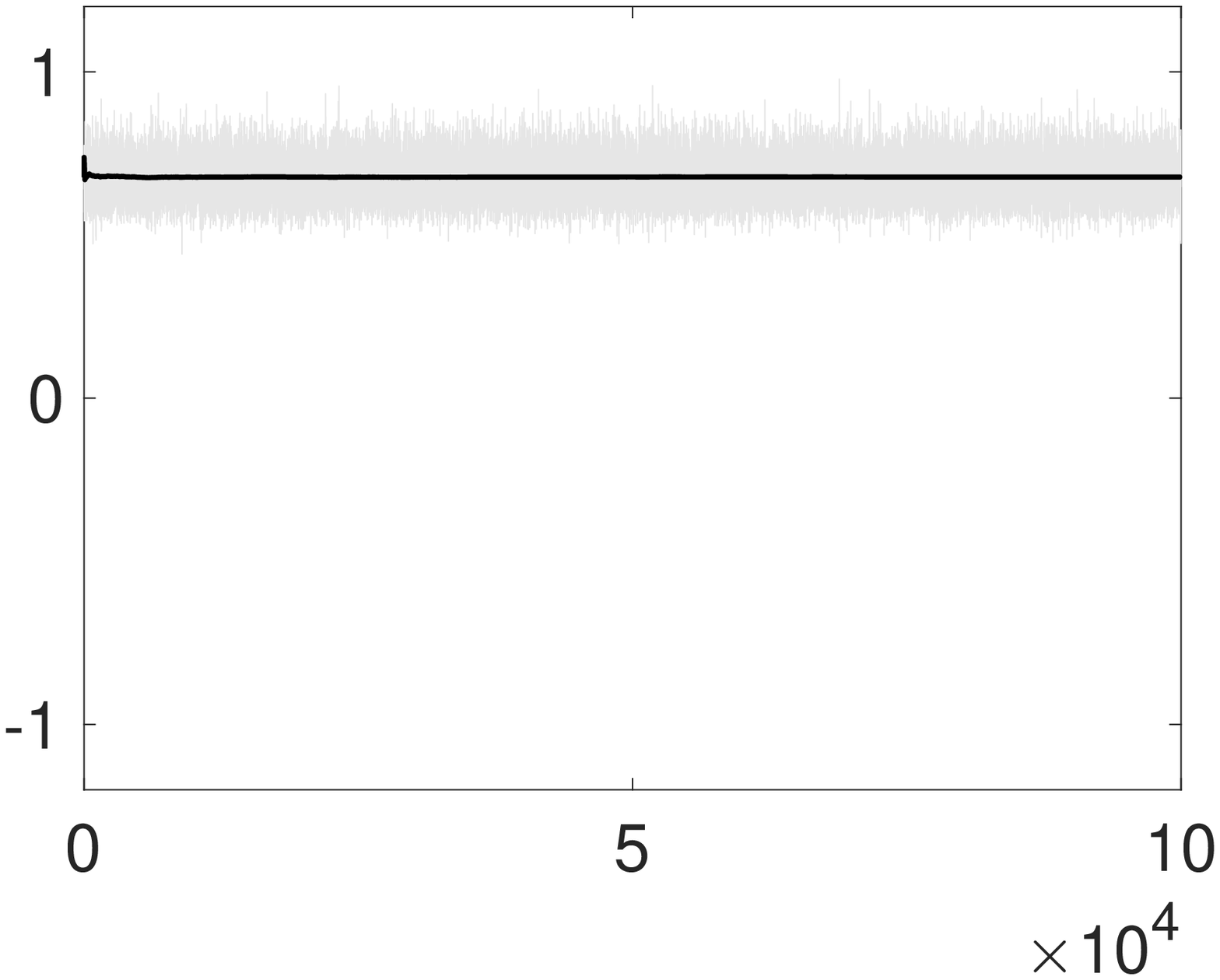}} 
  \subcaptionbox{$v^N_{66},\ N=81$ chain and cumulative mean.}{\includegraphics[width=0.32\textwidth]{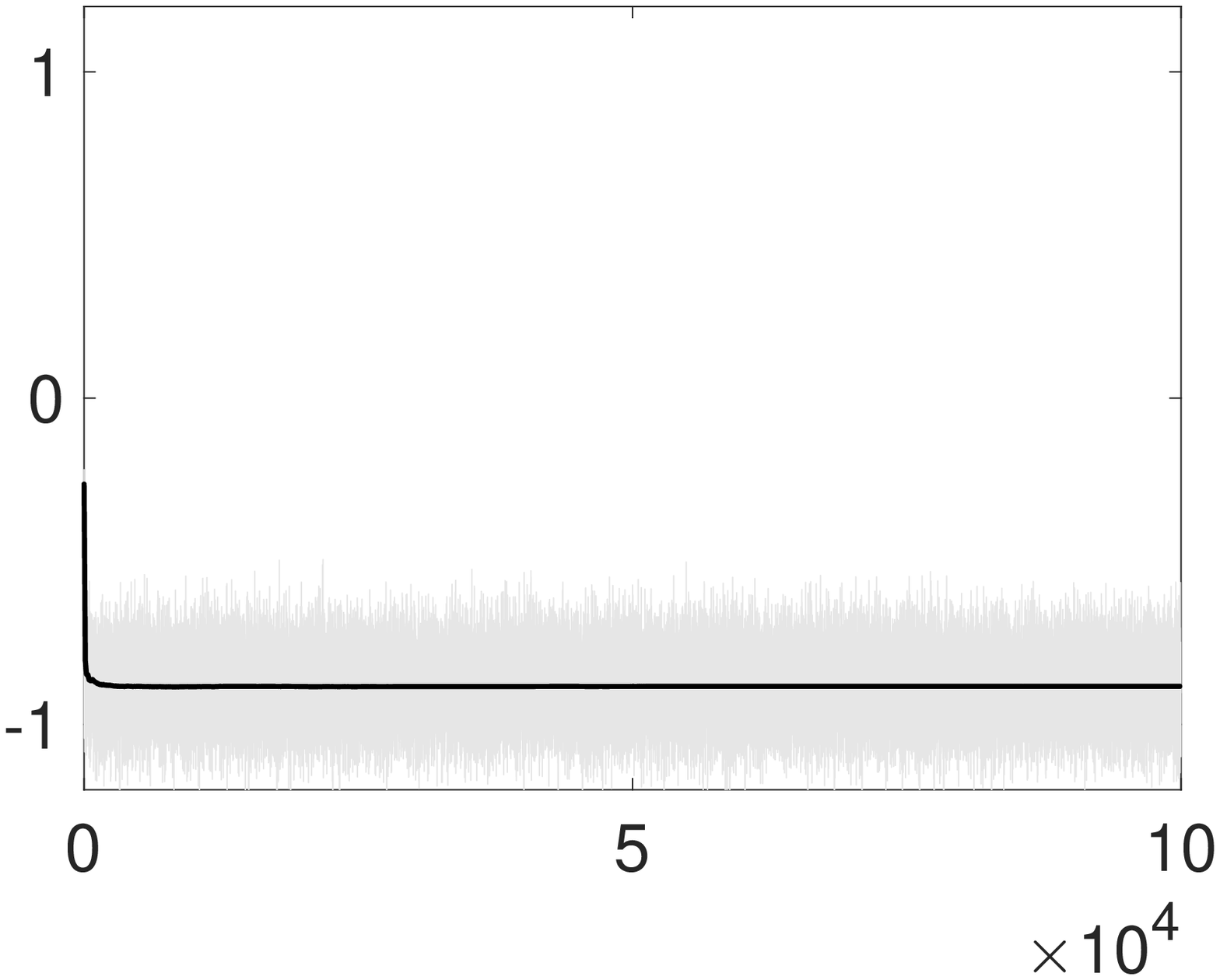}} 
  
  \caption{Estimates of $\ell^N$ and $v^N$ with a Gaussian hyperprior $u^N$ on different lattices with $81$ measurements, with the number of unknowns varying as in figures.  Bottom four subfigures (G-J) are chains and cumulative means of certain $\ell^N$ and $v^N$ elements. } \label{fig:CM_est_GMRF_discretisation_invariance2}
  \end{center}
\end{figure}

\subsection{Multimodal posterior densities}

Let us now consider posterior  densities of $v^N_j$.
If we use the Gaussian hypermodel,  and use exponentiating $u^N$  (Equation \eqref{eqn:exp}), it is easy to show numerically that the posterior densities of $v^N$ are Gaussian. 
It would be tempting to model  longest and shortest length-scaling, i.e.\ a priori lower and upper bounds for $\ell^N$.
One such model, could be given as 
\begin{equation} \label{eqn:lengthscaling222}
g(s)= \begin{cases} \gamma, &  s < P_{\mathrm{upper}} \\ \exp(a \vert  s\vert)-b, & \text{otherwise}\\ \lambda , & s > P_{\mathrm{lower}}\end{cases}
\end{equation}
where $a>0$, $b\in(0,1)$ are some  constants, and $P_{\mathrm{upper}}>P_{\mathrm{lower}}>0$.
It is is convenient to model the lower bound as $\ell(0)= 1-b= P_{\mathrm{lower}}$.
However, using this model leads to multimodal posterior densities due to the max-min cutoff, especially at the jumps.

In Figure \ref{fig:multimodal2}, we have considered the same Gaussian hypermodel and interpolation problem as in Figure \ref{fig:CM_est_1}.
Given the MCMC chains, we compute kernel density estimates of of the posterior densities at the jump at $x\approx 8$.
This corresponds to grid elements at $j=129,\ 130,\ 131$, where we have a jump from $+1$ to $-1$.
The densities at grid elements $129$ and $131$ are Gaussian. 
However, the at grid element $130$, the density is trimodal. 
We have plotted also the Gaussian density estimate with dashed line, and clearly it fails to capture the trimodality.
The reason for the trimodal density is, assumably, in the non-linear transformation of Equation \eqref{eqn:lengthscaling222}. 
Hence, the algorithm does detect edges, but we need to be careful when assessing the multimodality of the posterior densities.

\begin{figure}[htp]
\begin{center}
  \subcaptionbox{$v^N_{129},\ N=161$}{\includegraphics[width=0.32\textwidth]{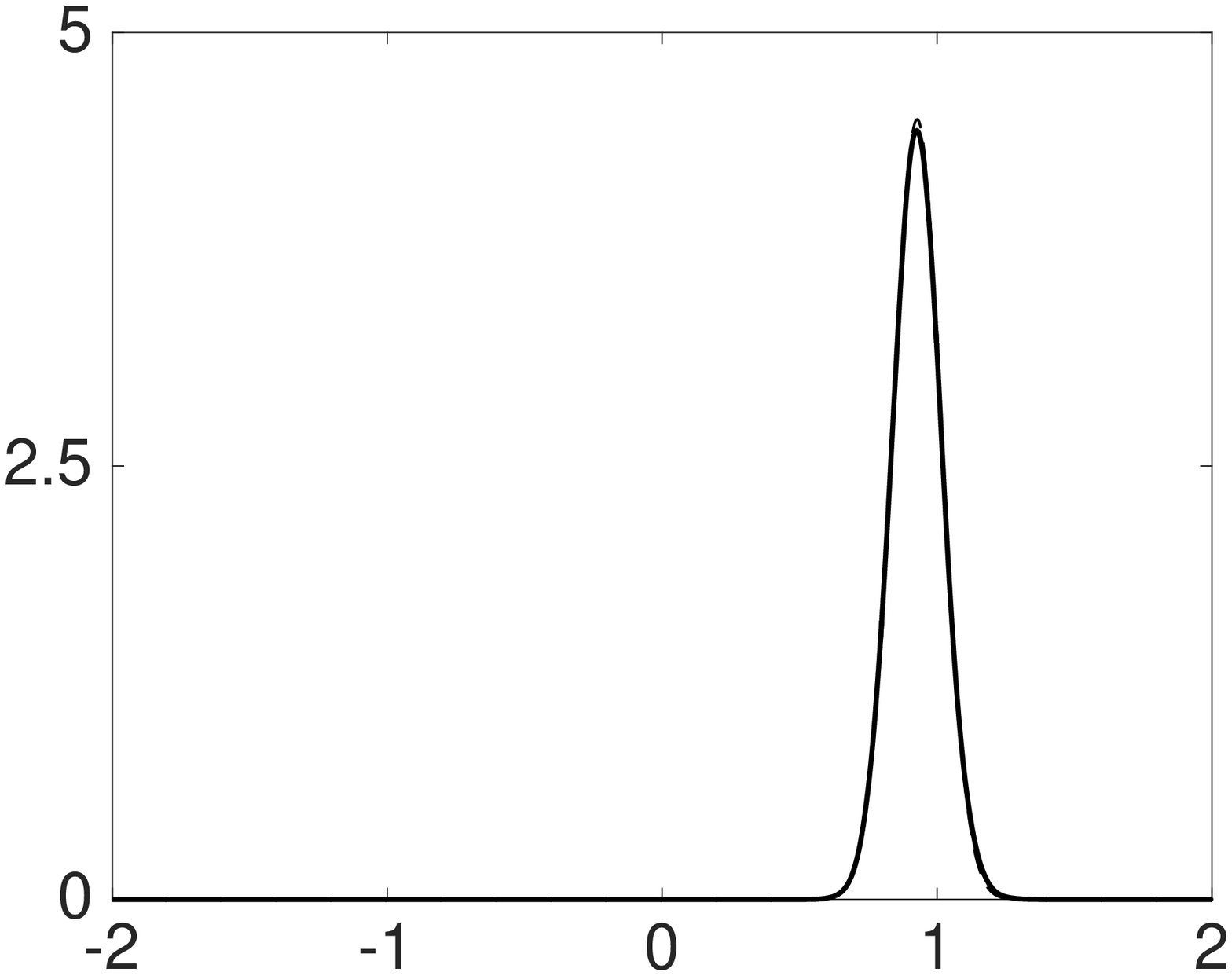}} 
  \subcaptionbox{$v^N_{130},\ N=161$}{\includegraphics[width=0.32\textwidth]{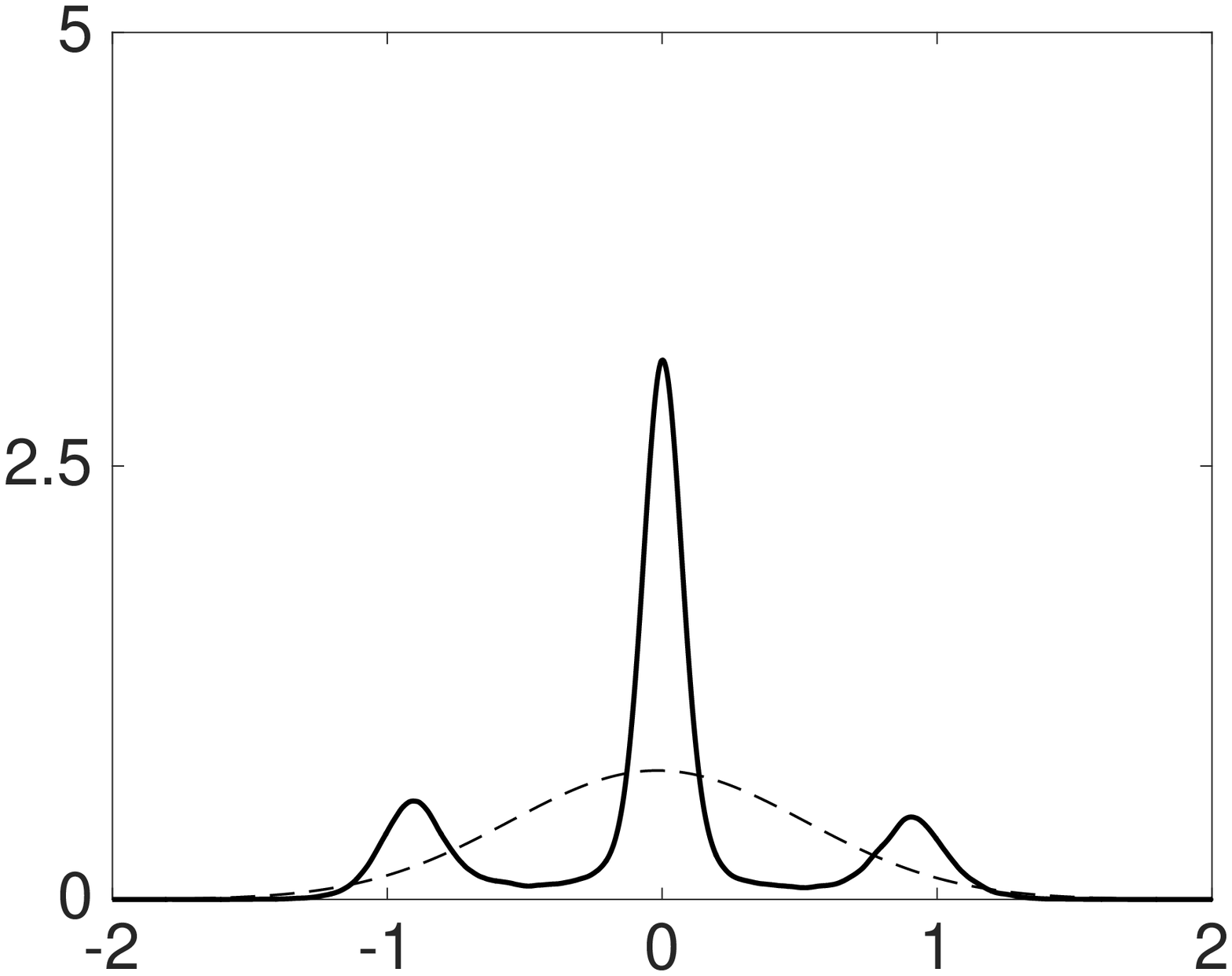}}
  \subcaptionbox{$v^N_{131},\ N=161$}{\includegraphics[width=0.32\textwidth]{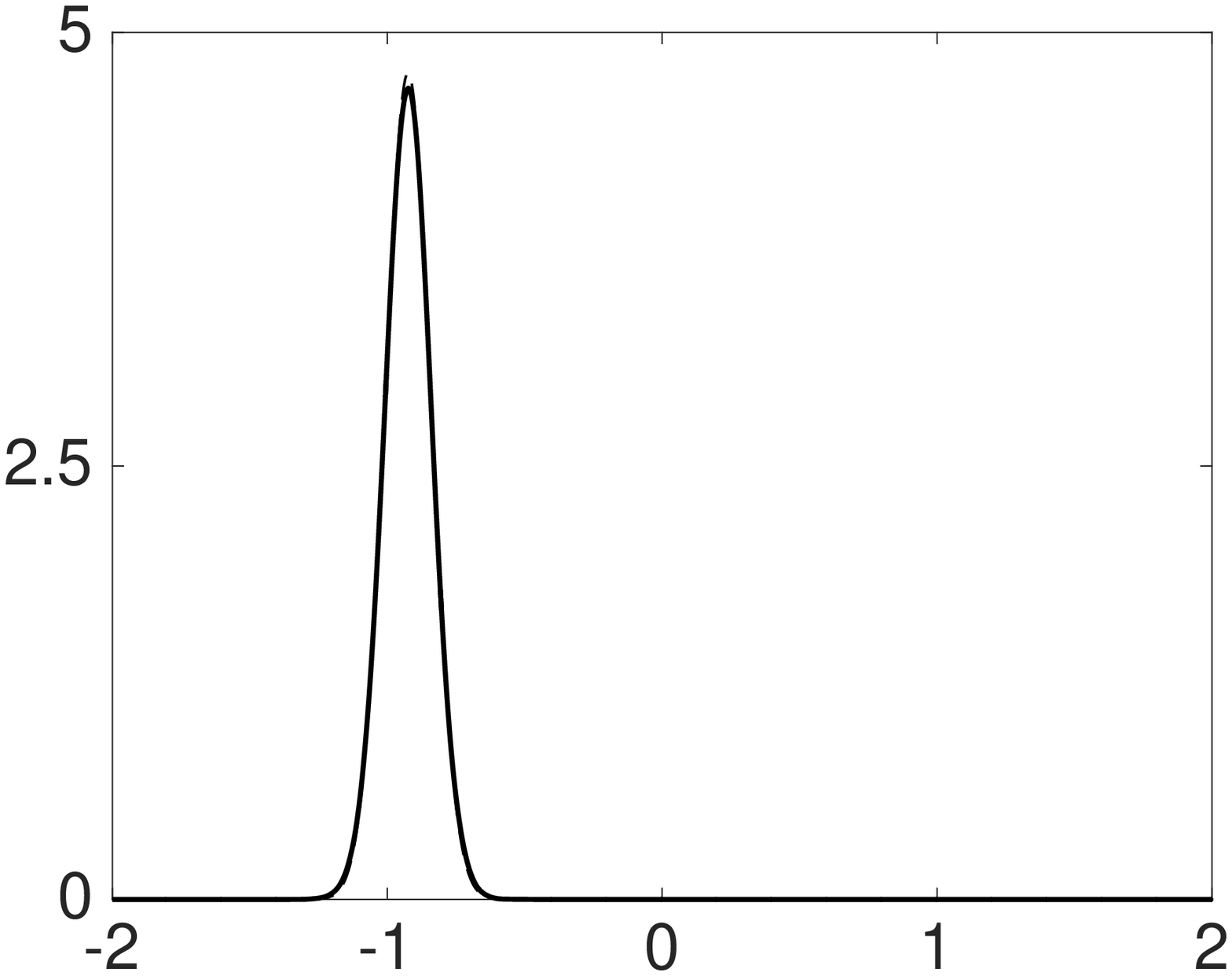}}

  \caption{Multimodality of the posterior densities $v^N_j$ at the edge around $x\approx 8$, when using Gaussian hyperprior and Equation \eqref{eqn:lengthscaling222}.  
  } \label{fig:multimodal2}
  \end{center}
\end{figure}

\subsection{Numerical differentiation}

As a second numerical example, we consider numerical differentiation of noisy data.
The continuous forward model is given with a first kind Fredholm integral equation 
\begin{equation*}
y(jh) = \int H(jh-x)v(x)dx + e_j,
\end{equation*}
where the convolving kernel $H$ is a Heaviside step function.
If we had a deterministic equation (i.e.\ no noise term $e_j$), then $y'=v$ is the differentiated function. 
Hence, we can formulate differentiation as a linear Bayesian inverse problem with observations given as $y=Av^N+e$.

For numerical tests, we choose an unknown consisting of a mollifier and a triangle function
\begin{equation*}
\int H(x'-x)v(x)dx = 
\begin{cases}
\exp\left(4-\frac{25}{x'(5-x')}\right), &  x'\in(0,5) \\
x'-7, & x'\in[7,8] \\
-x'+9, & (8,9] \\
0, & \text{otherwise}.
 \end{cases}
\end{equation*}
Derivative of a triangle function is a piecewise constant function, the same two boxcar functions as in Equation \eqref{eqn:interpolationunknown}.
Differentiated mollifier is again a smoothly varying function. 
We have chosen the mollifier constants in such a way that the differentiated mollifier stays in the same range as the mollifier, simply to avoid any visualisation problems in scaling.
The unknown function is then
\begin{equation*}
v(x) = 
\begin{cases}
\left(\frac{25}{x^2(5-x)}-\frac{25}{x(5-x)^2}\right)\exp\left(4-\frac{25}{x(5-x)}\right), &  x\in(0,5) \\
1, & x\in[7,8] \\
-1, & (8,9] \\
0, & \text{otherwise}.
 \end{cases}
\end{equation*}

In Figure \ref{fig:diff}, we have numerical derivatives $v^N$ on three different meshes, as well as the Gaussian hyperprior process $\ell^N$.
In the simulations, we have 101 observations $y$ with measurement noise $\sigma=0.03$.
We note that as numerical differentiation is an ill-posed problem, so we cannot use as high noise-levels as in the interpolation examples. 
However, with the used noise levels, the algorithm finds the edges, as well as the smooth structures.

\begin{figure}[htp]
\begin{center}

  \subcaptionbox{101 noisy measurements}{\includegraphics[width=0.32\textwidth]{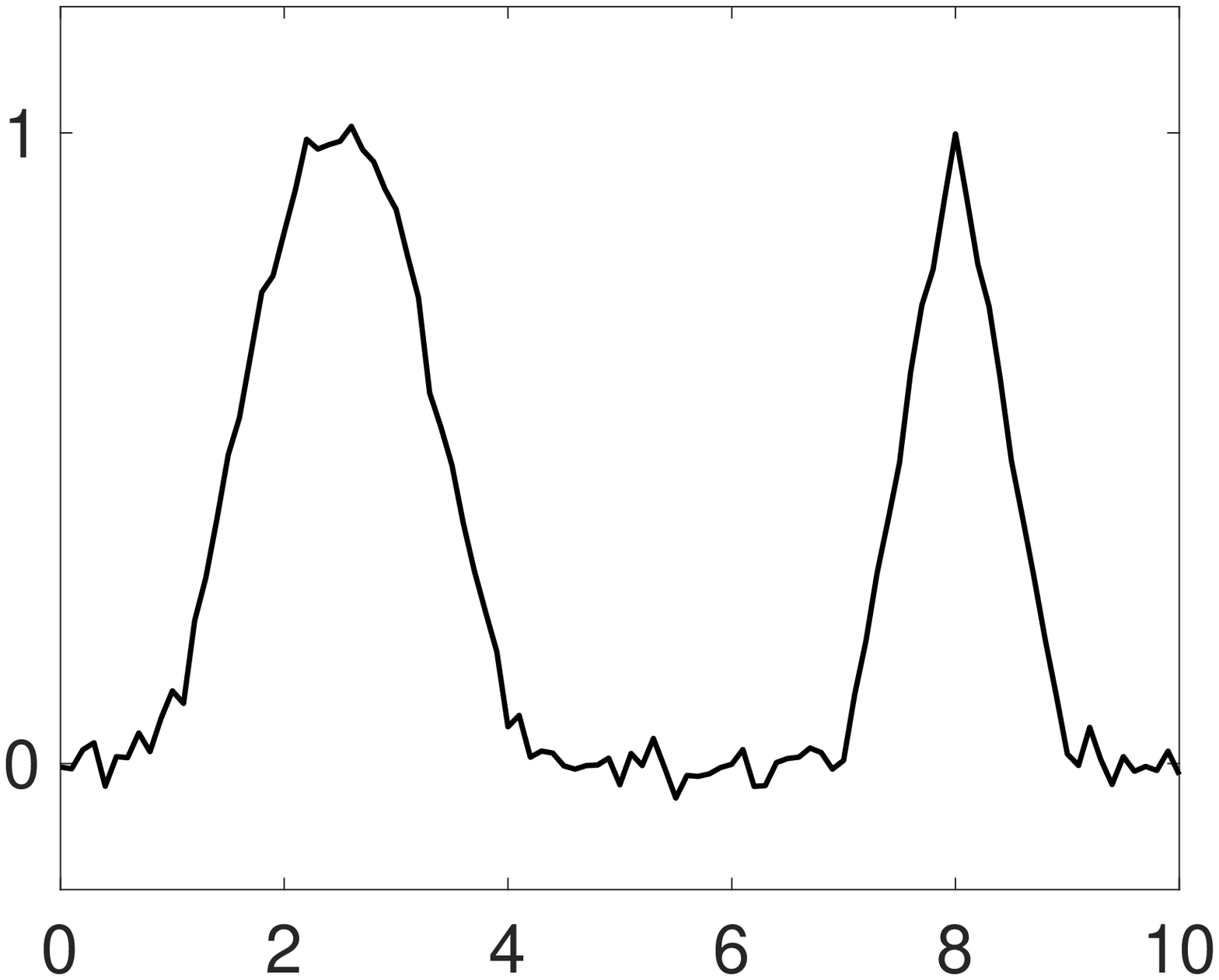}} \\

  \subcaptionbox{$\ell^N,\ N=101$}{\includegraphics[width=0.32\textwidth]{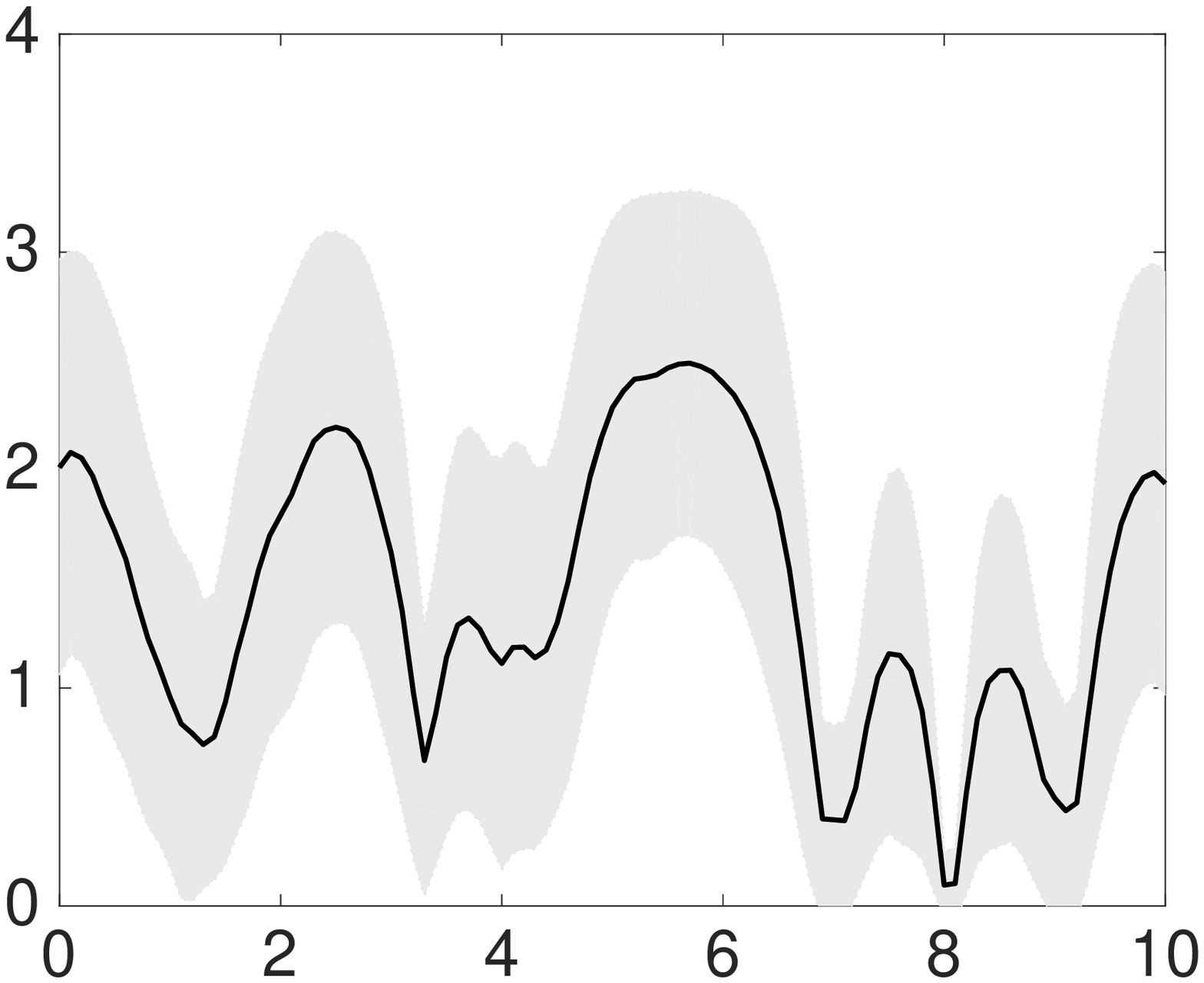}} 
  \subcaptionbox{$\ell^N,\ N=201$}{\includegraphics[width=0.32\textwidth]{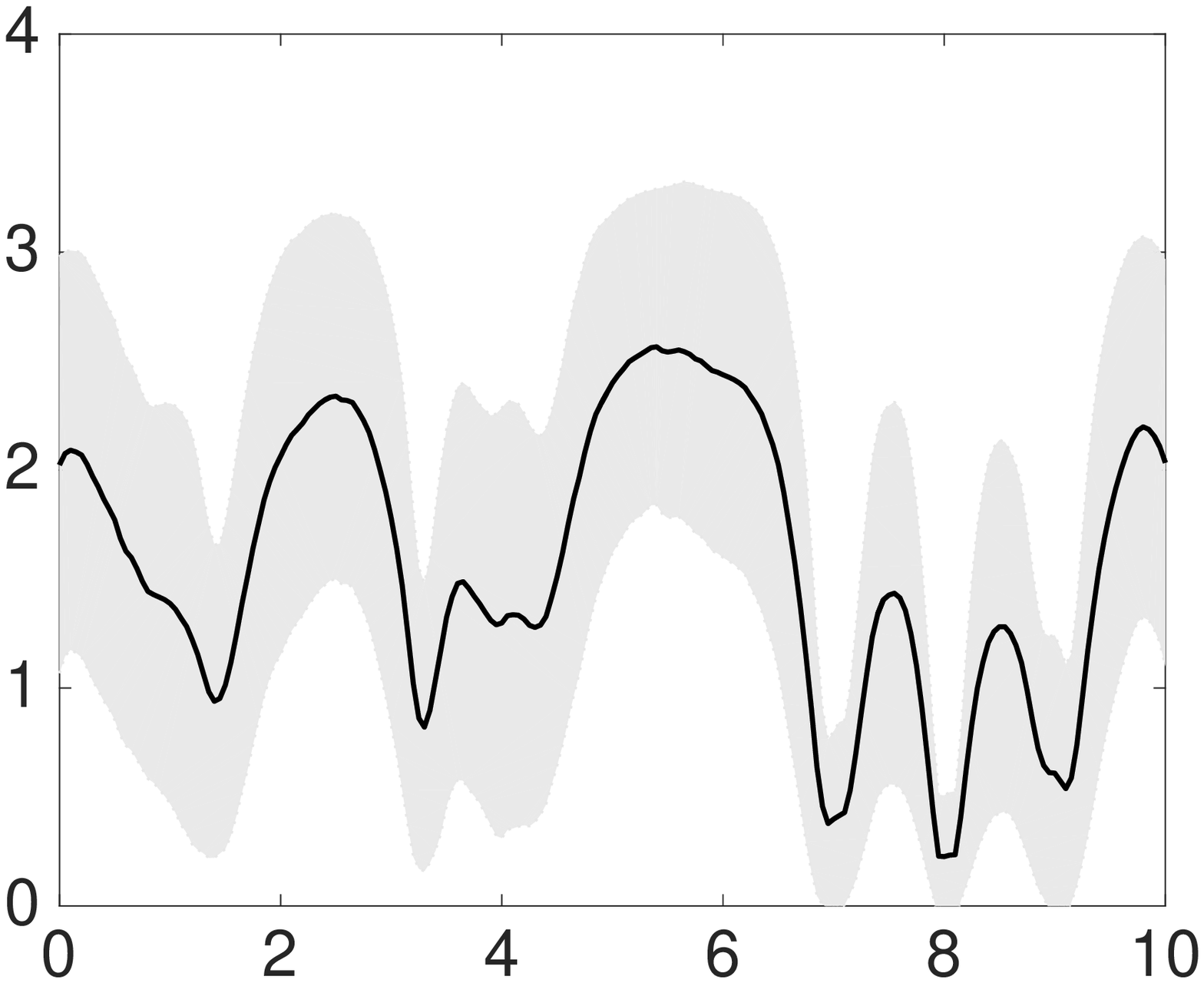}} 
  \subcaptionbox{$\ell^N,\ N=401$}{\includegraphics[width=0.32\textwidth]{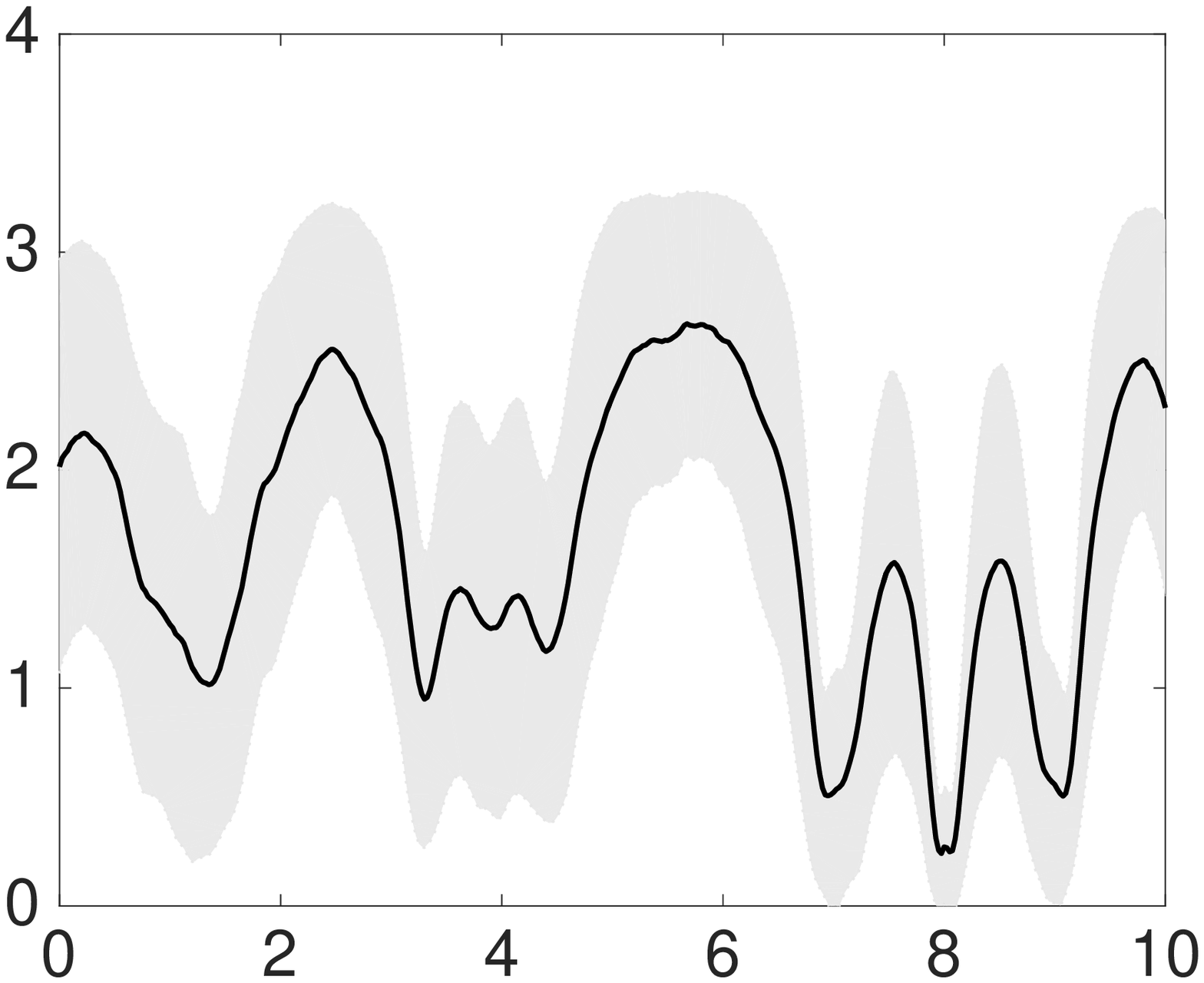}} \\
  
  \subcaptionbox{$v^N,\ N=101$}{\includegraphics[width=0.32\textwidth]{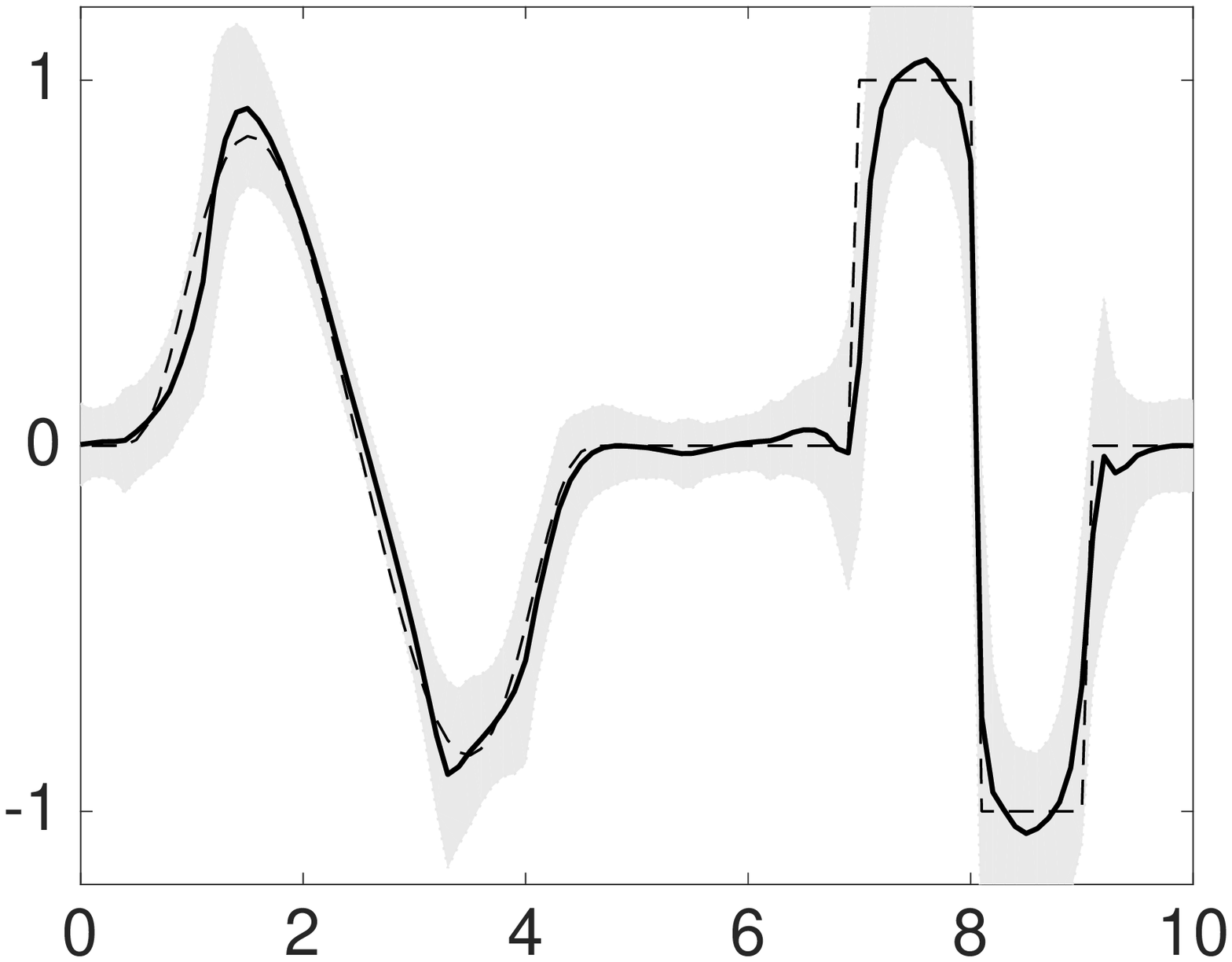}} 
  \subcaptionbox{$v^N,\ N=201$}{\includegraphics[width=0.32\textwidth]{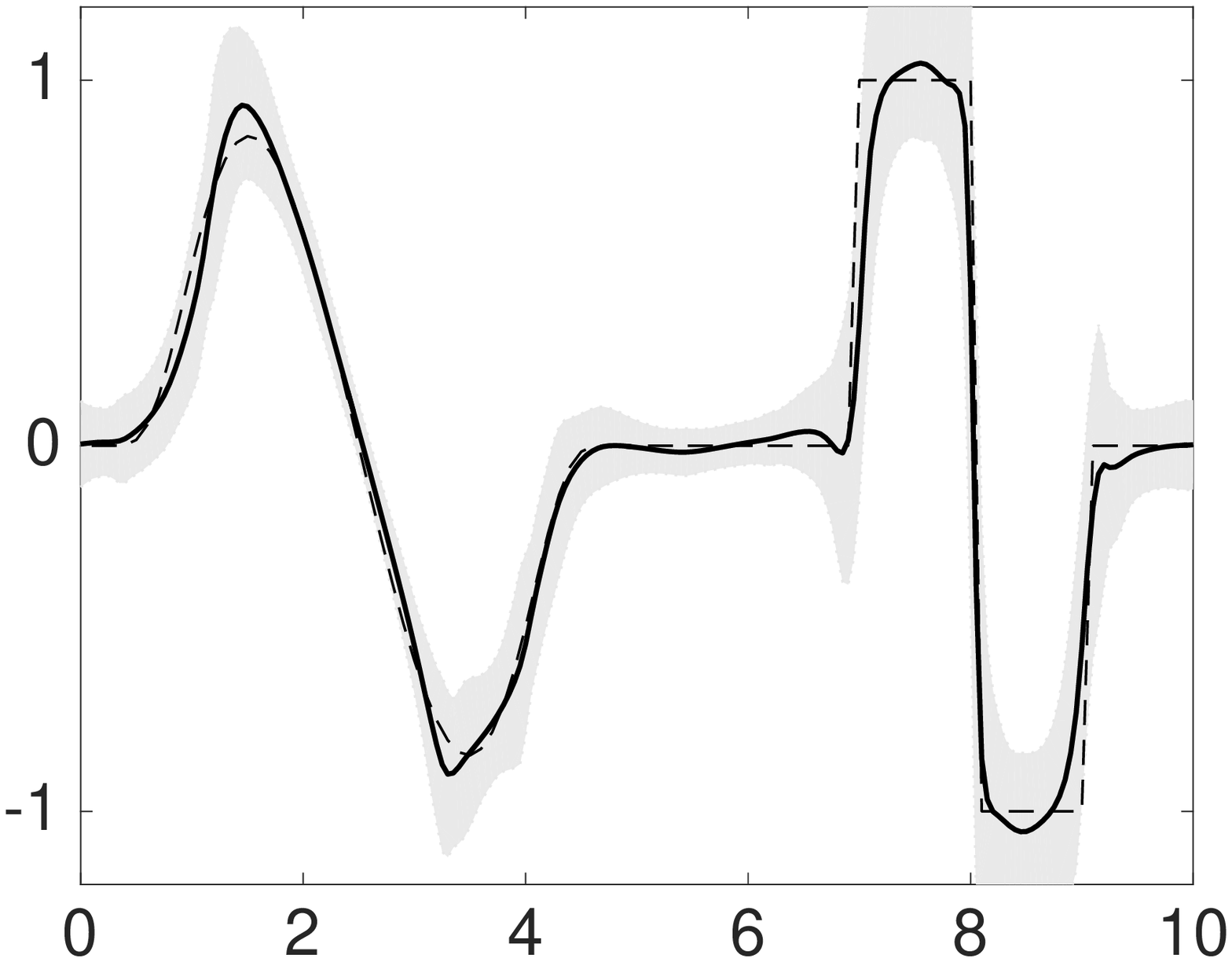}} 
  \subcaptionbox{$v^N,\ N=401$}{\includegraphics[width=0.32\textwidth]{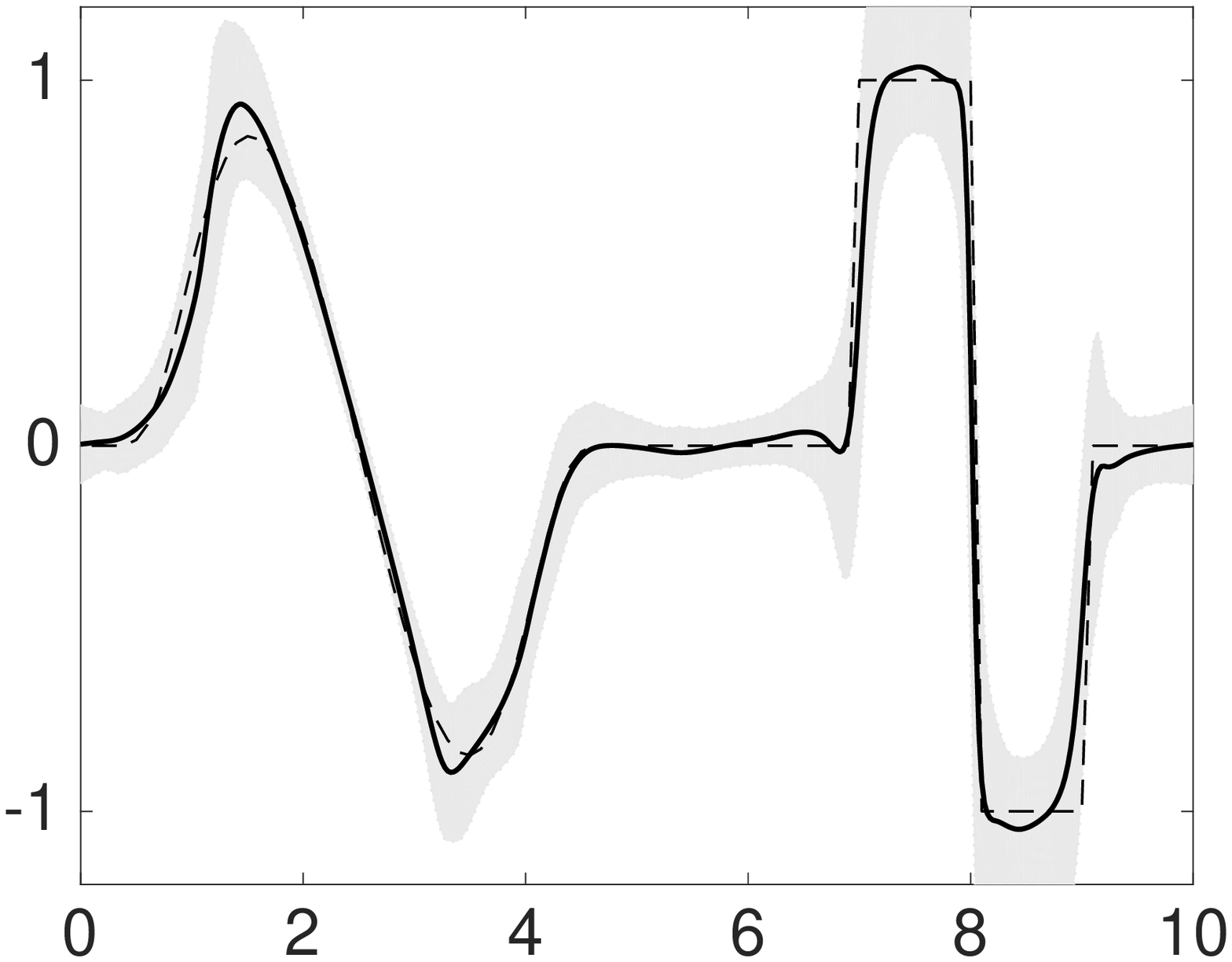}} 
  
  \caption{Numerical differentiation of a noisy signal with the developed Gaussian hypermodel. We plot $v^N$ on different meshes for seeing the discretisation-invariance of the estimates.}    \label{fig:diff}
  \end{center}
\end{figure}

\subsection{Two-dimensional interpolation}

Similarly to the one-dimensional interpolation examples with Cauchy and Gaussian hypermodels, we can make two-dimensional interpolation. 
In Figure \ref{fig:2D_gaussblock}, we have interpolation of noisy observations, originally on a $41\times 41$ mesh, and we estimate the unknown on a $81\times 81$ mesh.
Measurement noise standard deviation is $\sigma=0.025$.
As a hyperprior, we have used a two-dimensional Mat\'ern field with short length-scaling and periodic boundary conditions.
The unknown consists of a rectangle-shaped box of height 0.75 and a Gaussian-shaped smooth function of height 1.
We can clearly detect both smooth and edgy properties in this case also.

\begin{figure}[htp]
\begin{center}
  \subcaptionbox{Unknown}{\includegraphics[width=0.49\textwidth]{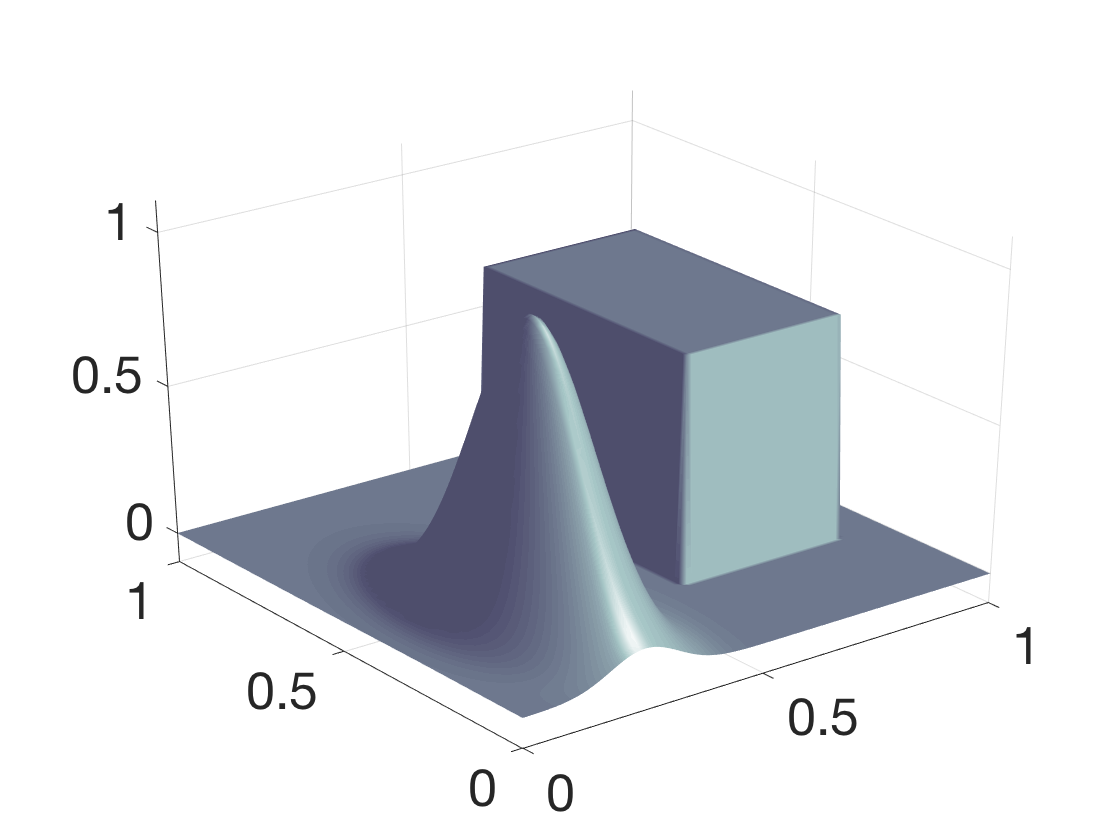}} 
  \subcaptionbox{Noisy measurements on $41\times 41$ mesh}{\includegraphics[width=0.49\textwidth]{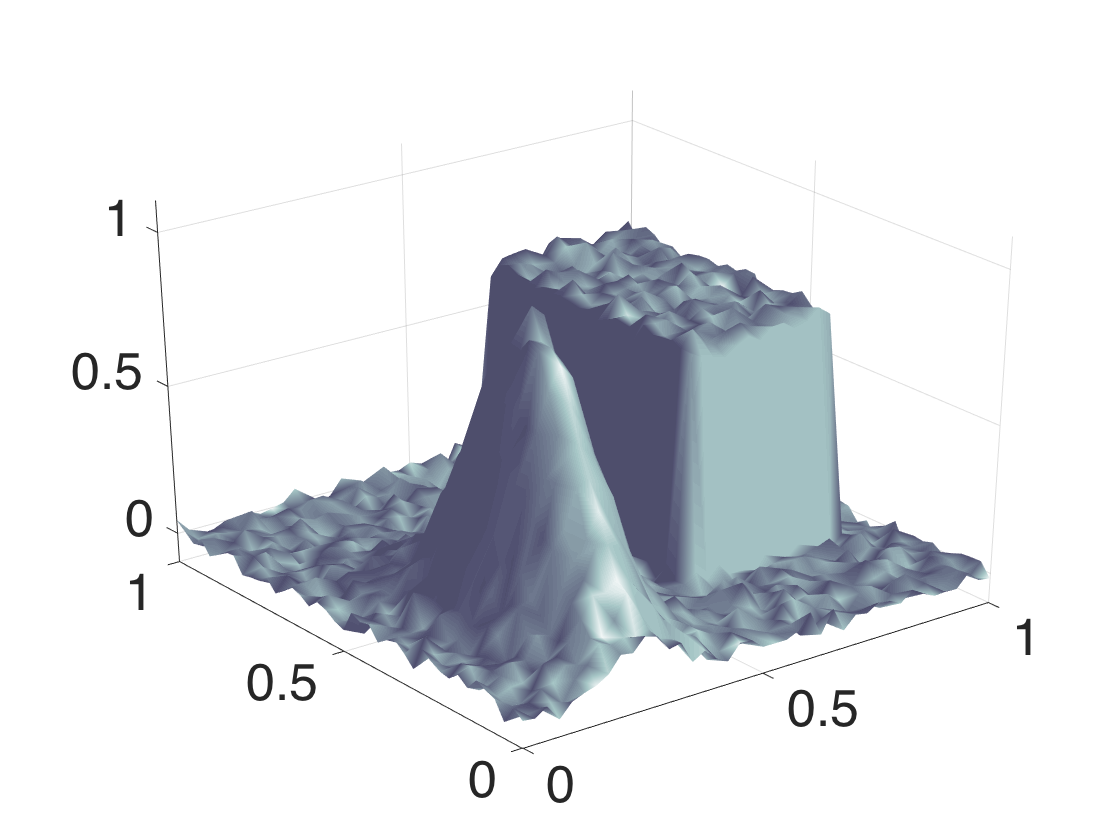}} 
  \subcaptionbox{Estimated $\ell^N$ on $81\times 81$ mesh}{\includegraphics[width=0.49\textwidth]{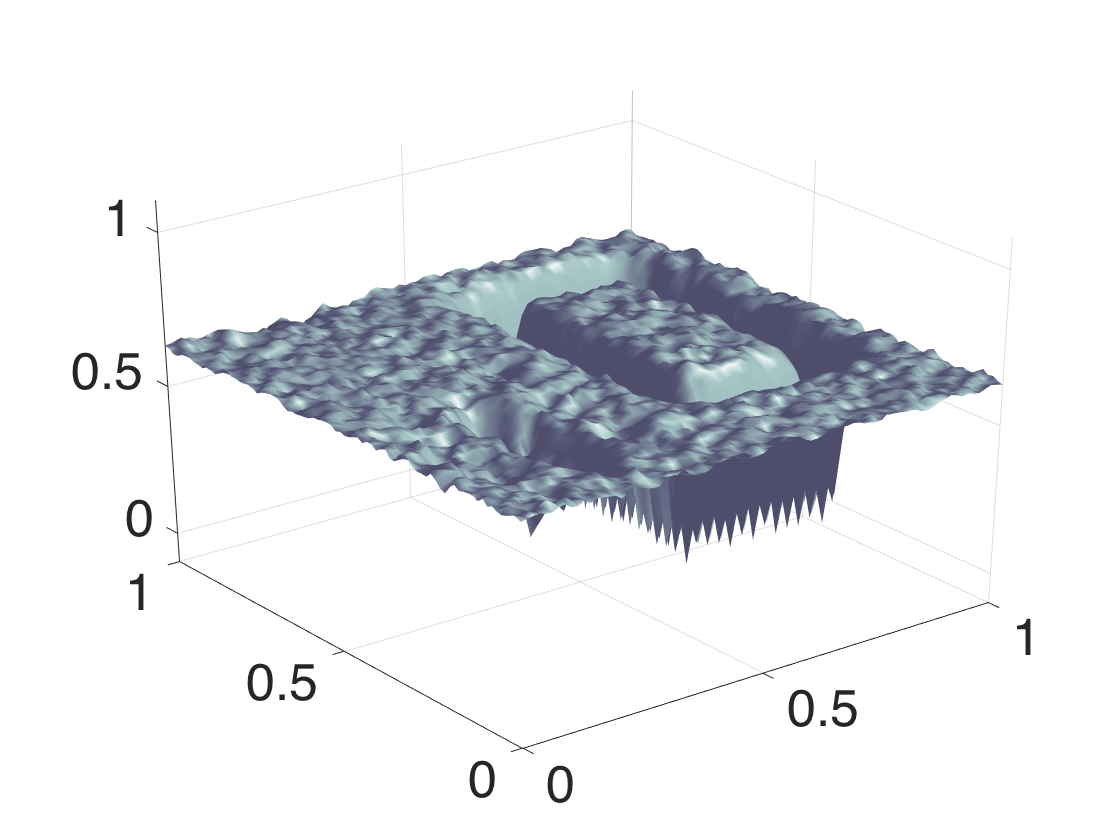}} 
  \subcaptionbox{Estimated $v^N$ on $81\times 81$ mesh}{\includegraphics[width=0.49\textwidth]{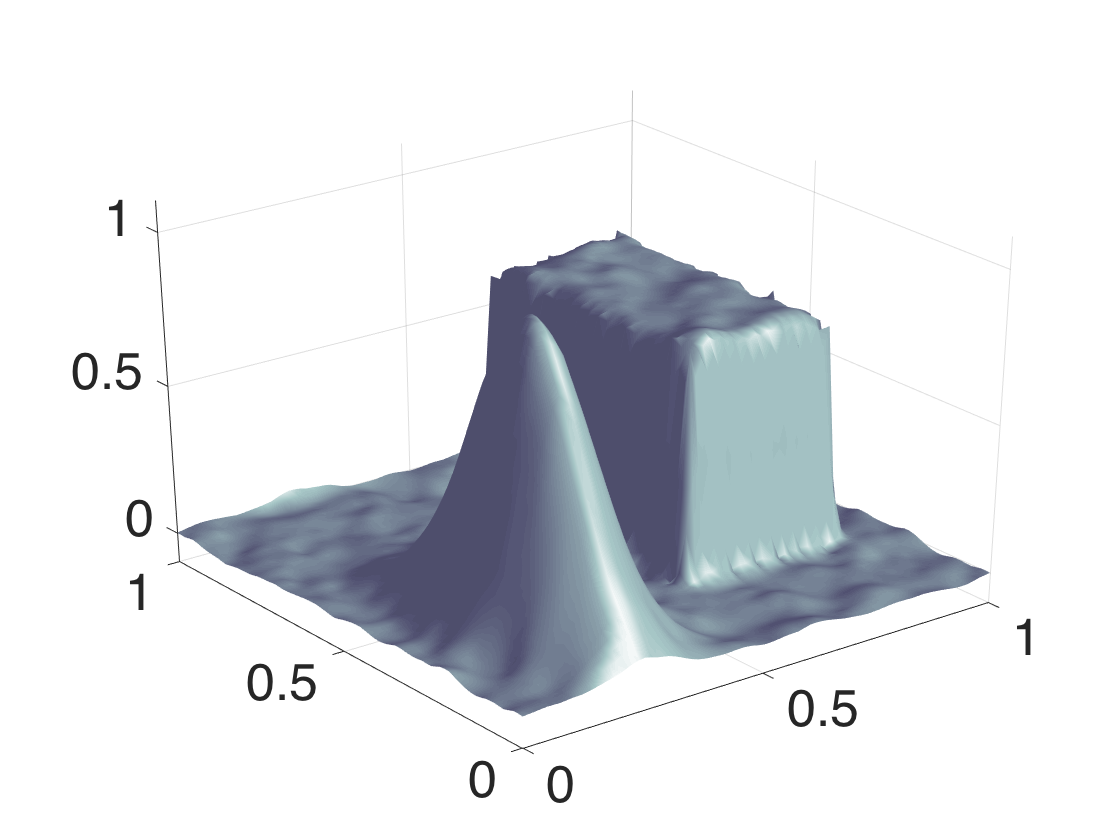}} 
  \caption{Two-dimensional interpolation of block-shaped and Gaussian-shaped structures from noisy observations. A Mat\'ern hyperprior is used in the analysis. } \label{fig:2D_gaussblock}
  \end{center}
\end{figure}

\section{Conclusion and discussion}

We have considered the construction of hypermodels which promote both smoothness and rapid oscillatory features.
The methodology is based on constructing Cauchy and Gaussian hypermodels for Mat\'ern field length-scaling $\ell^N$.
We constructed a combined Gibbs and Metropolis-within-Gibbs algorithm for computing estimates of the unknown and length-scaling, respectively.
In addition, we have shown both analytically and numerically discretisation-invariance of the estimates.
The estimates provide significant advances in comparison to standard constant-parameter Mat\'ern field priors, as we can detect more versatile features.
In this study, we did not include all the  Mat\'ern field parameters in the hyperprior.
In the future studies, e.g.\ in the two-dimensional problems, we should have hypermodel fields for $\ell_1,\ell_2,\theta$ and in addition to the variance scaling mask $\sigma^2$.

We consider this paper to be a concept paper and hence we have considered simple inversion examples.
However,  the methodology can be applied to e.g.\ electrical impedance tomography, Darcy flow models and X-ray tomography. 
In addition, implementing spatiotemporal models with infinite-dimensional Kalman filter techniques would be an interesting path forwards.
In the more theoretical side, we should study the discretisation-invariance issues more rigorously.
Also, the computational machinery needs to be developed further, for example by using  MCMC algorithms, i.e.\ the Metropolis-within-Gibbs can be run with multicore computers. 
Utilisation of GPUs would also be of interest.

\section*{Acknowledgments}
The authors thank Professor Andrew Stuart for useful discussions and proposals.
This work has been funded by Engineering and Physical Sciences Research Council, United Kingdom (EPSRC Reference:	EP/K034154/1 --  Enabling Quantification of Uncertainty for Large-Scale Inverse Problems), and, Academy of Finland (application number 250215, Finnish Program for Centers of Excellence in Research 2012-2017).

%\bibliography{referenssit}
%\bibliographystyle{plain}

\end{document}